\documentclass[a4paper,11pt]{article}
\usepackage{amsmath} \usepackage{amssymb} 
\usepackage{eucal}
\usepackage{geometry}
\usepackage{graphicx}
\usepackage{color}
\usepackage[amsmath,thmmarks]{ntheorem}
\theoremstyle{change}
\theoremseparator{.}

\theorembodyfont{\normalfont\slshape}
\newtheorem{theorem}{Theorem}[section]
\newtheorem{corollary}[theorem]{Corollary}
\newtheorem{lemma}[theorem]{Lemma}
\newtheorem{proposition}[theorem]{Proposition}

\theorembodyfont{\normalfont}
\newtheorem{definition}[theorem]{Definition}
\newtheorem{remark}[theorem]{Remark}
\newtheorem{remarks}[theorem]{Remarks}

\newtheorem{examples}[theorem]{Examples}

\def\proofsymbol{\rule{0.5em}{0.5em}}

\theoremheaderfont{\normalfont\itshape\bfseries}
\theoremstyle{nonumberplain}
\theoremsymbol{\enspace\ensuremath{{\proofsymbol}}}
\newtheorem{proof}{Proof}

\theoremstyle{empty}
\newtheorem{proofof}{}

\usepackage[shortlabels]{enumitem}
\setlist{
  listparindent=\parindent,
  parsep=0pt,
  itemsep=0.5em plus 0.25em minus 0.2em,
}

\setenumerate{
leftmargin=*,
labelsep=0.5em,
}
\setitemize{
leftmargin=*,
labelsep=0.5em,
}

\SetEnumerateShortLabel{s}{\textup{($\sigma$\arabic*)}}
\SetEnumerateShortLabel{a}{\textup{(\alph*)}}
\SetEnumerateShortLabel{A}{\textup{(\Alph*)}}
\SetEnumerateShortLabel{i}{\textup{(\roman*)}}
\SetEnumerateShortLabel{I}{\textup{(\Roman*)}}
\SetEnumerateShortLabel{m}{\textup{(m\arabic*)}}
\SetEnumerateShortLabel{d}{\textup{($\delta$\arabic*)}}
\SetEnumerateShortLabel{p}{\textup{(pm\arabic*)}}
\SetEnumerateShortLabel{1}{\textup{(\arabic*)}}
\SetEnumerateShortLabel{j}{\textup{(i\arabic*)}}
\SetEnumerateShortLabel{n}{\textup{(n\arabic*)}}

\newlist{lenumerate}{enumerate}{3}
\setlist[lenumerate]{
  listparindent=\parindent,
  itemsep=0.5em plus 0.25em minus 0.3em,
  fullwidth,
  labelsep=0.75em,
  label=\arabic*.
}

\numberwithin{equation}{section} 

\def\phi{\varphi}

\def\C{{\mathbb C}}

\def\II{{\mathbb I}}

\def\N{{\mathbb N}}

\def\R{{\mathbb R}}

\def\CA{{\mathcal A}}
\def\CB{{\mathcal B}}
\def\CC{{\mathcal C}}
\def\CD{{\mathcal D}}
\def\CE{{\mathcal E}}
\def\CF{{\mathcal F}}

\def\CH{{\mathcal H}}

\def\CL{{\mathcal L}}
\def\CM{{\mathcal M}}
\def\CN{{\mathcal N}}

\def\CP{{\mathcal P}}

\def\CS{{\mathcal S}}

\def\GA{{\mathfrak A}}

\newcommand\SM{\mathcal{S}\mathcal{M}}
\newcommand\spL{L^{\mathrm{sp}}}
\newcommand\olM{\overline{\mathcal M}}
\newcommand\CBF{\mathcal{B}\mathcal{F}}

\newcommand\wto{\overset{\rm w}{\to}}
\def\qand{\quad\mbox{and}\quad}

\def\Poiss{\mathrm{Poiss}}

\def\Arg{\mathrm{Arg}}

\def\tsum{{\textstyle \sum}}
\def\medcup{\mathop{\textstyle \bigcup}}
\def\medcap{\mathop{\textstyle \bigcap}}
\def\tmedcup{{\textstyle \bigcup}}

\def\olPhi{\overline{\Phi}}
\def\oli{\overline{\iota}}
\def\olCM{\overline{\mathcal M}}

\def\olCN{\overline{\mathcal N}}

\def\d{{\mathrm{d}}}

\def\1{{{\ae}}}
\def\2{{{\o}}}
\def\3{{{\aa}}}
\def\6{\, {\mathrm d}}
\newcommand\ci{\operatorname{i}}
\newcommand\e{\operatorname{e}}
\def\<{{\langle}}
\def\>{{\rangle}}
\def\brinv{{\langle -1\rangle}}

\def\unit{{{\pmb 1}}}
\def\spe{\mathrm{sp}}

\def\id{\mathrm{id}}

\def\ID{\mathcal{ID}}

\def\supp{\mathrm{supp}}
\def\alg{\mathrm{Alg}}
\def\fc{\mathop{\boxplus}}

\begin{document}

\title{Completely Random Measures and L\'evy Bases \\ in Free probability}
\author{
Francesca Collet,  Fabrizio Leisen and
Steen Thorbj{\o}rnsen
}


\maketitle

\begin{abstract}
This paper develops a theory for completely random measures
in the framework of free probability. A general existence result
for free completely random measures is established, and in
analogy to the classical work of Kingman it is proved
that such random measures can be decomposed into the sum of a purely
atomic part and a (freely) infinitely divisible part. The latter part
(termed a free L\'evy basis) is studied in detail in terms of the free
L\'evy-Khintchine representation and a theory parallel to the
classical work of Rajput and Rosinski is developed. Finally a
L\'evy-It\^o type decomposition for general free L\'evy bases is
established.
\end{abstract}

\section{Introduction}

In the paper \cite{Ki67} J.F.C.~Kingman introduced the concept of completely
random measures. Specifically a \emph{random measure} on a measurable space
$(X,\CS)$ is a collection $N=\{N(B,\cdot)\mid B\in\CS\}$ of non-negative random
variables, defined on some probability space $(\Omega,\CF,P)$, such that
the mapping $B\mapsto N(B,\omega)$ is a measure on the
$\sigma$-algebra $\CS$ for
each fixed $\omega$ in $\Omega$. If the random variables
$N(B_1,\cdot),\ldots,N(B_n,\cdot)$ are further assumed to be
independent, whenever 
$B_1,\ldots,B_n$ are \emph{disjoint} sets from $\CS$, then $N$ is referred
to as a \emph{completely random measure}. Kingman established (under certain
additional conditions) that a completely random measure $N$ can always
be decomposed into a sum $N_a+N_c$ of two mutually independent
completely random measures, where, for each $\omega$,
$N_a(\cdot,\omega)$ is purely
atomic, while $N_c(\cdot,\omega)$ is atom less.
For the second term Kingman showed further that the
distribution of the random variable $N_c(B,\cdot)$ is infinitely
divisible for any $B$ in $\CS$, and hence $N_c$ is an example of
what is nowadays commonly referred to as a \emph{L\'evy basis}. The infinite
divisibility of the ``marginals'' allows for the
employment of L\'evy-Khintchine techniques, and the resulting
theory was developed by B.S.~Rajput and J.~Rosinski in the celebrated
paper \cite{RR}, where more general ``index sets'' than
$\sigma$-algebras were also considered. To be precise, a L\'evy
basis\footnote{In \cite{RR} a L\'evy basis was referred to as an
  infinitely divisible, independently scattered random measure.}
on a ring $\CE$ of subsets of $X$ is a
family $N=\{N(E,\cdot)\mid E\in\CE\}$ of real valued random variables,
defined on some probability space $(\Omega,\CF,P)$, such that

\begin{itemize}

\item For all $E$ in $\CE$ the distribution of $N(E,\cdot)$ is an
  infinitely divisible probability measure on $\R$.

\item If $n\in\N$ and $E_1,E_2,\ldots E_n$ are disjoint sets from
  $\CE$, then $N(E_1,\cdot),N(E_2,\cdot),\ldots,N(E_n,\cdot)$ are
  independent random variables. 

\item If $(E_n)_{n\in\N}$ is a sequence of disjoint sets from $\CE$,
  such that $\medcup_{n\in\N}E_n\in\CE$, then it holds with
  probability 1 that
  $N(\medcup_{n\in\N}E_n,\cdot)=\sum_{n=1}^{\infty}N(E_n,\cdot)$.

\end{itemize}

In recent years much of the theory of stochastic processes (with very
general index sets) has found a subsuming and unifying framework in
L\'evy bases, and from that perspective it is a natural step
in the development of free probability to manifest
a corresponding theory for \emph{free} L\'evy bases and more generally 
\emph{free} completely random measures. This theory
can also be expected to provide a concrete model for the asymptotics
of high dimensional random-matrix valued random measures (and integrals with
respect to such), which have received some attention recently in
various special cases (see e.g.\ \cite{PARA}, \cite{DMRA} and
\cite{PaPePA}), but the theory remains to be fully developed.

In this paper we introduce natural counterparts of completely random
measures and L\'evy bases in the context of free probability,
where the classical notion of independence is replaced by that of free
independence, and infinite divisibility refers to the corresponding
notion of free convolution (see \cite{vdn} or \cite{BNT06} for an
introduction to free probability).
We establish thus general existence results
for free completely random measures and for free L\'evy bases. In addition
we prove, in full analogy with the
mentioned results of Kingman, that a non-negative free completely
random measure $M$ can be decomposed into a sum 
\begin{equation}
M=M_a+M_c
\label{intro_eq1}
\end{equation}
of two freely independent
terms, where $M_a$ is purely atomic (in a natural sense) free
completely random measure, while $M_c$
is a free L\'evy basis (thus with freely infinitely
divisible marginals). We derive a similar decomposition in the more general
situation where the assumption of positivity of the marginals of $M$ is
dropped, although some moment conditions need to be imposed in this
case. We focus subsequently on free L\'evy bases, where the free infinite
divisibility of the marginals allows for invoking the Bercovici-Pata
bijection (see~Subsection~\ref{subsec:Free_ID}) and thus for
transferring major parts of the theory of Rajput and Rosinski to the
free setting. The resulting theory subsumes and unifies a major part
of the existing theory on free L\'evy processes and related topics.
Moreover, it includes a
theory of integration of deterministic functions with respect to a
free L\'evy basis, which is further used to
establish a L\'evy-It\^o type decomposition of a general free L\'evy
basis into the sum of two freely independent terms, the first of which
is of free Brownian motion type, while the second is of pure jump type.
This result 
covers in particular the free analog of the result by Pedersen (see
\cite{jp}) for classical L\'evy bases, and in another direction it
generalizes the L\'evy-It\^o type decomposition obtained for free
L\'evy processes in \cite{BNT05}. Inserting the L\'evy-It\^o decomposition
of $M_c$ in \eqref{intro_eq1} evidently leads to a refined
decomposition for general free completely random measures.

The remaining part of this paper is organized as follows.
In Section~\ref{sec:prelims} we provide
background material on $\delta$-rings (and measures thereon), the
measure topology and free infinite divisibility.
In Section~\ref{sec:FCRM} we give the formal
definition of free completely random measures and of
free L\'evy bases, and we state the mentioned general
existence result (the proof of which is deferred to
Section~\ref{sec:Proof_Existence_FLB}). We establish furthermore the
described analogs of Kingman's decomposition theorem for completely
random measures.
In Section~\ref{sec:Rajput_Rosinski_theory_for_FLB}
we develop free analogs of essential parts of the Rajput-Rosinski
theory, and in Section~\ref{sec:Integration_wrt_FLB}
we develop a theory of integration with respect to free L\'evy
bases. In particular we construct free L\'evy bases with a deterministic
``density'' with respect to another (given) free L\'evy basis, and
this is further used
in Section~\ref{sec:Levy-Ito_for_FLB}, where we
establish the described
L\'evy-It\^o type decomposition for free L\'evy bases.
In the final Section~\ref{sec:Proof_Existence_FLB} we prove the
general existence of free completely random measures and of free L\'evy
bases. While the existence of ``classical''  random measures and
L\'evy bases is generally based 
on the Kolmogorov extension theorem, our construction is based
on free products of von Neumann algebras and the theory of (unbounded)
operators affiliated with such. As 
this construction heavily builds on the theory of operator algebras
and is less probabilistic in nature, we have deferred it to the final
section of the paper. This is mainly
to bring focus to the probabilistic aspects of the developed theory and to
emphasize the analogies to the theories of Kingman and of Rajput and
Rosinski. Accordingly the first six sections of the paper can be read
without reference to the detailed construction given in
Section~\ref{sec:Proof_Existence_FLB}. To our knowledge this construction,
dealing throughout with unbounded operators, has not been carried out
in detail previously in the literature even for the case of free
L\'evy processes.
The paper concludes with an appendix that covers some specific aspects of the
theory of von Neumann algebras needed for the construction in
Section~\ref{sec:Proof_Existence_FLB}.

\section{Preliminaries}
\label{sec:prelims}

In this section we provide background material on various definitions
and results that are fundamental for the rest of the paper.

\subsection{Measures on $\delta$-rings}\label{subsec:delta_rings} 

Recall that a ring of subsets of a non-empty set $X$ is a collection
$\CE$ of subsets of $X$ satisfying that $A\cup B$, $A\setminus
B\in\CE$, whenever $A,B\in\CE$. Since $A\cap
B=A\setminus(A\setminus B)$, a ring is automatically closed under
finite intersections. If $\CE$ is even closed under all countable
intersections, then it is referred to as a $\delta$-ring. By
$\sigma(\CE)$ we denote the smallest $\sigma$-algebra on $X$
containing $\CE$. It is noteworthy that if $\CE$ is a $\delta$-ring,
then the following implication holds for all subsets $A,E$ of $X$:
\begin{equation}
A\in\sigma(\CE) \quad\text{and}\quad E\in\CE \quad\implies\quad  A\cap
E\in\CE.
\label{eq_Hereditary_prop_delta_ring}
\end{equation}
A (finite) \emph{signed measure} on a $\delta$-ring $\CE$ is a mapping
$\Theta\colon\CE\to\R$ satisfying the following two conditions:

\begin{enumerate}[a]

\item\label{def:maal(b)}
$\Theta(A\cup B)=\Theta(A)+\Theta(B)$ for any disjoint sets $A,B$ from $\CE$,

\item\label{def:maal(c)}
$\lim_{n\to\infty}\Theta(B_n)=0$ for any decreasing sequence $(B_n)_{n\in\N}$ of sets
from $\CE$, such that $\medcap_{n\in\N}B_n=\emptyset$.

\end{enumerate}

For sequences $(B_n)_{n\in\N}$ of sets as described in \ref{def:maal(c)}, we use the 
notation $B_n\downarrow\emptyset$.
Conditions \ref{def:maal(b)} and \ref{def:maal(c)}
(together) are equivalent to the condition that
$\Theta(\medcup_{n\in\N}E_n)=\sum_{n=1}^{\infty}\Theta(E_n)$ for any
sequence $(E_n)_{n\in\N}$ of disjoint sets from $\CE$, such that
$\medcup_{n\in\N}E_n\in\CE$.
For a mapping $\Theta\colon\CE\to[0,\infty]$ this latter condition, together
with the condition $\Theta(\emptyset)=0$, defines a (positive) measure
on $\CE$. In case $\Theta(A)\in[0,\infty)$ for all $A$ in $\CE$, we refer to
$\Theta$ as a \emph{finite} measure on $\CE$.
By a suitable variant of the Carath\'eodory Extension Theorem, a (positive)
measure on $\CE$ can always be extended to a (positive)
measure on $\sigma(\CE)$. Note however that the extension of a finite
measure may fail to be finite. Correspondingly it does not generally
hold that a signed measure on $\CE$ can be extended to a signed
measure on $\sigma(\CE)$. In fact, under the additional assumption:
\begin{equation}
\text{There exists a sequence $(U_n)_{n\in\N}$ of sets from $\CE$, such that
  $\medcup_{n\in\N}U_n=X$,}
\label{eq_determining_sequence}
\end{equation}
any signed measure $\Theta$ on a $\delta$-ring $\CE$ can be written
uniquely in the form:
\begin{equation}
\Theta(E)=\Theta^+(E)-\Theta^-(E), \qquad (E\in\CE),
\label{eq_Hahn-Jordan}
\end{equation}
where $\Theta^+,\Theta^-$ are two (positive)
measures on $\sigma(\CE)$, 
which are singular in the sense that there exists a set $S$ from
$\sigma(\CE)$, such that $\Theta^+(S^c)=\Theta^-(S)=0$. These
$\Theta^+,\Theta^-$ are not necessarily finite measures on $\sigma(\CE)$, but
$\Theta^+(E),\Theta^-(E)<\infty$ for all $E$ in $\CE$, so in particular
$\Theta^+,\Theta^-$ are $\sigma$-finite (cf.\ \eqref{eq_determining_sequence}).
The unique decomposition \eqref{eq_Hahn-Jordan}
further allows us to define the total variation measure $|\Theta|$ of $\Theta$ as
\[
|\Theta|(A)=\Theta^+(A)+\Theta^-(A), \qquad(A\in\sigma(\CE)).
\]
In this setup we mention finally, that if $\kappa$ is a $\sigma$-finite
(positive) measure on $\sigma(\CE)$, such that $|\Theta|\le\kappa$, then $\Theta^+$
and $\Theta^-$ are both absolutely continuous with respect to $\kappa$
with $\sigma(\CE)$-measurable densities $h^+,h^-\colon X\to\R$, which
may be chosen such that $h^+(x),h^-(x)\in[0,1]$ for all $x$ in
$X$. Then if we put $h=h^+-h^-$, it follows for any set $E$ from $\CE$
that
\[
\Theta(E)=\int_Eh^+\6\kappa-\int_Eh^-\6\kappa=\int_Eh\6\kappa,
\]
so that $\Theta$ has density $h$ with respect to $\kappa$. Note
however that $h$ need not be an element of $\CL^1(\kappa)$.


\subsection{Free Independence}

Free independence (introduced by Voiculescu) is a notion of
independence which in many respects behaves similarly to the classical
notion of independence of random variables, and it is possible to
develop a probability theory based on this notion in parallel to
the classical theory of probability.
At the same time free independence
cannot be observed among classical random variables (except for
trivial cases). The right framework for free independence is that of
quantum probability, where the random variables are
modelled mathematically as
selfadjoint operators affiliated with a $W^*$-probability space. A
$W^*$-probability space is a pair $(\CM,\tau)$ consisting of a von
Neumann algebra $\CM$ acting on a Hilbert space $\CH$
equipped with a normal faithful tracial state
$\tau\colon\CM\to\C$ (see \cite{vdn} for details). A (possibly
unbounded) operator $a$ in $\CH$ is affiliated with $\CM$, if $au=ua$
for any unitary operator $u$ on $\CH$ satisfying that $ub=bu$ for all
$b$ in $\CM$. If $a$ is selfadjoint, this is equivalent to the
condition that $f(a)\in\CM$ for any bounded Borel function $f\colon\R\to\R$, where
$f(a)$ is defined in terms of spectral calculus. In this case the \emph{spectral
    distribution} of $a$ is the unique Borel-probability measure
  $\spL\{a\}$ on $\R$, satisfying that
\[
\int_{\R}f(t)\,\spL\{a\}(\d t)=\tau(f(a))
\]
for any bounded Borel-function $f\colon\R\to\R$. Throughout the paper
we denote by $\CBF(\R)$ the algebra of all real-valued \emph{Borel}
functions on $\R$. Furthermore we let $\CBF_b(\R)$ denote the
subalgebra of bounded functions from $\CBF(\R)$.

If $a_1,\ldots,a_n$ are (possibly unbounded) selfadjoint
operators affiliated with $\CM$, they are said to be freely
independent (with respect to $\tau$), if
\begin{equation}
\tau\big(\big[f_1(a_{i_1}) -\tau(f_1(a_{i_1}))\big]
\big[f_2(a_{i_2}) -\tau(f_2(a_{i_2}))\big]
\cdots\big[f_m(a_{i_m}) -\tau(f_m(a_{i_m}))\big]\big)=0,
\label{prel_eq1}
\end{equation}
for any $m\in\N$, any functions
$f_1,\ldots,f_m$ from $\CBF_b(\R)$
and any $i_1,\ldots,i_m$ from
$\{1,\ldots,n\}$, such that $i_1\ne i_2, i_2\ne i_3, \ldots,
i_{m-1}\ne i_m$.

If $\CA$ is a unital subalgebra of $\CM$, we denote by $\CA^\circ$ the
subspace of centered elements of $\CA$, i.e.\
$\CA^\circ=\{a\in\CA\mid\tau(a)=0\}$.
A finite number of unital subalgebras $\CA_1,\ldots,\CA_n$ of $\CM$
are said to be freely independent, provided that
\[
\tau(a_{1}a_{2}\cdots a_{m})=0,
\]
whenever $a_1\in\CA_{i_1}^\circ,\ldots, a_{m}\in\CA_{i_m}^\circ$ for
suitable $i_1,\ldots,i_m\in\{1,\ldots,n\}$ such that $i_1\ne i_2$,
$i_2\ne i_3$, $\ldots$, $i_{m-1}\ne i_m$.
For any collection $\{T_i\mid i\in I\}$ of selfadjoint operators
affiliated with $\CM$ we denote by $\alg(\{T_i\mid i\in\ I\})$
(respectively $W^*(\{T_i\mid i\in I\})$)
the unital subalgebra (respectively $W^*$-subalgebra)
of $\CM$ generated by the subset $\{f(T_i)\mid i\in
I, f\in\CBF_b(\R)\}$ of $\CM$. We refer to $\alg(\{T_i\mid i\in\ I\})$ and
$W^*(\{T_i\mid i\in\ I\})$ as, respectively, the unital subalgebra
and the $W^*$-subalgebra of $\CM$ generated by $\{T_i\mid i\in I\}$.
It follows then that selfadjoint operators
$a_1,\ldots,a_n$ affiliated with $\CM$ are freely independent, exactly
when they generate freely independent unital subalgebras (or, equivalently,
$W^*$-subalgebras) of $\CM$.

\subsection{The measure topology}

For a $W^*$-probability space $(\CM,\tau)$ we denote by
$\overline{\CM}$ the space of closed, densely defined
(possibly unbounded) operators affiliated with $\CM$. Then
$\overline{\CM}$ is a $*$-algebra under the adjoint operation and the so
called strong sum and strong product. For example the strong sum of
two operators $a,a'$ from $\olM$ is the closure of the operator
$a+a'$, and the strong product is defined similarly.

For any positive numbers $\epsilon,\delta$ we introduce the
subset $N(\epsilon,\delta)$ of $\olM$ given by
\[
N(\epsilon,\delta)=\{a\in\olM\mid
\tau[1_{(\epsilon,\infty)}(|a|)]<\delta\},
\]
where $|a|=(a^*a)^{1/2}$.
\emph{The measure topology} on $\olM$ is the vector space topology on $\olM$
for which the sets $N(\epsilon,\delta)$,
$\epsilon,\delta\in(0,\infty)$, form a neighborhood basis at $0$. In
this topology the adjoint operation and the (strong-) sum and product
are all continuous operations. In addition the measure topology
satisfies the first axiom of countability and is a complete Hausdorff
topology.

For a sequence $a,a_1,a_2,a_3,\ldots$ of operators from $\olM$ we have that
\[
a_n\to a \quad\text{in the measure topology}\quad \iff \quad
\spL\{|a_n-a|\}\overset{\rm w}{\to}\delta_0 \quad\text{as
  $n\to\infty$,}
\]
and thus convergence in the measure topology is the quantum
probability analog of convergence in probability. For that reason, and
for brevity, we will occasionally use the notation
$a_n\overset{\rm P}{\to} a$ or $a=\text{P-}\lim_{n\to\infty}a_n$ to
express that a sequence $(a_n)_{n\in\N}$ converges to $a$ in the measure topology.

If $a,a_1,a_2,a_3,\ldots$ are all selfadjoint, we note further the implications:
\begin{equation}
\begin{split}
a_n\to a \quad\text{in the measure topology}\quad &\iff \quad
\spL\{a_n-a\}\overset{\rm w}{\to}\delta_0 \quad\text{as
  $n\to\infty$}
\\
&\, \implies \quad \spL\{a_n\}\overset{\rm w}{\to}\spL\{a\}.
\end{split}
\label{prel_eq2}
\end{equation}
For more details about the measure topology (and some proofs) we refer
to the appendix of \cite{BNT06}. A more complete account on this topic
can be found in \cite{te}.

\subsection{Free infinite divisibility and the free cumulant
  transform}\label{subsec:Free_ID}

If $a$ and $b$ are two freely independent selfadjoint operators
affiliated with $\CM$, then the free convolution
$\spL\{a\}\boxplus\spL\{b\}$ of their spectral distributions is
defined as the spectral distribution of the sum $a+b$. Since 
$\spL\{a+b\}$ is uniquely determined by $\spL\{a\},\spL\{b\}$ and the
condition \eqref{prel_eq1}, and since any pair $(\mu,\nu)$ of
Borel-probability measures on $\R$ may be realized as the spectral
distributions of two freely independent selfadjoint operators affiliated with
some $W^*$-probability space, the operation $\boxplus$ is a
well-defined binary operation on the class $\CP(\R)$ of all (Borel-) probability
measures on $\R$ (see \cite{bv} for details). The corresponding
class of infinitely divisible probability laws is denoted by $\ID(\boxplus)$. Thus a 
measure $\nu$ from $\CP(\R)$ belongs to $\ID(\boxplus)$, if and only
if
\[
\forall n\in\N \ \exists \nu_{1/n}\in\CP(\R)\colon
\nu=\nu_{1/n}\boxplus\nu_{1/n}\boxplus
\cdots\boxplus\nu_{1/n} \quad(\text{$n$ terms}).
\]
The class of infinitely divisible probability laws with respect to
classical convolution $*$ of probability measures is correspondingly
denoted by $\ID(*)$.

As in classical probability, free infinite divisibility is generally
studied through a L\'evy-Khintchine type representation of the free
analog of the Fourier transform; the so called free
cumulant transform. Specifically the free cumulant transform of a
measure $\nu$ from $\ID(\boxplus)$ is defined by the formula:
\[
\CC_\nu(z)=zG_\nu^{\brinv}(z)-1, 
\]
where $G_\nu^{\brinv}$ denotes the inverse of the Cauchy transform $G_\nu$
given by
\[
G_{\nu}(z)=\int_{\R}\frac{1}{z-x}\,\nu(\d x), \qquad(z\in\C^+).
\]
This inverse (and hence $\CC_\nu$) is always well-defined in a region
(depending on $\nu$) of
the lower half complex plane $\C^-$ in the form:
\[
\CD_\nu=\big\{z\in\C^-\bigm| |z|<\delta \ \text{and} \
\Arg(z)\in(-\tfrac{\pi}{2}-\epsilon,-\tfrac{\pi}{2}+\epsilon)\big\}.
\]
For a selfadjoint operator $a$ affiliated with a
  $W^*$-probability space $(\CM,\tau)$ 
we shall also use the notation $\CC_a$ for the free cumulant transform
$\CC_{\spL\{a\}}$ of the spectral distribution of $a$.
The key property of the free cumulant transform is that it linearizes
free additive convolution in the sense that
\begin{equation}
\CC_{\nu\boxplus\nu'}(z)=\CC_{\nu}(z)+\CC_{\nu'}(z)
\label{Linearizing_Property_I}
\end{equation}
for any probability measures $\nu,\nu'$ on $\R$.

In case $\nu$ has finite $p$-th moment for some $p$ in $\N$, $\CC_\nu$
admits a Taylor expansion centered at 0 in the form:
\begin{equation}
\CC_{\nu}(z)=\sum_{j=1}^p\kappa_j(\nu)z^{j}+o(z^p),
\label{Taylor_Expansion}
\end{equation}
where the coefficients $\kappa_1(\nu),\ldots,\kappa_p(\nu)$ are the
\emph{free cumulants} of $\nu$ (see \cite[Theorem~1.3]{BG}). These
were introduced by Speicher via the \emph{moment-cumulant
formula}:
\begin{equation}
\kappa_p(\nu)=\int_{\R}t^p\,\nu(\d t) - 
\sum_{\pi\in\textrm{NC}'(p)}\prod_{V\in\rm{BL(\pi)}}\kappa_{\# V}(\nu),
\label{Moment_Cumulant_Formula}
\end{equation}
from which the free cumulants are defined recursively.
In \eqref{Moment_Cumulant_Formula} $\textrm{NC}'(p)$ is the set of
non-crossings partitions of 
$\{1,\ldots,p\}$ with at least two blocks. For such a partition
$\pi$, $\textrm{BL}(\pi)$ denotes the family of blocks of $\pi$, while,
for $V$ in $\textrm{BL}(\pi)$, $\#V$ denotes the cardinality of the corresponding
subset of $\{1,\ldots,p\}$ (see \cite[Chapter~2]{MiSp} for
details). In accordance with \eqref{Linearizing_Property_I} and
\eqref{Taylor_Expansion} the free cumulants linearize free additive
convolution in the sense that
\begin{equation}
\kappa_j(\nu\boxplus\nu')=\kappa_j(\nu)+\kappa_j(\nu'), \qquad(j=1,\ldots,p),
\label{Linearizing_Property_II}
\end{equation}
whenever $\nu,\nu'$ both have finite $p$-th moment.

A measure $\nu$ from $\CP(\R)$ is in $\ID(\boxplus)$, if and only if
$\CC_\nu$ has the \emph{free L\'evy-Khintchine representation}:
\[
\CC_\nu(z)=az+bz^2+\int_{\R}\Big(\frac{1}{1-tz}-1-z\varsigma(t)\Big)\,r(\d
t),
\qquad(z\in\C^-),
\]
where $a\in\R$, $b\in[0,\infty)$, $r$ is a L\'evy measure on $\R$ and
$\varsigma$ is the function given by\footnote{
To emphasize the analogy to the theory developed by
Rajput and Rosinski, we have chosen to work throughout with
the same centering function $\varsigma$ as the one used in \cite{RR},
one of the advantages of which is continuity.}
\[
\varsigma(t)=-1_{(-\infty,1)}(t)+t1_{[-1,1]}(t)+1_{(1,\infty)}(t), \qquad(t\in\R).
\]
The triplet $(a,b,r)$ is uniquely determined and is referred to as the
free characteristic triplet of $\nu$. 
Recall in comparison that a measure $\mu$ from $\CP(\R)$ belongs to $\ID(*)$ if and
only if its Fourier transform $\hat{\mu}$ has the
L\'evy-Khintchine representation:
\[
\hat{\mu}(y)=\exp\Big(\ci ay-\tfrac{1}{2}bt^2+\int_{\R}\big(\e^{\ci ty}-1-\ci
y\varsigma(t)\big) r(\d t)\Big), \qquad(y\in\R),
\]
where the parameters $(a,b,r)$ are exactly as above, uniquely determined
by $\mu$ and referred to as the (classical) characteristic triplet of $\mu$. 

From the two L\'evy-Khintchine representations above, it is apparent
that there is a one-to-one correspondence $\Lambda$ from $\ID(*)$ onto
$\ID(\boxplus)$. Specifically $\Lambda$ maps the probability measure
in $\ID(*)$ with classical characteristic triplet $(a,b,r)$ onto the
probability measure in $\ID(\boxplus)$ with free characteristic
triplet $(a,b,r)$. Although $\Lambda$ may appear as a rather formal
correspondence, it has the following fundamental properties for all
$\mu_1,\mu_2$ in $\ID(*)$ and all $c$ in $\R$:

\begin{enumerate}[i]

\item\label{Lambda_egnsk(i)}
$\Lambda(\mu_1*\mu_2)=\Lambda(\mu_1)\boxplus\Lambda(\mu_2)$,

\item\label{Lambda_egnsk(ii)}
$\Lambda(D_c\mu)=D_c\Lambda(\mu)$,

\item\label{Lambda_egnsk(iii)} 
$\Lambda(\delta_c)=\delta_c$

\item\label{Lambda_egnsk(iv)} 
$\Lambda$ is a homeomorphism with respect to weak convergence.

\end{enumerate}
In \ref{Lambda_egnsk(ii)} and in the following we use the notation
$D_c\mu$ for the scaling of $\mu$ by the constant $c$, i.e.\
$D_c\mu(B)=\mu(c^{-1}B)$, if $c\ne0$, while $D_0\mu=\delta_0$.
Here, as in \ref{Lambda_egnsk(iii)}, $\delta_c$ denotes the Dirac
measure at $c$.

The probability laws appearing in the 
free analogs of the Central Limit  Theorem and the Poisson Limit
Theorem are the semi-circle distributions:
\[
\gamma_{c,\ell}(\d t)=\frac{2}{\pi
  \ell^2}\sqrt{\ell^2-(t-c)^2}1_{[c-\ell,c+\ell]}(t)\6t, \qquad(\ell>0, \ c\in\R),
\]
and, respectively, the Marchenko-Pastur distributions
\begin{equation*}
\text{mp}_\ell(\d t)=
\begin{cases}
(1-\ell)\delta_0+\frac{1}{2\pi t}\sqrt{(t-s)(u-t)}1_{[s,u]}(t)\6t,
&\text{if $\ell\in(0,1)$,}
\\
\frac{1}{2\pi t}\sqrt{(t-s)(u-t)}1_{[s,u]}(t)\6t, &\text{if $\ell\in[1,\infty)$,}
\end{cases}
\end{equation*}
where $s=(1-\ell)^2$ and $u=(1+\ell)^2$. 
Correspondingly the mapping $\Lambda$ maps Gaussian distributions to
semi-circular distributions and Poisson distributions onto
Marchenko-Pastur distributions. The latter are also referred to as
free Poisson distributions.

It will prove important for us to express weak convergence of probability
measures in $\ID(\boxplus)$ in terms of the free characteristic
triplets in analogy with e.g.\ the classical result
\cite[Theorem~8.7]{Sato99}.

\begin{theorem}\label{thm:konv_vha_kar_triplet} 
Let $\nu,\nu_1,\nu_2,\nu_3,\ldots$ be probability
  measures from $\ID(\boxplus)$ with free characteristic triplets,
  respectively $(a,b,r), (a_1,b_1,r_1),(a_2,b_2,r_2),\ldots$. Then
  $\nu_n\to\nu$ weakly as $n\to\infty$, if and only if the following
  conditions are satisfied:

\begin{enumerate}[i]

\item $a_n\to a$ as $n\to\infty$.

\item $\int_{\R}f(t)\,r_n(\d t)\to\int_{\R}f(t)\,r(\d t)$ as
  $n\to\infty$ for any continuous bounded function $f\colon\R\to\R$
  vanishing in a neighborhood of 0.

\item $\displaystyle{
\lim_{\epsilon\downarrow0}\Big(\limsup_{n\to\infty}\Big|b_n-b
+\int_{[-\epsilon,\epsilon]}t^2\,r_n(\d t)\Big|\Big)=0.
}$
\end{enumerate}
\end{theorem}

Theorem~\ref{thm:konv_vha_kar_triplet} follows immediately by
combining the corresponding classical result
(\cite[Theorem~8.7]{Sato99}) with the fact that $\Lambda$ is a
homeomorphism.

Another useful characterization of weak convergence for measures in
$\ID(\boxplus)$ is the following:

 \begin{proposition}\label{prop:svag_konv_vha_fri_kum_tranf}
Let $\nu,\nu_1,\nu_2,\nu_3,\ldots$ be probability
  measures from $\ID(\boxplus)$. Then the following two conditions are
  equivalent:

\begin{enumerate}[i]

\item\label{prop:svag_konv_vha_fri_kum_tranf(i)}
$\nu_n\overset{\rm w}{\to}\nu$ as $n\to\infty$.

\item\label{prop:svag_konv_vha_fri_kum_tranf(ii)}
$
\lim_{n\to\infty}\CC_{\nu_n}(\ci y)=\CC_\nu(\ci y)
$
for all $y$ in $(-\infty,0)$.

\end{enumerate}
\end{proposition}

For general probability measures (outside $\ID(\boxplus)$) the
condition in Proposition~\ref{prop:svag_konv_vha_fri_kum_tranf} needs
to be supplemented by the condition:
\[
\sup_{n\in\N}\big|\CC_{\nu_n}(\ci y)\big|\longrightarrow0,
  \quad\text{as $y\uparrow0$,}
\]
in order to ensure weak convergence. 
Proposition~\ref{prop:svag_konv_vha_fri_kum_tranf} may be
established as a consequence of Theorem~\ref{thm:konv_vha_kar_triplet}
and the corresponding classical result. We provide here, for the
reader's convenience,  a short proof which bypasses arithmetics
with characteristic triplets. 

\begin{proofof}[Proof of
  Proposition~\ref{prop:svag_konv_vha_fri_kum_tranf}.]

As indicated, it is well-known that
\ref{prop:svag_konv_vha_fri_kum_tranf(i)} implies
\ref{prop:svag_konv_vha_fri_kum_tranf(ii)}.
For the converse implication we recall from
\cite{BNT04} that the formula:
\[
\widehat{\Upsilon(\mu)}(y)=\exp(\CC_{\Lambda(\mu)}(\ci y))
\qquad(\mu\in\ID(*), \ y\in(-\infty,0)),
\]
defines an injective mapping $\Upsilon\colon\ID(*)\to\ID(*)$, which is
a homeomorphism onto its range with respect to weak convergence (see
\cite{BNRT08}). Putting $\mu=\Lambda^{-1}(\nu)$ and
$\mu_n=\Lambda^{-1}(\nu_n)$, the assumption of the proposition implies
that
$
\lim_{n\to\infty}\widehat{\Upsilon(\mu_n)}(y)=\widehat{\Upsilon(\mu)}(y)
$
for all $y$ in $(-\infty,0)$, and by complex conjugation this also
holds for all positive $y$. Hence by the continuity theorem for Fourier transforms
we have that $\Upsilon(\mu_n)\overset{\rm w}{\to}\Upsilon(\mu)$ as
$n\to\infty$, and since $\Upsilon$ is a homeomorphism, this means that
$\mu_n\overset{\rm w}{\to}\mu$ as $n\to\infty$. Since $\Lambda$ is
continuous with respect to weak convergence, this further implies that
$\nu_n\overset{\rm w}{\to}\nu$ as $n\to\infty$, as desired.
\end{proofof}

\section{Free Completely Random Measures}
\label{sec:FCRM}

\begin{definition}\label{def_FLB}
Let $X$ be a non-empty set, and let $\CE$ be a ring of subsets of
$X$. Then a free L\'evy basis (FLB) on $(X,\CE)$ is a family
$M=\{M(E)\mid E\in\CE\}$ of selfadjoint operators, affiliated with some
$W^*$-probability space $(\CM,\tau)$, satisfying the following four
conditions:

\begin{enumerate}[a]

\item\label{def_FLB(b)} 
If $E_1,\ldots,E_n$ are disjoint sets from $\CE$, then
  $M(E_1),\ldots,M(E_n)$ are freely independent with respect to
  $\tau$.

\item\label{def_FLB(c)} 
If $E_1,\ldots,E_n$ are disjoint sets from $\CE$, then
$M(\medcup_{j=1}^nE_j)=M(E_1)+\cdots+M(E_n)$.

\item\label{def_FLB(d)} 
$\spL\{M(E_n)\}\overset{\rm w}{\to}\delta_0$ 
as $n\to\infty$ for any decreasing sequence of
  sets $E_n$ from $\CE$, satisfying that
  $\medcap_{n\in\N}E_n=\emptyset$.

\item\label{def_FLB(a)}
$\spL\{M(E)\}\in\ID(\boxplus)$ for all $E$ in $\CE$.

\end{enumerate}

If $M$ satisfies only conditions \ref{def_FLB(b)}-\ref{def_FLB(c)}
above, it is referred to as a finitely additive free random
measure (FAFRM). If $M$ satisfies conditions \ref{def_FLB(b)}-\ref{def_FLB(d)}
it is termed a free completely random measure (FCRM).
\end{definition}

\begin{remarks}\label{rem:Equiv_def_of_FLB}
\begin{lenumerate}[1]

\item\label{rem:Equiv_def_of_FLB(A)}
From Definition~\ref{def_FLB} it is immediate to check that if
  $t\in\R$ and $M^1$ and $M^2$ are two freely independent free L\'evy
  bases (respectively FAFRMs or FCRMs) then $tM^1+M^2$ is again a free
  L\'evy basis (respectively FAFRM or FCRM).

\item\label{rem:Equiv_def_of_FLB(B)}
In full analogy with the classical theory of random measures we note
that
if $M=\{M(E)\mid E\in\CE\}$ is a FAFRM then the remaining condition
\ref{def_FLB(d)} in Definition~\ref{def_FLB} in order for $M$ to be a FCRM
is equivalent to the requirement that
\begin{equation}
\sum_{k=1}^nM(E_n)\xrightarrow[n\to\infty]{} M\big(\medcup_{n\in\N}E_n\big)
\quad\text{in the measure topology}
\label{Equiv_def_of_FLB_eq1}
\end{equation}
for any sequence $(E_n)_{n\in\N}$ of disjoint sets from $\CE$,
satisfying that $E:=\medcup_{n\in\N}E_n\in\CE$. Indeed, for such a
sequence $(E_n)_{n\in\N}$ condition \ref{def_FLB(d)} entails that
\begin{equation*}
\spL\big\{M\big(E\setminus(\tmedcup_{j=1}^nE_j)\big)\big\}
\overset{\rm w}{\longrightarrow}\delta_0,
\quad\text{as $n\to\infty$}, 
\end{equation*}
and hence by property \ref{def_FLB(c)} we obtain for any positive
$\epsilon$ that
\begin{equation*}
\begin{split}
\tau\Big[1_{\R\setminus[-\epsilon,\epsilon]}\big(M(E)-\tsum_{j=1}^nM(E_j)\big)\Big]
&=\tau\Big[1_{\R\setminus[-\epsilon,\epsilon]}
\big(M(E\setminus(\tmedcup_{j=1}^nE_j))\big)\Big]
\\[.2cm]
&=\int_{\R}1_{\R\setminus[-\epsilon,\epsilon]}\,
\spL\big\{M(E\setminus(\tmedcup_{j=1}^nE_j))\big\}(\d t)
\\[.2cm]
&\xrightarrow[n\to\infty]{}
\delta_0(\R\setminus[-\epsilon,\epsilon])=0,
\end{split}
\end{equation*}
which means that $\sum_{j=1}^nM(E_j)\to M(\medcup_{n\in\N}E_n)$ in the
measure topology. Similar argumentation yields that condition
\eqref{Equiv_def_of_FLB_eq1} implies condition \ref{def_FLB(d)} of
Definition~\ref{def_FLB}.

\end{lenumerate}
\end{remarks}

Concerning existence of the various types of free random measures
we have the following main result:

\begin{theorem}\label{existence_FLB}
Let $\CE$ be a $\delta$-ring in a non-empty set $X$, and for each
$E$ in $\CE$ let $\nu(E,\cdot)$ be a Borel
probability measure on $\R$.
Assume that whenever $E_1,\ldots,E_n$ are disjoint
sets from $\CE$ we have that
\begin{equation}
\nu(\textstyle{\medcup_{j=1}^nE_j},\cdot)
=\nu(E_1,\cdot)\boxplus\cdots\boxplus\nu(E_n,\cdot).
\label{key_existence_eq}
\end{equation}
Then the following assertions hold:

\begin{enumerate}[i]

\item\label{existence_FLB(a)}
There exists a $W^*$-probability space $(\CM,\tau)$ and a FAFRM
$M=\{M(E)\mid E\in\CE\}$ affiliated
with $(\CM,\tau)$, such that $\spL\{M(E)\}=\nu(E,\cdot)$ for all $E$ in
$\CE$.

\item\label{existence_FLB(b)}
If the given family $\{\nu(E,\cdot)\mid E\in\CE\}$ satisfies that
$\nu(E_n,\cdot)\overset{\rm w}{\to}\delta_0$ as $n\to\infty$ for any
sequence $(E_n)_{n\in\N}$ of sets from $\CE$, such that
$E_n\downarrow\emptyset$, then $M$ described in \ref{existence_FLB(a)} 
is automatically a FCRM.

\item\label{existence_FLB(c)}
If the given family $\{\nu(E,\cdot)\mid E\in\CE\}$ satisfies, in
addition to the condition in \ref{existence_FLB(b)}, that
$\nu(E_n,\cdot)\in\ID(\boxplus)$ for all $E$ in $\CE$, then $M$
described in \ref{existence_FLB(a)} is automatically a FLB. 
\end{enumerate}
\end{theorem}

Note that condition \eqref{key_existence_eq} entails that
$\nu(\emptyset)=\delta_0$.
The proof of Theorem~\ref{existence_FLB} is obtained as the
culmination of a series of preliminary lemmas and is deferred to the concluding
section of the paper. In the remainder of the present  section we focus on
establishing, under various additional assumptions, a decomposition of a FCRM into
the sum of a FLB and a ``purely atomic part'', such that the two
terms in the decomposition are freely independent. In order to derive
these decompositions we shall need the following natural result.

\begin{lemma}\label{lem:disjunkt_opspaltning}
Let $\CE$ be a $\delta$-ring in a non-empty set $X$, and let
$M=\{M(E)\mid E\in\CE\}$
be a FAFRM (respectively FCRM or FLB) affiliated with a
$W^*$-probability space $(\CM,\tau)$. Let further $A$ be a fixed set
from $\sigma(\CE)$. Then the formulae
\[
M_1(E)=M(E\cap A), \qand M_2(E)=M(E\cap(X\setminus A)),
\qquad(E\in\CE),
\]
define new FAFRMs (respectively FCRMs or FLBs) on $(X,\CE)$, and $M_1$
and $M_2$ are freely independent in the sense that the $W^*$-algebras
generated by the families $\{M_1(E)\mid E\in\CE\}$ and $\{M_2(E)\mid
E\in\CE\}$ are freely independent.
\end{lemma}

\begin{proof}
Note first of all that \eqref{eq_Hereditary_prop_delta_ring} ensures
that $M_1$ and $M_2$ are well-defined. Secondly, it is straightforward
to check that $M_1$ and $M_2$ satisfy any of the conditions
\ref{def_FLB(b)}-\ref{def_FLB(a)} in Definition~\ref{def_FLB},
provided that $M$ satisfies that same condition. It remains therefore
to argue that $M_1$ and $M_2$ are freely independent as stated, and for this it
suffices to argue that the unital algebras $\CA_1$ and $\CA_2$ generated by $M_1$
and $M_2$, respectively, are freely independent (cf.\
\cite[Proposition~2.5.7]{vdn}). To 
validate the latter assertion it is sufficient to argue that any given
finite subsets $\{a_1,\ldots,a_n\}$ of $\CA_1$ and $\{b_1,\ldots,b_m\}$ of
$\CA_2$ generate freely independent unital
subalgebras of $\CA$. In this setup there exist finitely many sets
$E_1,\ldots,E_k$ from $\CE$ and functions
$f_1,\ldots,f_k$ from $\CBF_b(\R)$, such that each $a_j$ is a
(non-commutative) polynomial in (some of) the variables $f_j(M_1(E_j))$,
$j=1,\ldots,k$. Since $\CE$ is closed under intersections and
set-differences, we can subsequently choose finitely many disjoint sets
$F_1,\ldots,F_l$ from $\CE$, such that each $E_j$ is a union of some
of the $F_i$'s. Let $\CB_1$ denote the $W^*$-subalgebra of $\CA$
generated by $\{M_1(F_1),\ldots,M_1(F_l)\}$. Then each $M_1(E_j)$ is
affiliated with $\CB_1$, being the sum of some of the $M_1(F_i)$'s. In
particular $f_j(M_1(E_j))\in\CB_1$ for all $j$, and hence also
$a_1,\ldots,a_n\in\CB_1$. Similarly there exists a finite family
$G_1,\ldots,G_r$ of disjoint sets from $\CE$, such that
$b_1,\ldots,b_m\in\CB_2$, with $\CB_2$ being the $W^*$-subalgebra of
$\CA$ generated by $\{M_2(G_1),\ldots,M_2(G_r)\}$.
Now, by Definition~\ref{def_FLB}\ref{def_FLB(b)} and the definitions
of $M_1$ and $M_2$ the operators
\[
M_1(F_1),\ldots,M_1(F_l),M_2(G_1),\ldots,M_2(G_r)
\]
are freely independent, and hence $\CB_1$ and $\CB_2$ are also freely
independent (cf.\ \cite[Proposition~2.5.5]{vdn}). Obviously this
further entails that the unital algebras generated by
$\{a_1,\ldots,a_n\}$ and $\{b_1,\ldots,b_r\}$, respectively, are
freely independent as well.
\end{proof}

\subsection{Decomposition of a positive FCRM}
\label{subsec:decompI}

In this subsection we consider a non-empty set $X$ equipped
with a $\delta$-ring $\CE$ and a FCRM $M=\{M(E)\mid E\in\CE\}$
affiliated with a $W^*$-probability space $(\CM,\tau)$. We assume
throughout that $M(E)$ is positive for all $E$ in $\CE$ in the sense
that $\spe(M(E))\subseteq[0,\infty)$, or, equivalently, that
$\supp(\spL\{M(E)\})\subseteq[0,\infty)$ for all $E\in\CE$.
We shall derive a decomposition of $M$ into the sum of an ``atomic''
FCRM and a FLB, a kin to the fundamental decomposition obtained by
Kingman in \cite{Ki67} for classical completely random measures (CRM).
The latter was obtained via the Laplace transforms of the considered
CRM, which give rise to a positive measure on the underlying
measurable space. This approach cannot directly be transferred to the
non-commutative operator setting, since the formula
$\exp(A+B)=\exp(A)\exp(B)$ does not generally hold for selfadjoint
operators $A$ and $B$, unless they commute. 
Our construction given below therefore follows a different but related
path, where the mentioned Laplace transforms are replaced by the set function
\begin{equation}
\mu(E):=\int_0^{\infty}t\,\spL\{M(E)\}(\d t), \qquad(E\in\CE),
\label{dec_eq1}
\end{equation}
which we shall argue is a measure on $\CE$.
In case $M(E)$ is bounded (and hence an element of $\CM$)
it holds automatically that $\mu(E)<\infty$, since the
appearing integral equals $\tau(M(E))$. In general, when it is finite,
$\mu(E)$ may also be identified with
the first free cumulant of $\spL\{M(E)\}$
(cf.\ \eqref{Moment_Cumulant_Formula}), which we shall mostly denote
simply by $\kappa_1(M(E))$ rather than   
$\kappa_1(\spL\{M(E)\})$ to avoid too heavy notation.

\begin{lemma}\label{cumulant_measure_1}
In the setting described above the formula
\[
\mu(E)=\int_0^{\infty}t\, \spL\{M(E)\}(\d t),
\qquad(E\in\CE),
\]
defines a (positive) measure on $(X,\CE)$.
\end{lemma}

\begin{proof}
Throughout this proof we employ for brevity the notation $\nu_E$ for the
spectral distribution $\spL\{M(E)\}$ for any $E$ from $\CE$.
From Definition~\ref{def_FLB}\ref{def_FLB(c)} it follows that $M(\emptyset)=0$, so
that $\nu_{\emptyset}=\delta_0$, and hence $\mu(\emptyset)=0$. 

We show next that $\mu$ is finitely additive on $\CE$. For this note
first that if $A$, $B$ are sets from $\CE$, such that $A\subseteq
B$, then $M(A)\le M(B)$, since
$M(B)-M(A)=M(B\setminus A)$ is positive. According to
Lemma~3.3 in \cite{bv} this means that the distribution
functions $F_{\nu_A}$ and $F_{\nu_B}$ of $\nu_A$ and $\nu_B$ satisfy that
$F_{\nu_A}(t)\ge F_{\nu_B}(t)$ for all $t$ in $\R$.
For any $K$ in $[0,\infty)$ this further implies that
\begin{equation}
\begin{split}
\int_{(K,\infty)}t\, \nu_{A}(\d t)
&=\int_0^{\infty}\nu_{A}(\{t\in\R\mid t1_{(K,\infty)}(t)>s\})\6s
\\
&=\int_0^{\infty}\nu_{A}((K,\infty)\cap(s,\infty))\6s
=\int_0^{\infty}(1-F_{\nu_A}(K\vee s) )\6s
\\
&\le\int_0^{\infty}(1-F_{\nu_B}(K\vee s) )\6s
=\int_{(K,\infty)}t\, \nu_{B}(\d t).
\label{dec_eq19}
\end{split}
\end{equation}
Consider now two disjoint sets $E_1,E_2$ from $\CE$. If
$\mu(E_1)=\infty$ or $\mu(E_2)=\infty$, then by \eqref{dec_eq19} (in
the case $K=0$) it follows that $\mu(E_1\cup
E_2)\ge\mu(E_1)\vee\mu(E_2)=\infty$, and hence $\mu(E_1\cup
E_2)=\mu(E_1)+\mu(E_2)$ in this case. If $\mu(E_1),\mu(E_2)<\infty$,
then since $M(E_1)$ and $M(E_2)$ are freely independent, and since the free cumulants
linearize free convolution (cf.\ \eqref{Linearizing_Property_II}), it
follows that
\begin{equation*}
\begin{split}
\mu(E_1\cup E_2)&=\kappa_1(M(E_1\cup E_2))=\kappa_1(M(E_1)+M(E_2))
=\kappa_1(\nu_{E_1}\boxplus\nu_{E_2})
\\
&=\kappa_1(\nu_{E_1})+\kappa_1(\nu_{E_2})
=\mu(E_1)+\mu(E_2).
\end{split}
\end{equation*}
Consider finally a sequence $(E_n)_{n\in\N}$ of sets from $\CE$ such that
$E:=\medcup_{n\in\N}E_n\in\CE$. We must show that
$\mu(E)=\sum_{n=1}^{\infty}\mu(E_n)$.
Recall first from 
Remark~\ref{rem:Equiv_def_of_FLB}\ref{rem:Equiv_def_of_FLB(B)}
that $M(\medcup_{j=1}^nE_j)\to M(E)$ in the measure topology as
$n\to\infty$. In particular $\nu_{\cup_{j=1}^nE_j}\to\nu_E$ weakly as
$n\to\infty$ (cf.~\eqref{prel_eq2}), and this further entails that
\begin{equation}
\begin{split}
\mu(E)&=\int_0^{\infty}t\,\nu_{E}(\d t)
\le\liminf_{n\to\infty}\int_0^{\infty}t\,\nu_{\cup_{j=1}^nE_j}(\d t)
=\liminf_{n\to\infty}\mu\big(\medcup_{j=1}^nE_j\big)\\
&=\liminf_{n\to\infty}\sum_{j=1}^n\mu(E_j)
=\sum_{n=1}^{\infty}\mu(E_n),
\label{dec_eq20}
\end{split}
\end{equation}
where we also invoked the finite additivity of $\mu$ established
above. By \eqref{dec_eq20} we may assume in the following
that $\mu(E)<\infty$, and by \eqref{dec_eq19} this further entails that
$\mu(E_n)\le\mu(\medcup_{j=1}^nE_j)<\infty$ for all $n$. From the finite
additivity of $\mu$ we have that
\begin{equation*}
\mu(E)-\sum_{j=1}^n\mu(E_j)=\mu\big(E\setminus\medcup_{j=1}^nE_j\big)
\end{equation*}
for all $n$. Setting $G_n=E\setminus\medcup_{j=1}^nE_j$ for all $n$,
it suffices thus to show that
\begin{equation}
\mu(G_n)=\int_0^{\infty}t\,\nu_{G_n}(\d t)
\longrightarrow 0 \quad\text{as $n\to\infty$.}
\label{dec_eq3}
\end{equation}
Since $G_N\downarrow\emptyset$ as $n\to\infty$, we know from
condition \ref{def_FLB(d)} in Definition~\ref{def_FLB} that
$\nu_{G_n}\wto\delta_0$, and hence it is well-known that the convergence
in \eqref{dec_eq3} is equivalent to the condition that the family
$\{\nu_{G_n}\mid n\in\N\}$ is uniformly integrable in the sense
that
\begin{equation}
\forall\epsilon\in(0,\infty) \ \exists K\in(0,\infty)\colon
\sup_{n\in\N}\int_{(K,\infty)}t\,\nu_{G_n}(\d t)\le\epsilon.
\label{dec_eq4}
\end{equation}
From \eqref{dec_eq19} we have for any $n$ in $\N$ and $K$ in $[0,\infty)$
that
\begin{equation*}
\int_{(K,\infty)}t\, \nu_{G_n}(\d t)\le
\int_{(K,\infty)}t\, \nu_{G_1}(\d t).
\end{equation*}
Since the right hand side does not depend on $n$ and
converges to 0 as $K\to\infty$ (by dominated convergence) it follows
readily that \eqref{dec_eq4} is satisfied.
\end{proof}

As described in Subsection~\ref{subsec:delta_rings} the measure $\mu$
introduced in Lemma~\ref{cumulant_measure_1} can be extended to a
(positive) measure on $\sigma(\CE)$, which we also denote by $\mu$. We
shall assume in the following that $\mu$ is $\sigma$-finite. This
assumption may be seen as an analog of the (less restrictive) condition
``$\CC$'' presupposed in \cite{Ki67}.
In the remainder of this section we shall assume further that the $\delta$-ring
$\CE$ satisfies condition \eqref{eq_determining_sequence}. 

Recall that an atom for $\mu$ is a set $A$ from $\sigma(\CE)$, such
that $\mu(A)>0$ and $\mu(B\cap A)\in\{0,\mu(A)\}$ for any set $B$
from $\sigma(\CE)$. It is well-known that any
$\sigma$-finite measure may be decomposed into the sum of a purely
atomic part and an atom-less part. More specifically
there exists a subset $I$ of $\N$ and a corresponding family
$(A_n)_{n\in I}$ of disjoint atoms for $\mu$, such that if we put
$\GA=\medcup_{n\in I}A_n$, and
\[
\mu_c(B):=\mu(B\cap(X\setminus\GA)), \qquad(B\in\sigma(\CE)),
\]
then the measure $\mu_c$ does not have any atoms.
The ``atomic part'' of $\mu$ is then
concentrated to the measure
\[
\mu_a(B):=\mu(B\cap\GA)=\sum_{n\in I}\mu(B\cap A_n),
\qquad(B\in\sigma(\CE)),
\]
and the mentioned decomposition is
\begin{equation}
\mu=\mu_a+\mu_c.
\label{dec_eq13}
\end{equation}
Note that the $\sigma$-finiteness of $\mu$ prevents any atom of $\mu$
from having infinite $\mu$-measure. In particular
\begin{equation}
\mu(A_n)\in(0,\infty) \quad\text{for all $n$ in $\N$.}
\label{dec_eq18}
\end{equation}
By a theorem of W.~Sierpi\'nski, the atom-free part $\mu_c$ has the
following property: Any set $B$ from $\sigma(\CE)$, such that
$0<\mu_c(B)<\infty$, admits for any $n$ in $\N$ a decomposition 
$B=\medcup_{j=1}^nB_j$ into disjoint sets $B_1,\ldots,B_n$ from
$\sigma(\CE)$, such that 
\begin{equation}
\mu_c(B_j)=\frac{\mu_c(B)}{n}, \qquad j=1,\ldots,n.
\label{dec_eq7}
\end{equation}
Corresponding to \eqref{dec_eq13} we consider now the decomposition
$M=M_a+M_c$, where
\begin{equation}
M_a(E):=M(E\cap\GA), \qand M_c(E):=M(E\cap(X\setminus\GA)) \quad\text{for
  any $E$ in $\CE$.}
\label{dec_eq6}
\end{equation}
We then have the following result.

\begin{theorem}\label{thm:decomp_I}
Let $X$ be a non-empty set, and let $\CE$ be a $\delta$-ring in $X$
satisfying condition \eqref{eq_determining_sequence}. Let further
$M=\{M(E)\mid E\in\CE\}$ be a positive FCRM on $(X,\CE)$, satisfying that
the measure $\mu$ introduced in Lemma~\ref{cumulant_measure_1} is
$\sigma$-finite, and consider the decomposition
\[
M(E)=M_a(E)+M_c(E), \qquad(E\in\CE),
\]
described above. Then $M_c$ and $M_a$ are freely independent,
$M_c$ is a free L\'evy basis, and there exists a
countable family $(T_n)_{n\in I}$ of operators from $\{M_a(E)\mid E\in\CE\}$,
such that
\begin{equation}
M_a(E)=\sum_{n\in I}\tfrac{\mu(A_n\cap E)}{\mu(A_n)} T_n,
\qquad(E\in\CE),
\label{dec_eq5}
\end{equation}
where $(A_n)_{n\in I}$ is the family of disjoint atoms for $\mu$ described above.
\end{theorem}

Concerning formula \eqref{dec_eq5}, note that $\frac{\mu(A_n\cap
  E)}{\mu(A_n)}\in\{0,1\}$ for any $E$ in $\sigma(\CE)$ and any $n$ in
$I$, since $A_n$ is an atom for $\mu$. Note also that the sum
converges in the measure topology in case $I$ is infinite
(cf.~Lemma~\ref{lem:disjunkt_opspaltning} and
Remark~\ref{rem:Equiv_def_of_FLB}\ref{rem:Equiv_def_of_FLB(B)}).

\begin{proofof}[Proof of Theorem~\ref{thm:decomp_I}.]
It follows directly from Lemma~\ref{lem:disjunkt_opspaltning} that 
$M_a$ and $M_c$ defined by
\eqref{dec_eq6} are freely independent FCRM's on $(X,\CE)$. In
order to prove that $M_c$ is a FLB, it 
remains then to verify that $\spL\{M_c(E)\}\in\ID(\boxplus)$ for any
given $E$ from $\CE$. We assume first that $\mu_c(E)<\infty$.
As described above we may then, for any $n$ in $\N$,
choose disjoint sets $E_1^{(n)},\ldots,E_n^{(n)}$ from $\sigma(\CE)$, such that
$E=\medcup_{j=1}^nE_j^{(n)}$, and $\mu_c(E_j^{(n)})=\frac{\mu_c(E)}{n}$,
$j=1,\ldots,n$ (cf.\ \eqref{dec_eq7}). From \eqref{eq_Hereditary_prop_delta_ring}
it follows in particular that $E_j^{(n)}\in\CE$ for all $j,n$, and therefore
\begin{equation*}
\spL\{M_c(E)\}=\spL\{M_c(E_1^{(n)})+\cdots+M_c(E_n^{(n)})\}
=\spL\{M_c(E_1^{(n)})\}\boxplus\cdots\boxplus\spL\{M_c(E_n^{(n)})\}.
\end{equation*}
Appealing now to \cite[Theorem~1]{bp3} it suffices  to prove that
the family $\{\spL\{M_c(E_j^{(n)})\}\mid n\in\N, \ j=1,\ldots,n\}$ is a
null-array in the sense that
\begin{equation}
\forall \epsilon\in(0,\infty)\colon \max_{1\le j\le n}
\spL\{M_c(E_j^{(n)})\}([-\epsilon,\epsilon]^c)
\xrightarrow[n\to\infty]{}0.
\label{dec_eq8}
\end{equation}
Given $\epsilon$ in $(0,\infty)$ it follows from Markov's Inequality
that
\begin{equation*}
\begin{split}
\spL\{M_c(E_j^{(n)})\}([-\epsilon,\epsilon]^c)
&=\spL\{M_c(E_j^{(n)})\}((\epsilon,\infty))
\le\frac{1}{\epsilon}\int_0^{\infty}t\,\spL\{M(E_j^{(n)}\setminus\GA)\}(\d t)
\\
&=\frac{1}{\epsilon}\mu(E_j^{(n)}\setminus\GA)
=\frac{1}{\epsilon}\mu_c(E_j^{(n)})
=\frac{1}{n\epsilon}\mu_c(E)
\end{split}
\end{equation*}
for any $j,n$. Since the resulting expression does not depend on $j$,
this validates \eqref{dec_eq8}.

Assume next that $\mu_c(E)=\infty$. Since $\mu_c$ is $\sigma$-finite, we
may choose a sequence $(E_n)_{n\in\N}$ of disjoint sets from $\sigma(\CE)$,
such that $E=\medcup_{n\in\N}E_n$ and $\mu_c(E_n)<\infty$ for all
$n$. By \eqref{eq_Hereditary_prop_delta_ring} we have that $E_n\in\CE$ for all $n$,
and the argument above then ensures that
$\spL\{M_c(E_n)\}\in\ID(\boxplus)$ for all $n$. Furthermore
Remark~\ref{rem:Equiv_def_of_FLB}\ref{rem:Equiv_def_of_FLB(B)}
in conjunction with \eqref{prel_eq2} yield that
\[
\spL\{M_c(E_1)\}\boxplus\cdots\boxplus\spL\{M_c(E_n)\}=
\spL\big\{\textstyle{\sum_{j=1}^nM_c(E_j)}\big\}
\xrightarrow[n\to\infty]{\rm w}\spL\{M_c(E)\}.
\]
Since $\ID(\boxplus)$ is closed under free convolution and 
weak convergence, this yields that $\spL\{M_c(E)\}\in\ID(\boxplus)$ also
in this case.

It remains to establish \eqref{dec_eq5}. By
Remark~\ref{rem:Equiv_def_of_FLB}\ref{rem:Equiv_def_of_FLB(B)} 
we note first for any $E$ in $\CE$ that
\[
M_a(E)=M(E\cap\GA)=\sum_{n\in I}M(E\cap A_n),
\quad\text{where $\mu(E\cap A_n)\in\{0,\mu(A_n)\}$ for all $n$.}
\]
In case $0=\mu(E\cap A_n)=\int_0^{\infty}t\,\spL\{M(E\cap A_n)\}(\d
t)$, it follows that $\spL\{M(E\cap A_n)\}=\delta_0$, and hence
$M(E\cap A_n)=0$, since $\tau$ is faithful.
In order to establish \eqref{dec_eq5} it suffices therefore to verify that
\begin{equation*}
M(E\cap A_n)=M(E'\cap A_n)
\end{equation*}
whenever $E,E'\in\CE$ such that 
$\mu(E\cap A_n)=\mu(A_n)=\mu(E'\cap A_n)$. 
But given such $E,E'$, note that (cf.\ \eqref{dec_eq18})
\[
\mu(A_n\cap E\setminus E')\le\mu(A_n\setminus
E')=\mu(A_n)-\mu(A_n\cap E')=0,
\]
and hence it follows as above by faithfulness of $\tau$ that
$M(A_n\cap E\setminus E')=0$, and similarly that $M(A_n\cap
E'\setminus E)=0$. Therefore
\[
M(A_n\cap E)=M(A_n\cap E\cap E')+M(A_n\cap E\setminus E')
=M(A_n\cap E\cap E')=M(A_n\cap E'),
\]
as desired. This completes the proof.
\end{proofof}

\subsection{Decomposition of a signed FCRM}
\label{subsec:decompII} 

Let $\CE$ be a $\delta$-ring on a non-empty set $X$, and let
$M=\{M(E)\mid E\in\CE\}$ be a FCRM. 
In this subsection we establish a decomposition similar to that
obtained in the previous subsection in the more general situation,
where we drop the assumption of positivity. The corresponding problem
for ``signed'' CRMs was not considered by Kingman; presumably because
the approach using Laplace transforms is not directly applicable. Our
approach to the case of ``signed'' FCRMs requires stronger
moment conditions than those considered in the positive case, where
$\sigma$-finiteness of the measure introduced in
Lemma~\ref{cumulant_measure_1} was presupposed.
Specifically we require in the following
existence of second moments, i.e.\
\begin{equation}
\int_{\R}t^2\,\spL\{M(E)\}(\d t)<\infty \qquad\text{for any $E$ in $\CE$,}
\label{dec_eq8a}
\end{equation}
but we shall actually need slightly more than that (see
Lemma~\ref{cumulant_measure_2} and Remark~\ref{rem:moments_conditions}
below). The existence of second moments allows us to consider the
second free cumulant (cf.\ \eqref{Moment_Cumulant_Formula})
\begin{equation}
\kappa_2(\spL\{M(E)\})
=\int_{\R}t^2\,\spL\{M(E)\}(\d t)-\Big(\int_{\R}t\,\spL\{M(E)\}(\d t)\Big)^2\ge0,
\label{dec_eq12}
\end{equation}
which we denote for brevity by $\kappa_2(M(E))$

\begin{lemma}\label{cumulant_measure_2}
Let $M=\{M(E)\mid E\in\CE\}$ be a FCRM satisfying condition
\eqref{dec_eq8a}. Assume additionally that
\begin{equation}
\lim_{n\to\infty}\kappa_2(M(E_n))=0
\qquad\text{for any sequence $(E_n)_{n\in\N}$ from $\CE$, such that
$E_n\downarrow\emptyset$.} 
\label{dec_eq9}
\end{equation}
Then the formulae
\begin{equation*}
\begin{split}
\mu_1(E)&=\kappa_1(M(E))=\int_{\R}t\,\spL\{M(E)\}(\d t),
\\
\mu_2(E)&=\kappa_2(M(E)),
\end{split}
\end{equation*}
define, respectively, a signed measure $\mu_1$ and a positive measure
$\mu_2$ on $(X,\CE)$.
\end{lemma}

\begin{proof}
As in the proof Lemma~\ref{cumulant_measure_1} it follows that
$\mu_1(\emptyset)=\mu_2(\emptyset)=0$, and that $\mu_1$ and $\mu_2$
are finitely additive on $\CE$, since $\kappa_1$ and $\kappa_2$ both
linearize $\boxplus$. For a sequence $(D_n)_{n\in\N}$ of disjoint sets
from $\CE$, such that $D:=\medcup_{n\in\N}D_n\in\CE$, the finite
additivity and condition \eqref{dec_eq8a} further ensure the validity
of the calculation:
\[
\mu_2(D)-\sum_{j=1}^{n}\mu_2(D_j)=\mu_2\big(D\setminus\medcup_{j=1}^nD_j\big)
=\kappa_2\big(M\big(D\setminus\medcup_{j=1}^nD_j\big)\big),
\]
and hence \eqref{dec_eq9} supplies the remaining condition for $\mu_2$
to be a measure on $(X,\CE)$. In order to complete the proof it remains
therefore only to verify that also
\begin{equation}
\lim_{n\to\infty}\kappa_1(M(E_n))=0
\qquad\text{for any sequence $(E_n)_{n\in\N}$ from $\CE$, such that
$E_n\downarrow\emptyset$.}
\label{dec_eq10}
\end{equation}
Consider thus such a sequence $(E_n)_{n\in\N}$, and for brevity put
$\nu_n=\spL\{M(E_n)\}$ for each $n$. Then 
Definition~\ref{def_FLB}\ref{def_FLB(d)} entails that $\nu_n\wto\delta_0$
as $n\to\infty$. For any positive $\epsilon$ this, in conjunction
with \eqref{dec_eq9}, leads to
\begin{equation*}
\begin{split}
\nu_n\big(\{t\in\R\mid |t-\kappa_1(\nu_n)|+|t|>2\epsilon\}\big)
&\le\nu_n\big(\{t\in\R\mid |t-\kappa_1(\nu_n)|>\epsilon\}\big)
+\nu_n\big([-\epsilon,\epsilon]^c\big)
\\
&\le\epsilon^{-2}\int_{\R}(t-\kappa_1(\nu_n))^2\,\nu_n(\d t)
+\nu_n\big([-\epsilon,\epsilon]^c\big)
\\
&=\epsilon^{-2}\kappa_2(\nu_n)+\nu_n\big([-\epsilon,\epsilon]^c\big)
\xrightarrow[n\to\infty]{}0.
\end{split}
\end{equation*}
In particular $\{t\in\R\mid| t-\kappa_1(\nu_n)|+|t|\le2\epsilon\}
\ne\emptyset$ for all sufficiently large $n$, and for such $n$ we can choose $t_n$
in $\R$ such that $|t_n-\kappa_1(\nu_n)|+|t_n|\le2\epsilon$. But then also
$|\kappa_1(\nu_n)|\le|\kappa_1(\nu_n)-t_n|+|t_n|\le2\epsilon$, and
this verifies \eqref{dec_eq10}.
\end{proof}

\begin{remark}\label{rem:moments_conditions}
The assumption \eqref{dec_eq9} in Lemma~\ref{cumulant_measure_2}
may appear rather ``artificial'', as it is essentially
equivalent to the statement that $\mu_2$ is a measure. The proof of 
Lemma~\ref{cumulant_measure_2} shows that \eqref{dec_eq9} implies that
$\kappa_1(M(E_n))\to0$ as $n\to\infty$, and hence also that
\begin{equation*}
\int_{\R}t^2\,\spL\{M(E_n)\}(\d t)
=\kappa_2(M(E_n))+\kappa_1(M(E_n))^2
\xrightarrow[n\to\infty]{}0
\end{equation*}
for any sequence $(E_n)_{n\in\N}$ from $\CE$, such that $E_n\downarrow\emptyset$.
Thus \eqref{dec_eq9} is in fact -- in the considered setup -- equivalent
to convergence to 0 in the square mean, which by standard results is equivalent to
uniform integrability of the sequence
\[
\{\spL\{M(E_n)\}\circ{\rm sq}^{-1}\mid n\in\N\}
\]
of transformations of $\spL\{M(E_n)\}$ by the mapping ${\rm
sq}\colon x\mapsto x^2\colon\R\to\R$.
Consequently a more elaborate condition on $M$,
ensuring \eqref{dec_eq9}, is that the family 
$\{\spL\{M(E')\}\circ{\rm sq}^{-1}\mid E'\in\CE, \ E'\subseteq E\}$ be
uniformly integrable for any $E$ in $\CE$. This latter condition is
satisfied, in particular, if there exists a positive number
$\epsilon$, such that
\[
\sup_{E'\in\CE \atop E'\subseteq E}
\int_{\R}|t|^{2+\epsilon}\,\spL\{M(E')\}(\d t)<\infty
\quad\text{for any $E$ in $\CE$,}
\]
and one could even allow for $\epsilon$ to depend on $E$.
\end{remark}

In the setting of Lemma~\ref{cumulant_measure_2} we consider next the
positive measure 
\begin{equation}
\mu=|\mu_1|+\mu_2,
\label{dec_eq10a}
\end{equation}
where $|\mu_1|$ denotes the total variation measure of the signed
measure $\mu_1$ (cf.\ Subsection~\ref{subsec:delta_rings}). We extend
$\mu$ to a measure on $\sigma(\CE)$ (also denoted $\mu$) and
assume again that $(X,\CE)$ satisfies condition
\eqref{eq_determining_sequence}. In combination with \eqref{dec_eq8a}
this entails that $\mu$ is
$\sigma$-finite and hence it admits an atomic decomposition: 
\[
\mu=\mu_a+\mu_c
\]
as described in Subsection~\ref{subsec:decompI}. Specifically we
introduce a countable family 
$(A_n)_{n\in I}\subseteq\sigma(\CE)$ of disjoint atoms for $\mu$, such
that
\[
\mu_a(B)=\mu(B\cap\GA) \qand \mu_c(B)=\mu(B\setminus\GA)
\quad\text{for any $B$ in $\sigma(\CE)$,}
\]
where $\GA=\medcup_{n\in I}A_n$. We consider then the corresponding
decomposition $M=M_a+M_c$ of $M$, where
\[
M_a(E)=M(E\cap\GA) \qand M_c(E)=M(E\setminus\GA) \quad\text{for any $E$ in
  $\CE$.}
\]

\begin{theorem}\label{thm:decomp_II}
Let $X$ be a non-empty set, and let $\CE$ be a $\delta$-ring in $X$
satisfying condition \eqref{eq_determining_sequence}. Let further
$M=\{M(E)\mid E\in\CE\}$ be a FCRM on $(X,\CE)$ affilliated with a
$W^*$-probability space $(\CM,\tau)$ and satisfying
\eqref{dec_eq8a} and \eqref{dec_eq9}. Consider also the
decomposition
\[
M(E)=M_a(E)+M_c(E), \qquad(E\in\CE),
\]
described above. Then $M_c$ and $M_a$ are freely independent,
$M_c$ is a free L\'evy basis, and there exists a
countable family $(T_n)_{n\in I}$ of operators from $\{M_a(E)\mid E\in\CE\}$,
such that
\begin{equation}
M_a(E)=\sum_{n\in I}\tfrac{\mu(A_n\cap E)}{\mu(A_n)} T_n,
\qquad(E\in\CE).
\label{dec_eq14}
\end{equation}
Here $\mu$ is given by \eqref{dec_eq10a} and $(A_n)_{n\in I}$ is the
family of disjoint atoms for $\mu$ described above.
\end{theorem}

\begin{proof}
The proof is similar to that of Theorem~\ref{thm:decomp_I}, and we
shall not repeat all details. It follows directly from
Lemma~\ref{lem:disjunkt_opspaltning} that $M_a$ and $M_c$ are freely
independent FCRM's. To show that $\spL\{M_c(E)\}\in\ID(\boxplus)$ for
any $E$ in $\CE$, we use the facts that $\mu_c$ is atom-less and that
$\mu_c(E)<\infty$ to choose,
for any $n$ in $\N$, disjoint sets $E^{(n)}_1,\ldots,E_n^{(n)}$ from $\CE$, such
that $E=\medcup_{j=1}^nE_j^{(n)}$, and such that
$\mu_c(E_j^{(n)})=\frac{\mu_c(E)}{n}$, $j=1,\ldots,n$. 
Note then for any $n$ in $\N$ that
\begin{equation*}
\spL\big\{M_c(E)\big\}
=\delta_{\mu_1(E\setminus\GA)}\boxplus\Big(
\fc_{j=1}^n\spL\big\{M_c(E_j^{(n)})-\mu_1(E_j^{(n)}\setminus\GA)\unit_{\CM}\big\} 
\Big),
\end{equation*}
where $\unit_{\CM}$ denotes the multiplicative unit of $\CM$.
By \cite[Theorem~1]{bp3} it suffices thus to show that
\begin{equation}
\forall \epsilon\in(0,\infty)\colon \max_{1\le j\le n}
\spL\{M_c(E_j^{(n)})-\mu_1(E_j^{(n)}\setminus\GA)\unit_{\CM}\}([-\epsilon,\epsilon]^c)
\xrightarrow[n\to\infty]{}0.
\label{dec_eq15}
\end{equation}
Given $\epsilon$ in $(0,\infty)$, we find for any $j,n$ by
Chebyshev's Inequality that
\begin{equation*}
\begin{split}
\spL\big\{M_c(E_j^{(n)})-&\mu_1(E_j^{(n)}\setminus\GA)\unit_{\CM}\big\}
([-\epsilon,\epsilon]^c)
\\[.2cm]
&=\spL\{M_c(E_j^{(n)})\}\big(\{t\in\R\mid
|t-\kappa_1(M_c(E_j^{(n)}))|>\epsilon\}\big)
\\[.2cm]
&\le\epsilon^{-2}\kappa_2(M_c(E_j^{(n)}))
=\epsilon^{-2}\kappa_2(M(E_j^{(n)}\setminus\GA))
\\
&\le\epsilon^{-2}\mu(E_j^{(n)}\setminus\GA)
=\epsilon^{-2}\mu_c(E_j^{(n)})
=\frac{\mu_c(E)}{n\epsilon^2},
\end{split}
\end{equation*}
from which \eqref{dec_eq15} follows readily.
It remains to verify \eqref{dec_eq14}. Note initially that
\[
M_a(E)=M(E\cap\GA)=\sum_{n\in I}M(E\cap A_n)
\]
for any $E$ in $\CE$ by \eqref{eq_Hereditary_prop_delta_ring} and 
Remark~\ref{rem:Equiv_def_of_FLB}\ref{rem:Equiv_def_of_FLB(B)}.
If $\mu(E\cap A_n)=0$, 
then in particular the variance $\kappa_2(M(E\cap A_n))=0$, and hence 
$\spL\{M_a(E\cap A_n)\}=\delta_c$ for some $c\in\R$. Since also
$|\mu_1|(E\cap A_n)=0$, and therefore 
$\kappa_1(M_a(E\cap A_n))=\mu_1(E\cap A_n)=0$, we must
then have that $c=0$. By faithfulness of $\tau$ this implies that
$M(E\cap A_n)=0$. To verify \eqref{dec_eq14} it suffices
therefore to argue for any $n$ in $\N$ and any $E,E'$ from $\CE$ that
\begin{equation}
\mu(E\cap A_n)=\mu(A_n)=\mu(E'\cap A_n) \implies
M(E\cap A_n)=M(E'\cap A_n).
\label{dec_eq17}
\end{equation}
Assuming the left hand side of \eqref{dec_eq17} it suffices as in the
proof of Theorem~\ref{thm:decomp_I} to show that $M(A_n\cap E\setminus
E')=M(A_n\cap E'\setminus E)=0$, and as argued above this follows by
faithfulness of $\tau$, if
we validate that $\mu(A_n\cap E\setminus E')=\mu(A_n\cap
E'\setminus E)=0$. But this follows exactly as in the proof of
Theorem~\ref{thm:decomp_I}.
\end{proof}

\section{Free L\'evy bases}
\label{sec:Rajput_Rosinski_theory_for_FLB}

For a Free L\'evy basis the $\boxplus$-infinite divisibility of the
marginals makes it possible to transfer major elements of the theory
of classical L\'evy bases, as developed in \cite{RR}, to the free
setting via the Bercovici-Pata bijection. Following that strategy we
record in this section some basic results on free L\'evy bases. The
starting point is the following:

\begin{theorem}\label{BP-bij_for_Levy_Bases}
Let $\CE$ be a $\delta$-ring of subsets of a non-empty set $X$.

\begin{enumerate}[i]

\item\label{BP-bij_for_Levy_Bases(i)}
For any free L\'evy basis $M=\{M(E)\mid E\in\CE\}$ (affiliated
  with some $W^*$-probability space) there exists a classical L\'evy
  basis $N=\{N(E)\mid E\in\CE\}$ defined on some probability space
  $(\Omega,\CF,P)$, such that
\begin{equation}
\Lambda(L\{N(E)\})=\spL\{M(E)\} \quad\text{for all $E$ in $\CE$.}
\label{BP-bij_for_Levy_Bases_eq1}
\end{equation}

\item\label{BP-bij_for_Levy_Bases(ii)}
For any classical L\'evy basis $N=\{N(E)\mid E\in\CE\}$
  (defined on some classical probability space) there exists a
  free L\'evy basis $M=\{M(E)\mid E\in\CE\}$ affiliated with some
  $W^*$-probability space $(\CM,\tau)$, such that  the relation
  \eqref{BP-bij_for_Levy_Bases_eq1} holds.

\end{enumerate}
\end{theorem}

\begin{proof} 
\begin{lenumerate}[i]

\item Consider a free L\'evy basis $M=\{M(E)\mid E\in\CE\}$
affiliated with some $W^*$-probability space $(\CM,\tau)$, and for
each $E$ in $\CE$ put
\[
\mu(E,\cdot)=\Lambda^{-1}(\spL\{M(E)\}).
\]
If $E_1,\ldots,E_n$ are disjoint sets from $\CE$ we have then that
\begin{equation}
\begin{split}
\mu\big(\textstyle{\medcup_{j=1}^n}E_j,\cdot\big)
&=\Lambda^{-1}\big(\spL\big\{M(E_1)+\cdots+M(E_n)\big\}\big)
\\
&=\Lambda^{-1}\big(\spL\{M(E_1)\}\boxplus\cdots\boxplus \spL\{N(E_n)\}\big)
\\
&=\Lambda^{-1}(\spL\{M(E_1)\})*\cdots*\Lambda^{-1}(\spL\{M(E_n)\})
\\
&=\mu(E_1,\cdot)*\cdots*\mu(E_n,\cdot).
\end{split}
\label{BP-bij_for_Levy_Bases_eq2}
\end{equation}
In addition, for any decreasing sequence $(F_n)_{n\in\N}$ from
$\CE$, such that $\medcap_{n\in\N}F_n=\emptyset$, we have that
$\spL\{M(F_n)\}\overset{\mathrm w}{\to}\delta_0$ as $n\to\infty$, and hence
by continuity of $\Lambda^{-1}$,  
\begin{equation}
\mu(F_n,\cdot)\xrightarrow[n\to\infty]{\mathrm w}\delta_0. 
\label{BP-bij_for_Levy_Bases_eq3}
\end{equation}
It follows from \eqref{BP-bij_for_Levy_Bases_eq2} and the
Kolmogorov Extension Theorem that there exists a finitely additive,
infinitely divisible random measure 
$N=\{N(E)\mid E\in\CE\}$, defined on some probability space
$(\Omega,\CF,P)$, such that $L\{N(E)\}=\mu(E,\cdot)$ for all $E$ in
$\CE$. If $(E_n)_{n\in\N}$ is a sequence of disjoint sets from $\CE$, such
that $E:=\medcup_{n\in\N}E_n\in\CE$, then
\eqref{BP-bij_for_Levy_Bases_eq3} implies that
$\sum_{j=1}^nN(E_j)\to N(E)$ in probability as $n\to\infty$. Since the
terms $N(E_1),N(E_2),N(E_3),\ldots$ are independent,
the convergence also holds almost surely. Hence $N$ is a classical
L\'evy basis.

\item Let $N=\{N(E)\mid E\in\CE\}$ be a classical L\'evy basis defined on
some probability space $(\Omega,\CF,P)$, and for any $E$ in $\CE$ put
\[
\nu(E,\cdot)=\Lambda(L\{N(E)\})\in\ID(\boxplus).
\]
Argumentation similar to that of the proof of
\ref{BP-bij_for_Levy_Bases(i)} verifies that
the family $\{\nu(E,\cdot)\mid E\in\CE\}$ satisfies
\eqref{key_existence_eq} and the conditions in \ref{existence_FLB(b)}
and \ref{existence_FLB(c)} of Theorem~\ref{existence_FLB}.
Hence that same theorem provides the
existence of a FLB with the described properties.
\end{lenumerate}
\end{proof}

Next we transfer some fundamental results
from \cite{RR} on classical L\'evy bases to corresponding results for
free L\'evy bases. In the remaining part of this section we
consider thus, as in \cite{RR}, a $\delta$-ring $\CE$ in $X$ satisfying
condition \eqref{eq_determining_sequence}.
Note that without loss of generality we may
assume that the $U_n$'s from \eqref{eq_determining_sequence}
are disjoint or increasing in $n$.

\begin{proposition}\label{prop:free_RR_1}

\begin{enumerate}[i]

\item\label{prop:free_RR_1(i)}
Let $M=\{M(E)\mid E\in\CE\}$ be a
free L\'evy basis affiliated with some $W^*$-probability space
$(\CM,\tau)$. Then there exist a signed measure $\Theta\colon\CE\to\R$, a
  finite (positive) measure $\Sigma\colon\CE\to[0,\infty)$ and a
  $\sigma$-finite (positive) measure
  $F\colon\sigma(\CE)\otimes\CB(\R)\to[0,\infty]$ such that the free
  L\'evy-Khintchine representation of $M(E)$ is given by
\begin{equation}
\CC_{M(E)}(z)=z\Theta(E)+z^2\Sigma(E)+\int_{\R}\Big(\frac{1}{1-tz}-1-z\varsigma(t)\Big)\,
F_E(\d t), \quad(z\in\C^-),
\label{prop:free_RR_1_eq1}
\end{equation}
for any $E$ in $\CE$. Here $F_E$ is the measure on $\CB(\R)$ given
by: $F_E(B)=F(E\times B)$ for any $B$ in $\CB(\R)$, and $F_E$ is a
L\'evy measure on $\R$ for all $E$ in $\CE$.

\item\label{prop:free_RR_1(ii)} 
For any triplet $(\Theta,\Sigma,F)$ of measures as described in
  \ref{prop:free_RR_1(i)}, there exists a free L\'evy basis
  $M=\{M(E)\mid E\in\CE\}$ (affiliated with some $W^*$-probability
  space), such that \eqref{prop:free_RR_1_eq1} holds.

\item\label{prop:free_RR_1(iii)} Let $M$ and $(\Theta,\Sigma,F)$ be as
  stated in \ref{prop:free_RR_1(i)}. Then there exists a unique,
  $\sigma$-finite and positive measure $\kappa$ on $\sigma(\CE)$
  with the following properties:

\begin{enumerate}[a]

\item\label{prop:free_RR_1(a)}    
$\kappa(E)=|\Theta|(E)+\Sigma(E)+\int_{\R}\min\{1,x^2\}\,F_E(\d x)$ for
all $E$ in $\CE$.

\item\label{prop:free_RR_1(b)}
If $(E_n)_{n\in\N}$ is a sequence of sets from $\CE$, such that
$\kappa(E_n)\to0$ as $n\to\infty$, then $M(E_n)\to0$ in the measure
topology.

\item\label{prop:free_RR_1(c)}
Suppose $(E_n)_{n\in\N}$ is a sequence of sets from $\CE$, such that
$M(E_n')\to0$ in the measure topology for any sequence
$(E_n')_{n\in\N}$ from
$\CE$, such that $E_n'\subseteq E_n$ for all $n$. Then
$\kappa(E_n)\to0$ as $n\to\infty$.

\end{enumerate}
\end{enumerate}
\end{proposition}

The triplet $(\Theta,\Sigma,F)$ of measures introduced in 
Proposition~\ref{prop:free_RR_1}\ref{prop:free_RR_1(i)}
is referred to as the \emph{free characteristic triplet} of the free L\'evy
basis $M$. 
The measure $\kappa$ is referred to as the \emph{control
  measure} of $M$.
 
For the measure $F_E$ in \ref{prop:free_RR_1}\ref{prop:free_RR_1(i)} 
it follows e.g.\ by a standard extension
argument that a Borel function $f\colon\R\to\C$ is in $\CL^1(F_E)$, if
and only if $1_E\otimes f\in\CL^1(F)$, in which case
\begin{equation}
\int_{\R}f(t)\, F_E(\d t)=\int_{X\times\R}1_E(x)f(t)\, F(\d x,\d t).
\label{integralformel_1}
\end{equation}
We note also that the measure $F$ is uniquely determined on the $\sigma$-algebra
$\sigma(\CE)\otimes\CB(\R)$ by the condition: $F(E\times B)=F_E(B)$
for all $E$ in $\CE$ and $B$ in $\CB(\R)$, since this also implies that
$F(U_n\times(\R\setminus[-\frac{1}{n},\frac{1}{n}]))<\infty$ for all
$n$, because $F_{U_n}$ is a L\'evy measure. Here $(U_n)_{n\in\N}$ is the
sequence from \eqref{eq_determining_sequence}, chosen to be increasing.

\begin{proofof}[Proof of Proposition~\ref{prop:free_RR_1}.]    
\begin{lenumerate}[i]

\item
Let $N=\{N(E)\mid E\in\CE\}$ be a classical L\'evy basis
corresponding to $M$ as described in
Theorem~\ref{BP-bij_for_Levy_Bases}. Then by Proposition~2.1 and
Lemma~2.3 in \cite{RR} there exist a signed measure
$\Theta\colon\CE\to\R$, a finite measure
$\Sigma\colon\CE\to[0,\infty)$ and a $\sigma$-finite measure
$F\colon\sigma(\CE)\otimes\CB(\R)\to[0,\infty]$, such that 
\begin{equation}
C_{N(E)}(y)=\ci y\Theta(E)-\tfrac{1}{2}y^2\Sigma(E)
+\int_{\R}\big(\e^{\ci ty}-1-\ci y\varsigma(t)\big)\,F_E(\d t), \quad(y\in\R),
\label{prop:free_RR_1_eq2}
\end{equation}
for all $E$ in $\CE$. Since $\spL\{M(E)\}=\Lambda(L\{N(E)\})$ for
all $E$ in $\CE$, it follows immediately from the definition of
$\Lambda$ that $(\Theta,\Sigma,F)$ satisfies \eqref{prop:free_RR_1_eq1} as well.

\item If $(\Theta,\Sigma,F)$ is a triplet as described in
  \ref{prop:free_RR_1(i)}, then Proposition~2.1 in \cite{RR} ensures
  the existence of a classical L\'evy basis $N=\{N(E)\mid E\in\CE\}$
  such that \eqref{prop:free_RR_1_eq2} holds. Subsequently
  Theorem~\ref{BP-bij_for_Levy_Bases} provides a free L\'evy basis
  $M=\{M(E)\mid E\in\CE\}$ such that \eqref{prop:free_RR_1_eq1}
  holds. 

\item Let $N$ be as in the proof of \ref{prop:free_RR_1(i)}. It
  follows then from Proposition~2.1 in \cite{RR} that there exists a
  $\sigma$-finite measure $\kappa$ on $\sigma(\CE)$, such that 
\ref{prop:free_RR_1(a)} is satisfied, and such that
\ref{prop:free_RR_1(b)} and \ref{prop:free_RR_1(c)} hold with $M$
replaced by $N$ and convergence in the measure topology replaced by
convergence in probability. Since e.g.\ in \ref{prop:free_RR_1(b)}
$M(E_n)\to0$ in the measure topology, if and only if $N(E_n)\to0$ in
probability, it follows readily that $\kappa$ satisfies 
\ref{prop:free_RR_1(b)} and \ref{prop:free_RR_1(c)} as they stand.
\end{lenumerate}
\end{proofof}

Since the triplet $(\Theta,\Sigma,F)$ appearing in
Proposition~\ref{prop:free_RR_1}\ref{prop:free_RR_1(i)}
is obtained by application of Proposition~2.1 in \cite{RR} to a
classical L\'evy basis, it follows immediately from Lemma~2.3 of that
same paper, that there exists a mapping $\rho\colon
X\times\CB(\R)\to[0,\infty]$ with the following properties:

\begin{enumerate}[i]

\item\label{def_rho_kernel(i)} 
$\rho(x,\cdot)$ is a L\'evy measure on $\R$ for any fixed $x$ in
 $X$.

\item\label{def_rho_kernel(ii)}  
$\rho(\cdot,B)$ is a $\sigma(\CE)$-measurable function for any
  fixed Borel subset $B$ of $\R$.

\item\label{def_rho_kernel(iii)}   
For any function $h\colon X\times\R\to\C$ which is positive and
  $\sigma(\CE)\otimes\CB(\R)$-measurable or in $\CL^1(F)$ it
  holds that
\begin{equation}
\int_{X\times \R}h(x,t)\, F(\d x,\d t)=\int_X\Big(\int_\R h(x,t)\, \rho(x,\d
t)\Big)\,\kappa(\d x),
\label{integralformel_2}
\end{equation}
where $\kappa$ is the control measure introduced in 
Proposition~\ref{prop:free_RR_1}\ref{prop:free_RR_1(iii)}, and
the integral $\int_{\R}h(x,t)\,\rho(x,\d t)$ is well-defined for
$\kappa$-almost all $x$ in $X$.

\end{enumerate}
We note also that $\rho$ is unique up to $\kappa$-null-sets: If $\rho'\colon
X\times\R\to[0,\infty]$ is another mapping satisfying conditions
\ref{def_rho_kernel(i)}-\ref{def_rho_kernel(iii)}, then
$\rho(x,\cdot)=\rho'(x,\cdot)$ for $\kappa$-almost all $x$, since
$\CB(\R)$ is countably generated.

For the mapping $\rho$ we now have the following analog of
Proposition~2.4 in \cite{RR}.

\begin{proposition}\label{prop:free_RR_2}
Consider a free L\'evy basis 
$M=\{M(E)\mid E\in\CE\}$, 
and let $\Theta,\Sigma,F,\kappa$ be the associated measures described in
Proposition~\ref{prop:free_RR_1}. Furthermore let
$\rho$ be the corresponding mapping introduced above, and let $\theta$
and $\sigma^2$ denote, respectively, the Radon-Nikodym derivatives of
$\Theta$ and $\Sigma$ with respect to $\kappa$.
Then for any set $E$ from $\CE$ we have the formula:
\begin{equation*}
\CC_{M(E)}(z)=\int_{E}R(x,z)\,\kappa(\d x), \qquad(z\in\C^-),
\end{equation*}
where the kernel $R(\cdot,\cdot)$ is given by
\begin{equation*}
R(x,z)=z\theta(x)+z^2\sigma^2(x)
+\int_{\R}\Big(\frac{1}{1-tz}-1-z\varsigma(t)\Big)\,\rho(x,\d t), 
\end{equation*}
for all $z$ in $\C^-$ and $x$ in $X$.
\end{proposition}

\begin{proof}
By the definition of $\kappa$ in
\ref{prop:free_RR_1}\ref{prop:free_RR_1(iii)} it is clear that
$|\Theta|,\Sigma\le\kappa$, so that $\theta$ and $\sigma^2$ are
well-defined (see Subsection~\ref{subsec:delta_rings}).
Let $E$ be a given set from $\CE$. For any fixed $z$ in $\C^-$ the function
$t\mapsto\frac{1}{1-tz}-1-z\varsigma(t)$ belongs to $\CL^1(F_E)$ and to
$\CL^1(\rho(x,\cdot))$ for all $x$ in $X$, since the considered measures are
all L\'evy measures. Combining formulae \eqref{integralformel_1} and
\eqref{integralformel_2} it follows further that
\begin{equation*}
\begin{split}
\int_E\Big(\int_{\R}\Big(\frac{1}{1-tz}-1-z\varsigma(t)\Big)\,\rho(x,\d
t)\Big)\,\kappa(\d x)
&=\int_{X\times\R}1_E(x)\Big(\frac{1}{1-tz}-1-z\varsigma(t)\Big)\, F(\d x,\d t)
\\
&=\int_{\R}\Big(\frac{1}{1-tz}-1-z\varsigma(t)\Big)\, F_E(\d t).
\end{split}
\end{equation*}
Therefore, by the definitions of $R$, $\theta$ and $\sigma^2$,
\begin{equation*}
\int_{E}R(x,z)\,\kappa(\d x)
=z\Theta(E)+z^2\Sigma(E)+\int_{\R}\Big(\frac{1}{1-tz}-1-z\varsigma(t)\Big)\,
F_E(\d t)
=\CC_{M(E)}(z),
\end{equation*}
where the last equality is \eqref{prop:free_RR_1_eq1}.
\end{proof}

\begin{remark}\label{rem:Levy_seed}
For fixed $x$ in $X$ the ``slice-function'' $R(x,\cdot)$ of the kernel
$R$ in Proposition~\ref{prop:free_RR_2} is the free cumulant transform
of a freely infinitely divisible probability measure $\nu_x$ with free
characteristic triplet $(\theta(x),\sigma^2(x),\rho(x,\cdot))$. In the 
literature on classical L\'evy bases the measure $\mu_x$ in $\ID(*)$
with classical characteristic triplet
$(\theta(x),\sigma^2(x),\rho(x,\cdot))$ is often referred to as the
\emph{L\'evy seed} at $x$ of the classical L\'evy basis $N=\{N(E)\mid
E\in\CE\}$ corresponding to $M$ as in
Theorem~\ref{BP-bij_for_Levy_Bases}. By analogy we 
refer to $\nu_x$ as the (free) \emph{L\'evy seed of $M$ at
$x$}. Proposition~\ref{prop:free_RR_2} then asserts that the
distribution of the free L\'evy basis $M$ is uniquely determined by
the family $\{\nu_x\mid x\in X\}$ of L\'evy seeds and the control
measure $\kappa$. Accordingly we refer to the quadruplet
$(\theta,\sigma^2,\rho,\kappa)$ as the \emph{free characteristic quadruplet
  of $M$}. 
We note further, that since $\Lambda(\mu_x)=\nu_x$ for all $x$
in $X$, it is apparent that the one-to-one correspondence in
Theorem~\ref{BP-bij_for_Levy_Bases} really takes place at the
infinitesimal level, i.e.\ by applying $\Lambda$ at the level of the
``infinitesimal seeds''.
\end{remark}

We consider next a fundamental class of
examples of free L\'evy bases,
namely the so-called \emph{factorizable} free L\'evy bases.

\begin{examples}\label{ex:Factorizable_FLB}
\begin{lenumerate}[1]

\item\label{ex:Factorizable_FLB(1)}
Let $\nu$ be a measure from $\ID(\boxplus)\setminus\{\delta_0\}$ with
  free characteristic 
triplet $(a,b,r)$. Let further $X$ be a non-empty set
equipped with a $\delta$-ring $\CE$, and let
$\eta\colon\CE\to[0,\infty]$ be a measure on $\CE$.
Finally put
$\CE_0=\{E\in\CE\mid \eta(E)<\infty\}$, and note that $\CE_0$ is
again a $\delta$-ring. We assume that $\CE_0$ satisfies condition
\eqref{eq_determining_sequence}. This implies in particular that
$\sigma(\CE_0)=\sigma(\CE)$, 
and that $\eta$ extends uniquely to a $\sigma$-finite measure on
$\sigma(\CE)$.

For each $E$ in $\CE_0$ we denote by $\nu(E,\cdot)$ the measure
in $\ID(\boxplus)$ with free characteristic triplet
\[
\eta(E)\cdot(a,b,r):=(\eta(E)a,\eta(E)b,\eta(E)r).
\]
If $E_1,\ldots,E_n$ are disjoint sets from $\CE_0$, then
$\nu(E_1,\cdot)\boxplus\cdots\boxplus\nu(E_n,\cdot)$ has free
characteristic triplet
\[
\big(\textstyle{\sum_{k=1}^n\eta(E_k)}\big)\cdot(a,b,r)
=\eta\big(\medcup_{k=1}^nE_k)\cdot(a,b,r),
\]
and thus equals $\nu(\medcup_{k=1}^nE_k,\cdot)$. Hence by
Theorem~\ref{existence_FLB} there exists a free L\'evy basis
$M_{(\eta,\nu)}=\{M_{(\eta,\nu)}(E)\mid E\in\CE_0\}$ such that
$\spL\{M_{(\eta,\nu)}(E)\}=\nu(E,\cdot)$ for all $E$ in $\CE_0$.
It is straightforward to check that in the considered set-up, the
triplet $(\Theta,\Sigma,F)$ for $M_{(\eta,\nu)}$ described in
Proposition~\ref{prop:free_RR_1} is given by
\[
\Theta=a\eta, \quad \Sigma=b\eta \qand F=\eta\otimes r,
\]
and consequently the control measure $\kappa$ is given by
\[
\kappa(E)=|a|\eta(E)+b\eta(E)+\eta(E)\int_{\R}\min\{1,t^2\}\,r(\d t)
=c_{\nu}\eta(E), \qquad(E\in\sigma(\CE)),
\]
where the constant $c_\nu$ is given by
\begin{equation}
c_\nu=|a|+b+\int_{\R}\min\{1,t^2\}\,r(\d t).
\label{ex:Factorizable_FLB_eq1}
\end{equation}
Note that $c_{\nu}>0$, since $\nu\ne\delta_0$. 
The L\'evy seed $\nu_x$ at a point $x$ in
$X$ (cf.\ Remark~\ref{rem:Levy_seed}) consequently
has free characteristic triplet $(\theta(x),\sigma^2(x),\rho(x,\cdot))$
given by
\[
\theta(x)=c_\nu^{-1}a, \quad \sigma^2(x)=c_{\nu}^{-1}b,
\qand \rho(x,\cdot)=c_\nu^{-1}r.
\]
In particular the L\'evy seed $\nu_x$ does not depend on $x$. We refer
to free L\'evy bases in this form as \emph{factorizable free L\'evy
  bases}.

\item\label{ex:Factorizable_FLB(2)}
In the special case where the measure $\nu$ considered in \ref{ex:Factorizable_FLB(1)} is
the standard semi-circle distribution
$\frac{1}{2\pi}\sqrt{4-x^2}1_{[-2,2]}(x)\6x$
we replace the notation $M_{(\eta,\nu)}$ by $G_\eta$ and refer
to $G_\eta$ as a \emph{semi-circular L\'evy bases.}
As the free characteristic triplet $(a,b,r)$ of $\nu$
is $(0,1,0)$ in this case, the free cumulant
transform of $G_\eta(E)$ is given by
$\CC_{G_\eta(E)}(z)=z^2\eta(E)$ for all $E$ in $\CE_0$. In other words $G_\eta(E)$ has
the semi-circle distribution
\[ 
\frac{1}{2\pi\eta(E)}\sqrt{4\eta(E)-t^2}1_{[-2\eta(E)^{1/2},2\eta(E)^{1/2}]}(t)\6t.
\]
In particular $G_{\eta}(E)$ is a bounded operator for all $E$ in $\CE_0$.
The triplet $(\Theta,\Sigma,F)$ equals $(0,\eta,0)$, and
the constant $c_\nu$ in \eqref{ex:Factorizable_FLB_eq1}
equals 1, so the characteristic quadruplet is
$(0,1,0,\eta)$. For each $x$ in $X$ the L\'evy seed $\nu_x$ is
simply $\nu$ itself.

\item\label{ex:Factorizable_FLB(3)}
In the special case where the measure $\nu$ considered in
\ref{ex:Factorizable_FLB(1)} is the free Poisson distribution
$\frac{1}{2\pi t}\sqrt{t(4-t)}1_{[0,4]}(t)\6t$ with parameter 1, we
replace the notation $M_{(\eta,\nu)}$ by $P_\eta$.

As the free characteristic triplet $(a,b,r)$ of $\nu$ is
$(1,0,\delta_1)$ in this case, 
the free cumulant transform of $P_\eta(E)$ is given by
\[
\CC_{P_\eta(E)}(z)=\eta(E)z+\eta(E)\Big(\frac{1}{1-z}-1-z\Big)
=\eta(E)\Big(\frac{1}{1-z}-1\Big), \qquad(z\in\C^-),
\] 
for all $E$ in $\CE_0$. In other words (see e.g.\ \cite[page~35]{vdn})
the spectral distribution of 
$P_{\eta}(E)$ is the free Poisson distribution $\Poiss^{\boxplus}(\eta(E))$
with parameter $\eta(E)$ given by 
\begin{equation*}
\Poiss^{\boxplus}(\eta(E))(\d t)=
\begin{cases}
(1-\eta(E))\delta_0+\frac{1}{2\pi t}\sqrt{(t-s)(u-t)}1_{[s,u]}(t)\6t,
&\text{if $\eta(E)\le1$,}
\\
\frac{1}{2\pi t}\sqrt{(t-s)(u-t)}1_{[s,u]}(t)\6t, &\text{if $\eta(E)>1$,}
\end{cases}
\end{equation*}
where $s=(1-\eta(E))^2$ and $u=(1+\eta(E))^2$. In particular
$P_{\eta}(E)$ is a bounded operator for all $E$ in $\CE_0$.
The triplet $(\Theta,\Sigma,F)$ equals $(\eta,0,\eta\otimes\delta_1)$, and the
constant $c_\nu$ in \eqref{ex:Factorizable_FLB_eq1} equals 2, so that
the characteristic quadruplet is
$(\frac{1}{2},0,\frac{1}{2}\delta_1,2\eta)$. For each $x$ in $X$ the
L\'evy seed $\nu_x$ is the Poisson distribution with parameter
$\frac{1}{2}$.

Free L\'evy bases in this form were previously considered under the
name \emph{free Poisson random measures} in \cite{BNT05}.

\end{lenumerate}
\end{examples}


We close this section by stating two propositions, both of which describe
natural and useful constructions with free L\'evy bases. As the
proofs of these propositions are rather simple, we leave them as
exercises for the interested reader.

\begin{proposition}\label{prop:transformation_of_FLB}
Let $M=\{M(E)\mid E\in\CE\}$ be a free L\'evy basis with free
characteristic triplet $(\Theta,\Sigma,F)$. Let further $\phi\colon X\to Y$
be  mapping from $X$ into a non-empty set $Y$, and define
\[
\CF^0=\{H\subseteq Y\mid \phi^{-1}(H)\in\CE\}
\]
and
\[
M\circ\phi^{-1}=\{M(\phi^{-1}(H))\mid H\in\CF^0\}.
\]
Then $\CF^0$ is a $\delta$-ring and $M\circ\phi^{-1}$ is a free L\'evy
basis. If $\CF^0$ satisfies \eqref{eq_determining_sequence}, then the
free characteristic triplet of $M\circ\phi^{-1}$ is given by
$(\Theta\circ\phi^{-1},\Sigma\circ\phi^{-1},F\circ(\phi,\id_{\R})^{-1})$,
where $\id_{\R}$ denotes the identity function on $\R$ and
$(\phi,\id_{\R})\colon X\times\R\to Y\times\R$ is the function given
by
\begin{equation*}
(\phi,\id_{\R})(x,t)=(\phi(x),t) \qquad(x\in X, \ t\in\R).
\end{equation*}
\end{proposition}  

\begin{proposition}\label{prop:concentration_of_FLB}
Let $M=\{M(E)\mid E\in\CE\}$ be a free L\'evy basis with free
characteristic triplet $(\Theta,\Sigma,F)$. Let further $A$ be a fixed
set from $\sigma(\CE)$, and define
\[
\CE^A=\{E\in\sigma(\CE)\mid A\cap E\in\CE\},
\]
and
\[
M^A(E)=M(A\cap E), \qquad(E\in\CE^A).
\]
Then $M^A$ is a new free L\'evy basis on $(X,\CE)$ with free characteristic
triplet $(\Theta^A,\Sigma^A,F^A)$ given by
\[
\Theta^A(E)=\Theta(A\cap E), \quad \Sigma^A(E)=\Sigma(A\cap E),
\qand
F^A(E\times B)=F((A\cap E)\times B)
\]
for any $E$ from $\CE$ and any Borel subset $B$ of $\R$.
\end{proposition}

\section{Integration with respect to free L\'evy Bases}
\label{sec:Integration_wrt_FLB}

In this section we develop a theory of integration with respect to
free L\'evy bases in parallel to the corresponding theory for
classical L\'evy bases in \cite{RR}.
Throughout the section we consider a
$\delta$-ring in a non-empty set $X$, and we assume condition
\eqref{eq_determining_sequence}.
We consider further a free L\'evy basis $\{M(E)\mid E\in\CE\}$ affiliated with a
$W^*$-probability space $(\CM,\tau)$.
Now let $s\colon X\to\R$ be a simple $\CE$-measurable function in the form:
\begin{equation}
s=\sum_{k=1}^n\alpha_k1_{A_k},
\quad
\text{where $\alpha_1,\ldots,\alpha_n\in\R$, and $A_1,\ldots,A_n$ are
disjoint sets from $\CE$.}
\label{EFLB_6}
\end{equation}
Then for any $A$ in $\sigma(\CE)$ we define the integral $\int_As\6M$ of $s$
over $A$ with respect to $M$ by the formula:
\begin{equation}
\int_As\6M=\sum_{k=1}^n\alpha_kM(A_k\cap A).
\label{EFLB_5}
\end{equation}

\begin{remarks}\label{rem:integral_of_simple_functions}
\begin{lenumerate}

\item[(1)]
The right hand side of \eqref{EFLB_5} is well-defined, since
$A_k\cap A\in\CE$ for all $k$ (cf.\ \eqref{eq_Hereditary_prop_delta_ring}).
Denote by $\SM(\CE)$ the class of functions in the form \eqref{EFLB_6}.
Since $\CE$ is in particular stable under finite intersections, it 
follows from the definition of a free L\'evy basis and
standard argumentation that $\SM(\CE)$ is a vector space, that
the right hand side of \eqref{EFLB_5} does not depend on the choice of the
representation \eqref{EFLB_6} and that
\begin{equation*}
\int_A(\alpha s+s')\6M=\alpha\int_As\6M+\int_As'\6M
\end{equation*}
for any $s,s'$ in $\SM(\CE)$, any $A$ in $\sigma(\CE)$ and any $\alpha$ in
$\R$. The definition of a free L\'evy basis further entails that
$\spL\{\int_As\6M\}\in\ID(\boxplus)$ for any $s$ in $\SM(\CE)$ and $A$
in $\sigma(\CE)$.

\item[(2)]
 Let $N=\{N(E)\mid E\in\CE\}$ be a classical L\'evy basis
  corresponding to $M$ as in Theorem~\ref{BP-bij_for_Levy_Bases}. 
  Then for any $s$ in $\SM(\CE)$ and $A$ in $\sigma(\CE)$ the integral
  $\int_As\6N$ is defined
  in \cite{RR} exactly as above (with $M$ replaced by $N$). It
  follows then from the algebraic properties of $\Lambda$ that
\begin{equation*}
\Lambda\Big(L\Big\{\int_As\6N\Big\}\Big)=\spL\Big\{\int_As\6M\Big\}.
\end{equation*}
\end{lenumerate}
\end{remarks}

In parallel to \cite{RR} we define next the class
$\CL^1(M)$ of real valued functions on $X$ that are integrable with
respect to $M$.

\begin{definition}\label{def:M-integrable_function}
Let $f\colon X\to\R$ be a $\sigma(\CE)$-$\CB(\R)$-measurable function. Then
$f$ is called $M$-integrable, if there exists a
sequence $(s_n)_{n\in\N}$ from $\SM(\CE)$
such that the following two conditions are satisfied:

\begin{enumerate}[a]

\item $\lim_{n\to\infty}s_n=f$ almost everywhere with respect to the control
  measure for $M$.

\item For any $A$ in $\sigma(\CE)$, the sequence $(\int_As_n\6M)_{n\in\N}$
  converges in the measure topology on $(\CM,\tau)$.

\end{enumerate}
The class of $M$-integrable functions $f\colon X\to\R$ is denoted by $\CL^1(M)$.

\end{definition}

For a classical L\'evy basis $N=\{N(E)\mid E\in\CE\}$ the class of
$N$-integrable functions, here denoted by $\CL^1(N)$, was introduced
in \cite{RR} exactly as in
Definition~\ref{def:M-integrable_function}, but with $M$ replaced by $N$
and convergence in the measure topology replaced by convergence in
probability. 

For an $M$-integrable function $f$ it is natural to define the
integral $\int_Af\6M$ with respect to $M$ as the limit 
of the sequence appearing in
Definition~\ref{def:M-integrable_function}(b). We proceed next to show
that this limit does not depend on the choice of approximating
sequence satisfying conditions (a) and (b) from the afore mentioned
definition, while simultaneously establishing that
$\CL^1(M)=\CL^1(N)$, if $N$ is the classical L\'evy basis
corresponding to $M$ as in Theorem~\ref{BP-bij_for_Levy_Bases}.

\begin{proposition}\label{prop:integral_veldef}
Let $N=\{N(E)\mid
E\in\CE\}$ be a classical  L\'evy basis corresponding to $M$ as in
Theorem~\ref{BP-bij_for_Levy_Bases}. Then the following assertions hold:

\begin{enumerate}[i]

\item $\CL^1(M)=\CL^1(N)$.

\item If $f\in\CL^1(M)$ and $(s_n)_{n\in\N}$ and $(t_n)_{n\in\N}$ are two sequences
from $\SM(\CE)$, both satisfying conditions (a) and (b) of
Definition~\ref{def:M-integrable_function},
then for any $A$ in $\sigma(\CE)$ the sequences $(\int_As_n\6M)_{n\in\N}$ and
$(\int_At_n\6M$) share the same limit in the measure topology.

\end{enumerate}
\end{proposition}

\begin{proof} Let $f$ be a function from $\CL^1(M)$, and
let $(s_n)_{n\in\N}$ be a sequence from $\SM(\CE)$ satisfying
conditions (a) and (b) of Definition~\ref{def:M-integrable_function}.
Then for any $n,m$ in $\N$ it follows from (1) and (2) in
Remark~\ref{rem:integral_of_simple_functions} that
\begin{equation*}
L\Big\{\int_As_n\6N-\int_As_m\6N\Big\}
=\Lambda^{-1}\Big(\spL\Big\{\int_As_n\6M-\int_As_m\6M\Big\}\Big)
\longrightarrow\delta_0,
\quad\text{as $n,m\to\infty$,}
\end{equation*}
so that $(\int_As_n\6N)_{n\in\N}$ is a Cauchy-sequence in probability and hence
convergent in probability. Since $M$ and $N$ have the same control
measure, this verifies that $f\in\CL^1(N)$ and hence the inclusion
$\CL^1(M)\subseteq\CL^1(N)$. 
The reverse inclusion follows by similar
argumentation, applying $\Lambda$ rather than $\Lambda^{-1}$ and using
completeness of the measure topology. Hence (i) follows.

Assume next that $(t_n)_{n\in\N}$ is another sequence from
$\SM(\CE)$ satisfying conditions (a) and (b) of
Definition~\ref{def:M-integrable_function}. The same argumentation as above
then shows that the sequence
$(\int_At_n\6N)_{n\in\N}$ converges in probability as well for any $A$ in
$\sigma(\CE)$, and
it follows then from \cite{uw} that
the limit must equal that of $(\int_As_n\6N)_{n\in\N}$.
Therefore the mixed sequence 
\begin{equation*}
\int_As_1\6N,\int_At_1\6N,\int_As_2\6N,\int_At_2\6N,\ldots
\end{equation*}
is also convergent and hence a Cauchy sequence in
probability. Arguing as above the properties of $\Lambda$ entail that
the sequence
\begin{equation*}
\int_As_1\6M,\int_At_1\6M,\int_As_2\6M,\int_At_2\6M,\ldots
\end{equation*}
is then a Cauchy sequence and hence convergent in the measure
topology. Since the measure topology is Hausdorff, (ii) now follows by
sub-sequence considerations. 
\end{proof}

\begin{definition}\label{def:integral_of_L1_function}
Let $f$ be an $M$-integrable function, and let $(s_n)_{n\in\N}$ be a
sequence of functions from $\SM(\CE)$ satisfying conditions (a) and
(b) of Definition~\ref{def:M-integrable_function}. Then for any $A$ in $\sigma(\CE)$
the integral $\int_Af\6M$ of $f$ over $A$ with respect to $M$ is
defined by
\begin{equation*}
\int_Af\6M=\lim_{n\to\infty}\int_As_n\6M,
\end{equation*}
where the limit is in the measure topology.
\end{definition}

In the following remark we list next a number of rather immediate
properties of the integral introduced in the definition above.

\begin{remarks}\label{rem:properties_of_integral}

\begin{lenumerate}[1]

\item\label{rem:properties_of_integral(0)}
If $f$ is in $\CL^1(M)$ and $A\in\CE$, then $\int_Af\6M$ is a selfadjoint operator
affiliated with $(\CM,\tau)$ (because $f$ is real-valued and the
adjoint operation is continuous in the measure topology).

\item\label{rem:properties_of_integral(1)}
It follows Proposition~\ref{prop:integral_veldef}(i) that
  $\CL^1(M)$ is a vector space. This also follows directly from
  Definition~\ref{def:M-integrable_function} and the fact that the
  linear operations on $\overline{\CM}$ are continuous in the measure topology. 
This latter fact together with
Remark~\ref{rem:integral_of_simple_functions}(1) further entail that
\begin{equation*}
\int_A(\alpha f+g)\6M=\alpha\int_Af\6M+\int_Ag\6M
\end{equation*}
for any $f,g$ in $\CL^1(M)$, $A$ in $\sigma(\CE)$ and $\alpha$ in $\R$.

\item\label{rem:properties_of_integral(2)}
In Remark~\ref{rem:integral_of_simple_functions}(1) we observed
  that $\int_As\6M\in\ID(\boxplus)$ for any $s$ in $\SM(\CE)$ and
  $A$ in $\sigma(\CE)$. Since $\ID(\boxplus)$ is closed under weak convergence, and
  since convergence in the measure topology implies weak convergence
  of the spectral distributions (cf.\ \eqref{prel_eq2}),
  it follows that $\int_Af\6M\in\ID(\boxplus)$ for any $f$ in
  $\CL^1(M)$. Moreover, the continuity of $\Lambda$ in combination
  with Remark~\ref{rem:integral_of_simple_functions}(2) imply that 
\begin{equation*}
\Lambda\Big(L\Big\{\int_Af\6N\Big\}\Big)=\spL\Big\{\int_Af\6M\Big\}
\end{equation*}
for any $f$ in $\CL^1(M)=\CL^1(N)$, where $N$ is a classical L\'evy basis
corresponding to $M$ as in Theorem~\ref{BP-bij_for_Levy_Bases}.

\item\label{rem:properties_of_integral(3)}
Let $K$ be a positive integer, let $f_1,\ldots,f_K$ be
  functions from $\CL^1(M)$, and let $A_1,\ldots,A_K$ be
  \emph{disjoint} sets from $\sigma(\CE)$. Then the integrals
  $\int_{A_1}f_1\6M,\ldots,\int_{A_K}f_K\6M$ are freely independent
  operators in $\overline{\CM}$. This fact follows immediately from
  \eqref{EFLB_5} and the definition of a free L\'evy basis,
  in case $f_1,\ldots,f_K\in\SM(\CE)$. The extension to general
  functions in $\CL^1(M)$ subsequently follows directly from
  Definition~\ref{def:integral_of_L1_function}, since free
  independence is preserved under limits in the measure topology (see
  Proposition~5.4 in \cite{BNT05}). 

\item\label{rem:properties_of_integral(4)}
Since $\CL^1(M)=\CL^1(N)$ (with $N$ as in
\ref{rem:properties_of_integral(2)}), Theorem~2.7 in \cite{RR}
immediately provides the following 
characterization of $\CL^1(M)$ in terms of the characteristic quadruplet
 $(\theta,\sigma^2,\rho,\kappa)$ of $M$: A
 $\sigma(\CE)$-$\CB(\R)$-measurable function $f\colon X\to\R$ belongs to
 $\CL^1(M)$, if and only if the following three conditions are
 satisfied:

\begin{enumerate}[a]

\item $\int_X\big|
  f(x)\theta(x)+\int_{\R}\big(\varsigma(f(x)t)-f(x)\varsigma(t)\big)\,\rho(x,\d
  t)\big|\,\kappa(\d x)<\infty$.

\item $\int_Xf(x)^2\sigma^2(x)\,\kappa(\d x)<\infty$.

\item $\int_X\big(\int_{\R}\min\{1,f(x)^2t^2\}\,\rho(x,\d
  t)\big)\,\kappa(\d x)<\infty$.

\end{enumerate}

In the affirmative case it follows further from
\ref{rem:properties_of_integral(2)} and \cite[Theorem~2.7]{RR} that
the free characteristic triplet $(a_f,\sigma_f^2,F_f)$ of $\int_Xf\6M$ is given by

\begin{enumerate}[resume*]

\item\label{rem:properties_of_integral(4)_(a)} 
$a_f=\int_X\big(f(x)\theta(x)+\int_{\R}\big(\varsigma(f(x)t)-f(x)\varsigma(t)\big)\,\rho(x,\d
  t)\big)\,\kappa(\d x)$.

\item\label{rem:properties_of_integral(4)_(b)} 
$\sigma_f^2=\int_Xf(x)^2\sigma^2(x)\,\kappa(\d x)$.

\item\label{rem:properties_of_integral(4)_(c)} 
$F_f(B)=F\big(\{(x,t)\in X\times\R\mid f(x)t\in B\setminus\{0\}\}\big)$ 
for any Borel set $B$ in $\R$.

\end{enumerate}

In \ref{rem:properties_of_integral(4)_(c)} $F$ is the measure on
$\sigma(\CE)\otimes\CB(\R)$ described in
Proposition~\ref{prop:free_RR_1} (or equivalently given by
\eqref{integralformel_2}). 
For $f$ in $\CL^1(M)$ the measure $F_f$ is a L\'evy measure on $\R$, and by
e.g.\ an extension argument it follows that
\begin{equation}
\int_{\R}g(t)\, F_f(\d t)
=\int_{X\times\R}g(f(x)t)\cdot 1_{\R\setminus\{0\}}(f(x)t)\, F(\d x,\d
  t),
\label{integralformel_3}
\end{equation}
for any function $g$ in $\CL^1(F_f)$.

\end{lenumerate}
\end{remarks}

Knowing the free characteristic triplet of $\int_Xf\6M$ (as described in
Remark~\ref{rem:properties_of_integral}\ref{rem:properties_of_integral(4)}),
we can easily derive the following analog of Proposition~2.6 in
\cite{RR}, which generalizes Proposition~\ref{prop:free_RR_2} (in the present
paper) from indicator functions to general functions in
$\CL^1(M)$.

\begin{corollary}\label{cor:free_RR_3}
Let $M=\{M(E)\mid E\in\CE\}$ be a free L\'evy basis with free
characteristic quadruplet $(\theta,\sigma^2,\rho,\kappa)$, and let $f$ be
a function from $\CL^1(M)$. Consider further the kernel $R\colon X\times\C^-\to\C$
set out in Proposition~\ref{prop:free_RR_2}.
Then the function $x\mapsto R(x,zf(x))$ is in $\CL^1(\kappa)$ for all
$z$ in $\C^-$, and the free cumulant transform of 
$\int_Xf\6M$ is given by
\begin{equation}
\CC_{\int_Xf\6M}(z)=\int_{X}R(x,f(x)z)\,\kappa(\d x), \qquad(z\in\C^-).
\label{eq_cor:free_RR_3}
\end{equation}
\end{corollary}

Recall that for fixed $x$ in $X$ the function $R(x,\cdot)$ is
the free cumulant transform of the L\'evy seed $\nu_x$ at $x$. Thus, 
from the infinitesimal point of view, \eqref{eq_cor:free_RR_3} shows
that the distribution of the integral $\int_Xf\6M$ is obtained by
scaling $\nu_x$ by $f(x)$ at each $x$,
followed by an averaging with respect to the control measure $\kappa$.

\begin{proofof}[Proof of Corollary~\ref{cor:free_RR_3}.]
Consider the free characteristic triplet $(a_f,\sigma_f^2,F_f)$ for
$\int_Xf\6M$ (given in
Remark~\ref{rem:properties_of_integral}\ref{rem:properties_of_integral(4)}).
For $z$ in $\C^-$ it follows then by
\eqref{integralformel_3} and \eqref{integralformel_2} that
\begin{equation*}
\begin{split}
\int_{\R}\Big(\frac{1}{1-tz}-1-z\varsigma(t)\Big) F_f(\d t)
&=\int_{X\times\R}\Big(\frac{1}{1-tzf(x)}-1-z\varsigma(tf(x))\Big) F(\d
x,\d t)
\\
&=\int_{X}\Big(\int_{\R}\Big(\frac{1}{1-tzf(x)}-1-z\varsigma(tf(x))\Big)\rho(x,\d
t)\Big)\kappa(\d x),
\end{split}
\end{equation*}
and consequently
\begin{equation*}
\begin{split}
&\CC_{\int_Xf\6M}(z)=za_f+z^2\sigma_f^2
+\int_{\R}\Big(\frac{1}{1-tz}-1-z\varsigma(t)\Big) F_f(\d t)
\\
&=\int_Xz\Big(f(x)\theta(x)+\int_{\R}\big(\varsigma(f(x)t)-f(x)\varsigma(t)\big)\rho(x,\d
  t)\Big)\kappa(\d x)
\\
&\phantom{ \ = \ }
+z^2\int_Xf(x)^2\sigma^2(x)\,\kappa(\d x)
+\int_{X}\Big(\int_{\R}\Big(\frac{1}{1-tzf(x)}-1-z\varsigma(tf(x))\Big)\rho(x,\d
t)\Big)\kappa(\d x),
\\
&=\int_X\Big(zf(x)\theta(x)+z^2f(x)^2\sigma^2(x)
+\int_{\R}\Big(\frac{1}{1-tzf(x)}-1-zf(x)\varsigma(t)\Big)\rho(x,\d
t)\Big)\kappa(\d x),
\\
&=\int_X R(zf(x),x)\,\kappa(\d x),
\end{split}
\end{equation*}
as desired.
\end{proofof}

\begin{proposition}\label{prop:FLB_with_density}
Let $M=\{M(E)\mid E\in\CE\}$ be a
free L\'evy basis with free characteristic triplet $(\Theta,\Sigma,F)$
and quadruplet
$(\theta,\sigma^2,\rho(\cdot,\d t),\kappa)$. 
Let further $f\colon X\to\R$ be a
$\sigma(\CE)$-measurable function, and define 
\[
\CE(f)=\{E\in\sigma(\CE)\mid f1_E\in\CL^1(M)\},
\]
and
\[
f\cdot M(E)=\int_Xf1_E\6M, \qquad(E\in\CE(f)).
\]
Then the following statements hold:

\begin{enumerate}[i]

\item
The family $\CE(f)$ is $\delta$-ring and $f\cdot M=\{f\cdot M(E)\mid E\in\CE(f)\}$ is
free L\'evy basis.

\item Suppose $\CE(f)$ satisfies \eqref{eq_determining_sequence}. Then
the free characteristic triplet $(\Theta^f,\Sigma^f,F^f)$ for $f\cdot M$ is given by

\begin{enumerate}[a]

\item
$\Theta^f(E)=\int_E\big(f(x)\theta(x)+\int_{\R}
\big(\varsigma(f(x)t)-f(x)\varsigma(t)\big)\,\rho(x,\d t)\big)\,\kappa(\d x)$,

\item
$\Sigma^f(E)=\int_Ef(x)^2\sigma^2(x)\,\kappa(\d x)$,

\item
$\displaystyle{\int_{X\times\R}g(x,t)\,F^f(\d x,\d t)=
\int_{X\times\R}g(x,tf(x))1_{\R\setminus\{0\}}(tf(x))\, F(\d x,\d t)}$,

\end{enumerate}
\noindent
where (c) holds for any $\sigma(\CE(f))\otimes\CB(\R)$-measurable function $g\colon
X\times\R\to[0,\infty)$.

\end{enumerate}
\end{proposition}

\begin{proof}
\begin{lenumerate}[i]

\item
According to
Remark~\ref{rem:properties_of_integral}\ref{rem:properties_of_integral(4)}
a set $E$ from $\sigma(\CE)$ belongs to $\CE(f)$, if and only if the following
three conditions are satisfied:

\begin{itemize}

\item $\int_E\big|
  f(x)\theta(x)+\int_{\R}\big(\varsigma(f(x)t)-f(x)\varsigma(t)\big)\,\rho(x,\d
  t)\big|\,\kappa(\d x)<\infty$.

\item
$\int_Ef(x)^2\sigma^2(x)\,\kappa(\d x)<\infty$.

\item
$\int_E\big(\int_{\R}\min\{1,f(x)^2t^2\}\,\rho(x,\d
  t)\big)\,\kappa(\d x)<\infty$.

\end{itemize}

\noindent
In particular it is apparent that $\CE(f)$ is a $\delta$-ring with the
hereditary property that the conditions $E_1\in\sigma(\CE)$, $E_2\in\CE(f)$ and
$E_1\subseteq E_2$ imply that $E_1\in\CE(f)$. 

The fact that $f\cdot M$ satisfies conditions
\ref{def_FLB(b)},\ref{def_FLB(c)} and \ref{def_FLB(a)} in Definition~\ref{def_FLB}
follows from, respectively, \ref{rem:properties_of_integral(3)},
\ref{rem:properties_of_integral(1)} and
\ref{rem:properties_of_integral(2)} of
Remark~\ref{rem:properties_of_integral}. To verify
condition~\ref{def_FLB(d)} in Definition~\ref{def_FLB} consider a
decreasing sequence $(E_n)_{n\in\N}$ of sets from $\CE(f)$, such that
$E_n\downarrow\emptyset$ as $n\to\infty$. It follows then from
Corollary~\ref{cor:free_RR_3} that for any $z$ in $\C^-$ the function $x\mapsto
R(x,zf(x)1_{E_1}(x))=R(x,zf(x))1_{E_1}(x)$ is in $\CL^1(\kappa)$, and
hence by dominated convergence
\[
\CC_{\int_Xf1_{E_n}\6M}(z)=\int_{E_n}R(x,zf(x))1_{E_1}(x)\,\kappa(\d x)
\xrightarrow[n\to\infty]{}0,
\]
for any $z$ in $\C^-$. By
Proposition~\ref{prop:svag_konv_vha_fri_kum_tranf}
this implies that 
$\spL\{\int_Xf1_{E_n}\6M\}\overset{\text w}{\to}\delta_0$, i.e.\ 
$\int_Xf1_{E_n}\6M\to0$ in the measure topology (cf.\ \eqref{prel_eq2}).

\item
The formulae for $\Theta^f$ and $\Sigma^f$ follow readily from
Remark~\ref{rem:properties_of_integral}\ref{rem:properties_of_integral(4)},
which also yields that
\[
F^f(E\times B)=F\big(\{(x,t)\in X\times\R\mid
1_E(x)f(x)t\in B\setminus\{0\}\}\big)
\] 
for any $E$ in $\CE(f)$ and any Borel set $B$ in $\R$. From this
condition the measure $F^f$ on
$\sigma(\CE(f))\otimes\CB(\R)$ may be identified as the concentration
to $X\times(\R\setminus\{0\})$ of
the transformation of $F$ by the mapping $\Psi\colon X\times\R\to
X\times\R$ given by 
$
\Psi(x,t)=(x,tf(x)).
$
Specifically this means that
\begin{equation*}
F^f(C)=F\big(\Psi^{-1}(C\cap(X\times(\R\setminus\{0\})))\big)
=\int_{X\times\R}1_C(x,tf(x))1_{\R\setminus\{0\}}(tf(x))\,F(\d x,\d t),
\end{equation*}
for any set $C$ in $\sigma(\CE(f))\otimes\CB(\R)$, which is in
accordance with the formula for $F^f(E\times B)=F^f_E(B)$
given above. By e.g.\ an extension argument it follows further that
\begin{equation}
\int_{X\times\R}g(x,t)\,F^f(\d x,\d t)=
\int_{X\times\R}g(x,tf(x))1_{\R\setminus\{0\}}(tf(x))\, F(\d x,\d t)
\label{FLB_with_density_eq1}
\end{equation}
for any positive and measurable function $g$.
\end{lenumerate}
\end{proof}

\section{The L\'evy-It\^o decomposition for free L\'evy bases}
\label{sec:Levy-Ito_for_FLB}

Throughout this section we consider a non-empty set $X$ equipped
with a $\delta$-ring $\CE$, which contains a sequence
$(E_n)_{n\in\N}$ of sets such that $X=\medcup_{n\in\N}E_n$. We
consider further a free L\'evy basis $M=\{M(E)\mid E\in\CE\}$
with free characteristic triplet $(\Theta,\Sigma,F)$ (in particular
$F$ is a $\sigma$-finite measure on $\sigma(\CE)\otimes\CB(\R)$).
Our objective is to establish a L\'evy-It\^o type representation of $M$. For
this we introduce first a free Poisson random measure
$P_F=\{P_F(A)\mid A\in\CA_F\}$ on $X\times\R$ affiliated with some $W^*$-probability
space $(\CM,\tau)$, and where
\[
\CA_F=\{A\in\sigma(\CE)\otimes\CB(\R)\mid F(A)<\infty\},
\]
and
\[
\spL\{P_F(A)\}=\textrm{Poiss}^{\boxplus}(F(A))
\quad\text{for all $A$ in $\CA_F$}.
\]
We note in particular that
$\sigma(\CA_F)=\sigma(\CE)\otimes\CB(\R)$, since $F$ is
$\sigma$-finite. We recall also (cf.\
Example~\ref{ex:Factorizable_FLB}\ref{ex:Factorizable_FLB(3)})  
that the free characteristic triplet and quadruplet of $P_F$ are
$(F,0,F\otimes\delta_1)$ and $(\frac{1}{2},0,\frac{1}{2}\delta_1,2F)$,
respectively.

Consider now additionally the function $h\colon X\times\R\to\R$ given by
\[
h(x,t)=t, \qquad(x\in X, \ t\in\R).
\]
By application of Proposition~\ref{prop:FLB_with_density} 
we may then consider yet
another free L\'evy basis $h\cdot P_F=\{h\cdot P_F(A)\mid
A\in\CA_F(h)\}$ on $X\times\R$, where
\[
\CA_F(h)=\{A\in\sigma(\CE)\otimes\CB(\R)\mid h1_A\in\CL^1(P_F)\},
\]
and
\[
h\cdot P_F(A)=\int_{X\times\R}1_Ah \6P_F=\int_{X\times\R}1_A(x,t)t\,
P_F(\d x,\d t), \qquad(A\in\CA(h)).
\]

\begin{lemma}\label{lemma:Prep_for_Levy_Ito}
In the framework set up above the following assertions hold:

\begin{enumerate}[i]

\item\label{lemma:Prep_for_Levy_Ito(ii)}
For any positive number $\epsilon$ we have that
\[
\{E\times(\R\setminus[-\epsilon,\epsilon])\mid E\in\CE\}\subseteq\CA_F(h).
\]
In particular $\CA_F(h)$ satisfies \eqref{eq_determining_sequence}.

\item\label{lemma:Prep_for_Levy_Ito(i)}
The free characteristic triplet for ${h\cdot P_F}$ is given by
  $(F^h,0,(F\otimes\delta_1)^h)$, where
\begin{align*}
F^h(A)&=\int_A\varsigma(t)\, F(\d x, \d t), \qquad(A\in\CA_F(h)),
\\
(F\otimes\delta_1)^h(A\times B)&=F(A\cap (X\times(B\setminus\{0\}))),
\qquad(A\in\sigma(\CE)\otimes\CB(\R), \ B\in\CB(\R)).
\end{align*}

\item\label{lemma:Prep_for_Levy_Ito(iii)}
If $E\in\CE$ such that $\int_{-1}^1 |t|\, F_E(\d t)<\infty$,
  then $E\times\R\in\CA_F(h)$.

\end{enumerate}
\end{lemma}

\begin{proof} 
\begin{lenumerate}[i]

\item Let $E$ from $\CE$ be given. Clearly
  $E\times(\R\setminus[-\epsilon,\epsilon])
\in\sigma(\CE)\otimes\CB(\R)=\sigma(\CA_F)$. It remains thus to
verify that
$h1_{E\times(\R\setminus[-\epsilon,\epsilon])}\in\CL^1(P_F)$. Recalling
from Example~\ref{ex:Factorizable_FLB}\ref{ex:Factorizable_FLB(3)}
that the free characteristic quadruplet for $P_F$ is
$(\frac{1}{2},0,\frac{1}{2}\delta_1,2F)$,
it follows by a straightforward application of
Remark~\ref{rem:properties_of_integral}\ref{rem:properties_of_integral(4)} 
that this amounts to the conditions:
\[
\infty>\int_{E\times(\R\setminus[-\epsilon,\epsilon])}\big|\varsigma(h(x,t))\big|\, F(\d x, \d t)
=\int_{\R\setminus[-\epsilon,\epsilon]}|\varsigma(t)|\,F_E(\d t),
\]
and
\[
\infty>\int_{E\times(\R\setminus[-\epsilon,\epsilon])}
\min\{1,h(x,t)^2\}\,F(\d x,\d t)
=\int_{\R\setminus[-\epsilon,\epsilon]}\min\{1,t^2\}\,F_E(\d t).
\]
Both conditions are satisfied, since $F_E$ is a L\'evy measure.

\item
Since the free characteristic quadruplet for ${P_F}$ is
$(\frac{1}{2},0,\frac{1}{2}\delta_1,2F)$
it follows readily from
Proposition~\ref{prop:FLB_with_density} that the free characteristic
triplet for ${h\cdot P_F}$ is given as $(F^h,0,(F\times\delta_1)^h)$, where
\begin{equation*}
\begin{split}
F^h(A)&=\int_A\Big(\tfrac{1}{2}h(x,t)+\int_{\R}\big(\varsigma(h(x,t)s)
-h(x,t)\varsigma(s)\big)\,\tfrac{1}{2}\delta_1(\d s)\Big)\,2F(\d x, \d t)
\\
&=\int_A(t+\varsigma(t)-t)\, F(\d x,\d t)
=\int_A\varsigma(t)\, F(\d x,\d t)
\end{split} 
\end{equation*}
for any $A$ in $\CA_F(h)$. For $A$ in
$\sigma(\CE)\otimes\CB(\R)$ and $B$ in $\CB(\R)$ 
Proposition~\ref{prop:FLB_with_density} further yields that
\begin{equation*}
\begin{split}
(F\otimes\delta_1)^h(A\times B)
&=\int_{X\times\R\times\R}1_A(x,t)1_B(sh(x,t))1_{\R\setminus\{0\}}(sh(x,t))\,
F\otimes\delta_1(\d x, \d t, \d s)
\\
&=\int_{X\times\R}1_A(x,t)1_{B\setminus\{0\}}(t) F(\d x, \d t)
=F(A\cap(X\times(B\setminus\{0\}))).
\end{split}
\end{equation*}

\item Assume that $E\in\CE$ such that $\int_{-1}^1|t|\,F_E(\d t)<\infty$.
We must verify that $h1_{E\times\R}\in\CL^1(P_F)$, and as in the proof
of (ii) this amounts to the conditions:
\[
\infty>\int_{\R}|\varsigma(t)|\,F_E(\d t)
=F_E(\R\setminus[-1,1])+\int_{-1}^1|t|\,F_E(\d t),
\]
and
\[
\infty>\int_{\R}\min\{1,t^2\}\,F_E(\d t),
\]
which are clearly satisfied by the assumption on $E$ and
since $F_E$ is a L\'evy measure.
\end{lenumerate}
\end{proof}

\begin{proposition}\label{Free_Levy_Ito_I}
Consider the framework set up in the beginning of this section, and
assume that $\int_{-1}^1|\varsigma(t)|\,F_E(\d t)<\infty$ for all $E$ in $\CE$.
Assume further that there exists a free semi-circular L\'evy basis 
$G_{\Sigma}=\{G_{\Sigma}(E)\mid E\in\CE\}$ in $(\CM,\tau)$ with free characteristic triplet $(0,\Sigma,0)$,
which is freely independent of $P_F$. 
For each $E$ in $\CE$ put
\begin{align}
\tilde{M}(E)=
\Big(\Theta(E)-\int_{E\times\R}\varsigma(t)\, F(\d x,\d
              t)\Big)\unit_{\CM}+G_{\Sigma}(E)+{h\cdot P_F}(E\times\R)
\label{Levy_Ito_eq1}
\end{align}
where $\unit_{\CM}$ denotes the unit of $\CM$.

Then $\tilde{M}$ is a free L\'evy basis, and for any $E$ in $\CE$
the selfadjoint operators $M(E)$ and $\tilde{M}(E)$ share the same
spectral distribution.
\end{proposition}

Before the proof we note that the assumed existence of $G_{\Sigma}$ which is
freely independent of $P_F$ can always be realized by replacing
$(\CM,\tau)$ by its free product with another $W^*$-probability space
$(\CM',\tau')$ which contains a free semicircular basis with the
specified characteristic triplet. In comparison with the classical
L\'evy-It\^o Decomposition we note also that
\[
{h\cdot P_F}(E\times\R)=\int_{E\times\R}t\, P_F(\d x,\d t),
\]
by definition of ${h\cdot P_F}$.

\begin{proof} For each $E$ in $\CE$ denote by $M_1(E), M_2(E)$ and
  $M_3(E)$ the three terms on the right hand side of
  \eqref{Levy_Ito_eq1} (in order of appearance). Note in particular
  that $M_3(E)$ is well-defined according to
  Lemma~\ref{lemma:Prep_for_Levy_Ito}\ref{lemma:Prep_for_Levy_Ito(iii)}.
  Clearly $M_1$ and
  $M_2$ are free L\'evy bases on $(X,\CE)$ with free characteristic
  triplets $(\Pi,0,0)$ and $(0,\Sigma,0)$, respectively, where we
  have introduced the signed measure
\[
\Pi(E)=\Theta(E)-\int_{E\times\R}\varsigma(t)\,F(\d x,\d t),
\qquad(E\in\CE).
\]
Considering the mapping $\psi\colon X\times\R\to X$ given by
$\psi(x,t)=x$ for all $(x,t)$ in $X\times\R$, we note next that
$M_3(E)={h\cdot P_F}(\psi^{-1}(E))$ for all $E$ in $\CE$, and therefore
Proposition~\ref{prop:transformation_of_FLB} 
yields that $M_3$ is a free L\'evy basis with free characteristic
triplet
$(F^h\circ\psi^{-1},0,(F\otimes\delta_1)^h\circ(\psi,\id_{\R})^{-1})$,
where $F^h$ and $(F\otimes\delta_1)^h$ are as set out in
Lemma~\ref{lemma:Prep_for_Levy_Ito}\ref{lemma:Prep_for_Levy_Ito(i)}.

Since $M_1,M_2,M_3$ are  freely independent, their sum,
$\tilde{M}$, is again a free L\'evy basis with free characteristic
triplet
$(\Pi+F^h\circ\psi^{-1},\Sigma,(F\otimes\delta_1)^h\circ(\psi,\id_{\R})^{-1})$.
For any set $E$ from $\CE$ note here that by
Lemma~\ref{lemma:Prep_for_Levy_Ito}\ref{lemma:Prep_for_Levy_Ito(i)}
\[
\Pi(E)+F^h\circ\psi^{-1}(E)
=\Theta(E)-\int_{E\times\R}\varsigma(t)\,F(\d x,\d t)+F^h(E\times\R)
=\Theta(E),
\]
and for any Borel subset $B$ of $\R$
\begin{equation*}
\begin{split}
(F\otimes\delta_1)^h&\circ(\psi,\id_{\R})^{-1}(E\times B)
=(F\otimes\delta_1)^h(E\times\R\times B)
=F((E\times\R)\cap(X\times(B\setminus\{0\})))
\\
&=F(E\times(B\setminus\{0\}))
=F_E(B\setminus\{0\})
=F_E(B)
=F(E\times B),
\end{split}
\end{equation*}
where we have used that $F_E(\{0\})=0$.
Since the last equation above determines $F$ uniquely on
$\sigma(\CE)\otimes\CB(\R)$, we conclude
that $(F\otimes\delta_1)^h\circ(\psi,\id_{\R})^{-1}=F$. Altogether
$\tilde{M}$ has the same free characteristic triplet as $M$, which
clearly implies that $\spL(\tilde{M}(E))=\spL(M(E))$ for all $E$ in $\CE$.
\end{proof}

We return now to the general framework set up in the beginning of this
section without imposing the condition
$\int_{\R}|\varsigma(t)|\,F_E(\d t)<\infty$.
For each $\epsilon$ in $(0,\infty)$ we then define
$M^{(\epsilon)}=\{M^{(\epsilon)}(E)\mid E\in\CE\}$ by
\begin{equation}
M^{(\epsilon)}(E)={h\cdot P_F}(E\times(\R\setminus[-\epsilon,\epsilon]))
-\Big(\int_{\R\setminus[-\epsilon,\epsilon]}\varsigma(t)\, F_E(\d
t)\Big)\unit_{\CM}, 
\qquad(E\in\CE).
\label{eg_def_M-epsilon}
\end{equation}
Lemma~\ref{lemma:Prep_for_Levy_Ito}\ref{lemma:Prep_for_Levy_Ito(ii)} 
guarantees that $M^{(\epsilon)}$ is well-defined. 
With $\psi$ as in the proof of Proposition~\ref{Free_Levy_Ito_I} note that 
$h\cdot P_f(E\times(\R\setminus[-\epsilon,\epsilon]))=h\cdot
P_f(X\times(\R\setminus[-\epsilon,\epsilon])\cap\psi^{-1}(E))$.
Hence
Propositions~\ref{prop:transformation_of_FLB}-\ref{prop:concentration_of_FLB} 
in combination with 
Lemma~\ref{lemma:Prep_for_Levy_Ito}\ref{lemma:Prep_for_Levy_Ito(i)}
entail that $M^{(\epsilon)}$ is a free L\'evy basis with free characteristic
triplet $(0,0,F^{(\epsilon)})$, where
\begin{equation}
\begin{split}
F^{(\epsilon)}(E\times B)
&=(F\otimes\delta_1)^h\big((X\times(\R\setminus[-\epsilon,\epsilon])\times\R)\cap
((\psi,\id_{\R})^{-1}(E\times B))\big)
\\
&=F\big((X\times(\R\setminus[-\epsilon,\epsilon]))\cap(E\times\R)\cap(X\times
(B\setminus\{0\}))\big)
\\
&=F\big(E\times(B\setminus[-\epsilon,\epsilon])\big)
\\
&=F_E(B\setminus[-\epsilon,\epsilon])
\end{split}
\label{Levy-Ito_eqI}
\end{equation}
for any $E$ in $\CE$ and $B$ in $\CB(\R)$.

\begin{lemma}\label{lemma:PrepII_for_Levy_Ito}
\begin{enumerate}[i]

\item With $M^{(\epsilon)}$ as defined in \eqref{eg_def_M-epsilon},
the sequence $(M^{(1/n)}(E))_{n\in\N}$ is
convergent in the measure topology for any $E$ in $\CE$.

\item If we define (limit in the measure topology)
\[
M_4(E)=\lim_{n\to\infty}M^{(1/n)}(E), \qquad(E\in\CE),
\]
then $M_4=\{M_4(E)\mid E\in\CE\}$ is a free L\'evy basis with free
characteristic triplet $(0,0,F)$.

\end{enumerate}
\end{lemma}

\begin{proof}
\begin{lenumerate}[i]

\item
Let $E$ from $\CE$ be given.
Since the measure topology is complete, 
it suffices to show that $(M^{(1/n)}(E))_{n\in\N}$ is a Cauchy
sequence in the measure topology, i.e.\ that
$\spL\{M^{(1/n)}(E)-M^{(1/m)}(E)\}\overset{\rm w}{\to}\delta_0$ as
$n,m\to\infty$. Establishing this condition amounts to verifying that
\[
\spL\{M^{(1/n_k)}(E)-M^{(1/m_k)}(E)\}\xrightarrow[]{\text w}\delta_0
\quad\text{as $k\to\infty$}  
\]
for any sequence $(m_k,n_k)_{k\in\N}$ in $\N\times\N$,
such that $m_k\le n_k$ for all $k$, and such that $m_k\to\infty$ as
$k\to\infty$. Given such a sequence $(m_k,n_k)_{k\in\N}$, note first that
 \begin{align*}
M^{(1/n_k)}(E)&-M^{(1/m_k)}(E)
\\
&=h\cdot P_F
\big(E\times([-\tfrac{1}{m_k},-\tfrac{1}{n_k})\cup(\tfrac{1}{n_k},\tfrac{1}{m_k}])\big)
-\Big(\int_{[-\frac{1}{m_k},-\frac{1}{n_k})\cup(\frac{1}{n_k},\frac{1}{m_k}]}
\varsigma(t)\, F_E(\d t)\Big)\unit_{\CM},
\end{align*} 
and hence by Lemma~\ref{lemma:Prep_for_Levy_Ito} 
the free characteristic triplet for $M^{(1/n_k)}(E)-M^{(1/m_k)}(E)$ is
$(0,0,\varrho_k)$, where
\[
\varrho_k(B)=F_E\big(B\cap([-\tfrac{1}{m_k},-\tfrac{1}{n_k})
\cup(\tfrac{1}{n_k},\tfrac{1}{m_k}])\big)
\]
for any Borel set $B$ in $\R$.
It follows from Theorem~\ref{thm:konv_vha_kar_triplet} 
that it suffices to show that

\begin{enumerate}[1]

\item\label{PrepII_for_Levy_Ito:item(1)}
$\int_{\R}f(t)\,\varrho_k(\d t)\to0$ as $k\to\infty$ for any
  continuous bounded function $f\colon\R\to\R$, which vanishes in a
  neighborhood of 0.

\item\label{PrepII_for_Levy_Ito:item(2)}
  $\displaystyle{\lim_{\epsilon\downarrow0}\Big(\limsup_{k\to\infty}
\int_{-\epsilon}^{\epsilon}t^2\,\varrho_k(\d t)\Big)=0}$. 

\end{enumerate}

\noindent
Since $\int_{\R}f(t)\,\varrho_k(\d t)
=\int_{[-\frac{1}{m_k},-\frac{1}{n_k})\cup(\frac{1}{n_k},\frac{1}{n_k}]}f(t)
\, F_E(\d t)$,
condition \ref{PrepII_for_Levy_Ito:item(1)} follows
from the fact that for all sufficiently large $k$, the set
$[-\frac{1}{m_k},-\frac{1}{n_k})\cup(\frac{1}{n_k},\frac{1}{m_k}]$ 
is contained in the neighborhood of 0 upon which $f$ vanishes.
Condition \ref{PrepII_for_Levy_Ito:item(2)}, in turn, follows e.g.\ from the
fact that $\int_{-\epsilon}^{\epsilon}t^2\,\varrho_k(\d t)\le
\int_{-\epsilon}^{\epsilon}t^2\,F_E(\d t)$ for any $k$ in $\N$, and
here $\int_{-\epsilon}^{\epsilon}t^2\,F_E(\d t)\to0$ as
$\epsilon\downarrow0$, since $F_E$ is a L\'evy measure.

\item It follows immediately from the definition of $M_4$ that
  $\spL\{M_4(E)\}\in\ID(\boxplus)$ for all $E$ in $\CE$, since $\ID(\boxplus)$ is closed in the
  topology for weak convergence and by use of \eqref{prel_eq2}.
For each $n$ in $\N$ the operator $M^{(1/n)}(E)$ has free characteristic triplet
  $(0,0,F^{(1/n)})$, where $F^{(1/n)}$ is given by
  \eqref{Levy-Ito_eqI}. 
We check next that $M_4(E)$ has free characteristic triplet
$(0,0,F_E)$ for any $E$ in $\CE$. 
By another application of Theorem~\ref{thm:konv_vha_kar_triplet} 
this is a consequence of the following two facts:

\begin{enumerate}[1]

\item For any continuous bounded function $f\colon\R\to\R$ vanishing
  on a neighborhood of 0, it holds that
\[
\int_{\R}f \6F^{(1/n)}
=\int_{\R}f1_{[-\frac{1}{n},\frac{1}{n}]^c}\6F_E
\xrightarrow[n\to\infty]{}\int_{\R}f\6F_E,
\]
by Dominated Convergence, since $\int_{\R}|f|\6F_E<\infty$, because
$F_E$ is a L\'evy measure.

\item For any positive number $\epsilon$ we have that
\[
\int_{-\epsilon}^\epsilon t^2 \,F^{(1/n)}(\d t)
=\int_{-\epsilon}^\epsilon t^21_{[-\frac{1}{n},\frac{1}{n}]^c}(t)\, F_E(\d t)
\xrightarrow[n\to\infty]{}\int_{-\epsilon}^{\epsilon}t^2\,F_E(\d t),
\]
by Monotone Convergence. Hence
\[
\lim_{\epsilon\downarrow0}
\Big(\limsup_{n\to\infty}\int_{-\epsilon}^{\epsilon}t^2\,F^{(1/n)}(\d t)\Big)
=\lim_{\epsilon\downarrow0}\Big(\int_{-\epsilon}^{\epsilon}t^2\, F_E(\d
t)\Big)=0,
\]
since $\int_{-1}^{1}t^2\, F_E(\d t)<\infty$.
\end{enumerate}
\noindent
It remains to show that $M_4$ is a free L\'evy basis, i.e.\ to verify the four
conditions in Definition~\ref{def_FLB}:

\begin{enumerate}[a]

\item We already noted that $\spL\{M_4(E)\}\in\ID(\boxplus)$ for all
  $E$ in $\CE$.

\item If $E_1,\ldots,E_n$ are disjoint sets from $\CE$, then
  $M_4(E_1),\ldots,M_4(E_n)_{n\in\N}$ are freely independent. This follows from
  the definition of $M_4$ and the corresponding property for
  $M^{(1/n)}$, since free independence is preserved under
  limits in the measure topology (see Proposition~5.4 in \cite{BNT05}).

\item  If $E_1,\ldots,E_n$ are disjoint sets from $\CE$, then
  $M_4(\medcup_{k=1}^nE_k)=\sum_{k=1}^nM_4(E_k)$. Again this follows
  from the definition of $M_4$ and the corresponding property of $M^{(1/n)}$,
  since addition is continuous in the measure topology.

\item Let $(E_n)_{n\in\N}$ be a decreasing sequence of sets from $\CE$, such
  that $\medcap_{n\in\N}E_n=\emptyset$. Then
  $\spL\{M_4(E_n)\}\overset{\text w}{\to}\delta_0$. 
Indeed, the free characteristic triplet of $M_4(E_n)$ is
$(0,0,F_{E_n})$,
and hence by Theorem~\ref{thm:konv_vha_kar_triplet} it suffices to check
the following two conditions:

\begin{enumerate}[1]

\item $\int_{\R}f\6F_{E_n}\to0$ as $n\to\infty$ for any continuous
  bounded function $f\colon\R\to\R$ vanishing in a neighborhood, say
  $[-\epsilon,\epsilon]$, of 0. To see this, note that
\[
\Big|\int_{\R}f\6F_{E_n}\Big|\le\|f\|_{\infty}F_{E_n}([-\epsilon,\epsilon]^c)
=\|f\|_{\infty}F(E_n\times([-\epsilon,\epsilon]^c))\xrightarrow[n\to\infty]{}0,
\]
since
$F(E_1\times([-\epsilon,\epsilon]^c))=F_{E_1}([-\epsilon,\epsilon]^c)<\infty$,
and $\medcap_{n\in\N}E_n\times([-\epsilon,\epsilon]^c)=\emptyset$.

\item
  $\displaystyle{\lim_{\epsilon\downarrow0}\Big(\limsup_{n\to\infty}
\int_{-\epsilon}^{\epsilon}t^2\,F_{E_n}(\d t)\Big)=0}$. 
To see this, note for any fixed positive $\epsilon$ that
\[
\int_{-\epsilon}^{\epsilon}t^2\,F_{E_n}(\d t)
=\int_{E_n\times[-\epsilon,\epsilon]}t^2\,F(\d x,\d t)
\xrightarrow[n\to\infty]{}0,
\]
because $\int_{E_1\times[-\epsilon,\epsilon]}t^2\,F(\d x, \d t)
=\int_{-\epsilon}^{\epsilon}t^2\,F_{E_1}(\d t)<\infty$, and
$\medcap_{n\in\N}E_n\times[-\epsilon,\epsilon]=\emptyset$.

\end{enumerate}
\end{enumerate}
This completes the proof.
\end{lenumerate}
\end{proof}

As an immediate consequence of Lemma~\ref{lemma:PrepII_for_Levy_Ito}
we obtain the following general version of the L\'evy-It\^o-decomposition
for free L\'evy bases. Note that this decomposition coincides with
that of Proposition~\ref{Free_Levy_Ito_I} under the extra assumption
in that proposition.

\begin{corollary}\label{Free_Levy_Ito_II}
Consider the framework set up in the beginning of this section, and
assume further that there exists a semi-circular free L\'evy basis 
$G_{\Sigma}=\{G_{\Sigma}(E)\mid E\in\CE\}$ in $(\CM,\tau)$ with free characteristic triplet $(0,\Sigma,0)$,
which is freely independent of $P_F$.

For each $E$ in $\CE$ put
\begin{equation*}
\tilde{M}(E)=\Theta(E)\unit_{\CM}+G_{\Sigma}(E)+M_4(E),
\end{equation*}
where $\unit_{\CM}$ denotes the unit of $\CM$ and $M_4$ is introduced
in Proposition~\ref{Free_Levy_Ito_I}. 

Then $\tilde{M}$ is a free L\'evy basis, and for any $E$ in $\CE$
the selfadjoint operators $M(E)$ and $\tilde{M}(E)$ share the same
spectral distribution.
\end{corollary}

\section{Proof of Theorem~\ref{existence_FLB}}
\label{sec:Proof_Existence_FLB}

Throughout this section we consider a nonempty set $X$ and a ring
$\CE$ of subsets of $X$. We then put
\begin{equation*}
\II=\medcup_{k\in\N}\big\{\{E_1,\ldots,E_k\}\bigm|
E_1,\ldots,E_k\in\CE\setminus\{\emptyset\} \ \text{and
  $E_1,\ldots,E_k$ are disjoint}\big\}.
\end{equation*}
We emphasize that we consider an element of $\II$ merely as a
collection of sets, without paying attention to the order in which
these sets appear. Thus we identify an element
$\{E_1,\ldots,E_k\}$ from $\II$ with the element
$\{E_{\pi(1)},\ldots,E_{\pi(k)}\}$ for any permutation $\pi$ of $\{1,\ldots,k\}$.

We equip $\II$ with a partial order ``$\le$'' by declaring that
$\{E_1,\ldots,E_k\}\le \{F_1,\ldots, F_m\}$ exactly when each $E_i$ is a
union of some of the $F_j$'s. We note then that ``$\le$'' is an
upward-filtering order, since for $S=\{E_1,\ldots,E_k\}$ and
$T=\{F_1,\ldots,F_m\}$ from $\II$ we have that $S,T\le U$, where $U$ is
the element of $\II$ consisting of all non-empty sets in the following
family:
\begin{equation*}
E_i\cap F_j, \quad E_i\setminus(\medcup_{j=1}^mF_j), \quad
F_j\setminus(\medcup_{i=1}^kE_i), \qquad(i\in\{1,\ldots,k\}, \
j\in\{1,\ldots,m\}).
\end{equation*}
In the following we consider additionally a family $\{\nu(E,\cdot)\mid
E\in\CE\}$ of probability measures from $\ID(\boxplus)$, satisfying
that
\begin{equation*}
\nu\big(\textstyle{\medcup_{j=1}^nE_j},\cdot\big)
=\nu(E_1,\cdot)\boxplus\cdots\boxplus\nu(E_n,\cdot),
\end{equation*}
whenever $E_1,\ldots,E_n$ are disjoint sets from $\CE$. For any set
$E$ from $\CE$ we denote by $\tau_{E}$ the state on the abelian von
Neumann algebra\footnote{$L^{\infty}(\nu(E,\cdot))$ is the vector
  space of all $\nu(E,\cdot)$-essentially bounded functions $f\colon
  X\to\C$ identified up to $\nu(E,\cdot)$-null sets.}
$L^{\infty}(\nu(E,\cdot))$ given by integration with
respect to the probability measure $\nu(E,\cdot)$. Subsequently for any element
$S=\{E_1,\ldots,E_k\}$ from $\II$, we let $(\CM_S,\tau_S)$ denote the
$W^*$-reduced free product of the $W^*$-probability spaces
$(L^{\infty}(\nu(E_j,\cdot)),\tau_{E_j})$, $j=1,\ldots,k$ (see
\cite{vdn} for details).


\begin{lemma}\label{delI}
For any element $S=\{E_1,\ldots,E_k\}$ of $\II$ there exist freely independent
operators $M_S(E_1),\ldots,M_S(E_k)$ from $\olCM_S$, which generate $\CM_S$
as a von Neumann algebra\footnote{Operators $T_1,\ldots,T_k$
  affiliated with a von Neumann algebra $\CM$ are said to generate
  $\CM$ as a von Neumann algebra, if $\CM$ is the smallest von Neumann
  algebra on the considered Hilbert space containing the family
  $\medcup_{j=1}^k\{f(T_j)\mid f\in\CBF_b(\R)\}$.}, and such that
\begin{equation*}
\spL\{M(E_j)\}=\nu(E_j,\cdot), \qquad(j=1,\ldots,k).
\end{equation*}
\end{lemma}

\begin{proof} For each $j$ in $\{1,\ldots,k\}$ we have a canonical embedding
$\iota_j\colon L^{\infty}(\nu(E_j,\cdot))\hookrightarrow\CM_S$, such that
$\tau_{E_j}=\tau_S\circ \iota_j$ (see \cite{vdn}). By
Proposition~\ref{ext_to_unbounded_ops} $\iota_j$ 
gives rise to a $*$-homomorphism $\oli_j\colon
\overline{L^{\infty}(\nu(E_j,\cdot))}\to\olCM_S$. We then define
\begin{equation}
M_S(E_j)=\oli_j(\id_{\R}), \qquad(j=1,\ldots,k),
\label{EFLB_eq4}
\end{equation}
where $\id_{\R}$ denotes the identity function on $\R$ considered as an
element of $\overline{L^{\infty}(\nu(E_j,\cdot))}$. By
Proposition~\ref{ext_to_unbounded_ops} the range of $\oli_j$ equals
the class of operators affiliated with the von Neumann algebra
$\iota_j(L^{\infty}(\nu(E_j,\cdot))$, so in particular $M_S(E_j)$ is
affiliated with that von Neumann algebra. By construction of
$(\CM_S,\tau_S)$ the algebras $\iota_j(L^{\infty}(\nu(E_j,\cdot))$,
$j=1,\ldots,n$, are free in $(\CM_S,\tau_S)$, so in particular
$M_S(E_1),\ldots,M_S(E_k)$ are freely independent with respect to
$\tau$. 
For any $f$ in $\CBF_b(\R)$ we note next (cf.\
Proposition~\ref{ext_to_unbounded_ops}) that 
\begin{equation*}
f(M_S(E_j))=f(\oli_j(\id_{\R}))=\oli_j(f(\id_{\R}))=\oli_j(f)=\iota_j(f),
\end{equation*}
and hence $M_S(E_j)$ generates $\iota_j(L^{\infty}(\nu(E_j,\cdot))$ as a
von Neumann algebra. This further implies that
$\{M_S(E_1),\ldots,M_S(E_n)\}$ generates $\CM_S$ as a von Neumann
algebra (cf.\ \cite[Definition~1.6.1]{vdn}). For any $j$ in
$\{1,\ldots,k\}$ and any function $f$ from $\CBF_b(\R)$ we note finally that
\begin{equation*}
\tau_S\big[f(M_S(E_j))\big]=\tau_S\big[\iota_j(f)\big]=\tau_{E_j}(f)
=\int_{\R}f(t)\,\nu(E_j,\d t),
\end{equation*}
verifying that $\spL\{M_S(E_j)\}=\nu(E_j,\cdot)$.
\end{proof}

\begin{lemma}\label{delII}
Assume that $S=\{E_1,\ldots,E_k\}$ and $T=\{F_1,\ldots,F_m\}$ are elements
of $\II$ such that $S\le T$. Then there exists a normal
$*$-homomorphism $\iota_{S,T}\colon\CM_S\to\CM_T$ such that
$\tau_S=\tau_T\circ\iota_{S,T}$.

Specifically it holds for any $i$ in $\{1,\ldots,k\}$ (with notation from
Lemma~\ref{delI}) that
\begin{equation}
\oli_{S,T}(M_S(E_i))=M_T(F_{j(i,1)})+\cdots+M_T(F_{j(i,l_i)}),
\label{EFLB_eq1}
\end{equation}
whenever $E_i=F_{j(i,1)}\cup\cdots\cup F_{j(i,l_i)}$ for suitable
$j(i,1),\ldots,j(i,l_i)$ from $\{1,\ldots,m\}$.
\end{lemma}

\begin{proof} We adopt the notation from Lemma~\ref{delI}. 
Given any $i$ in $\{1,2,\ldots,k\}$ we may, since $S\le T$, write $E_i$
(unambiguously) as $F_{j(i,1)}\cup\cdots\cup F_{j(i,l_i)}$ for suitable
  $j(i,1),\ldots,j(i,l_i)$ from $\{1,\ldots,m\}$. Since the operators
  $M_T(F_{j(i,1)}),\ldots,M_T(F_{j(i,l_i)})$ are freely independent, it
  follows then that
\begin{equation*}
\begin{split}
\spL\{M_T(F_{j(i,1)})+\cdots+M_T(F_{j(i,l_i)})\}
&=\nu(F_{j(i,1)},\cdot)\boxplus\cdots\boxplus\nu(F_{j(i,l_i)},\cdot)
\\[.2cm]
&=\nu\big(F_{j(i,1)}\cup\cdots\cup F_{j(i,l_i)},\cdot)
=\nu(E_i,\cdot)=\spL\{M_S(E_i)\}.
\end{split}
\end{equation*}
Note also that since all the operators $M_T(F_1),\ldots,M_T(F_m)$ are
freely independent, the sums
$M_T(F_{j(i,1)})+\cdots+M_T(F_{j(i,l_i)})$, $i=1,\ldots,k$, are also
freely independent. Indeed, for each $i$ the sum
$M_T(F_{j(i,1)})+\cdots+M_T(F_{j(i,l_i)})$ is affiliated with the von
Neumann algebra generated by 
$L^{\infty}(\nu(F_{j(i,1)},\cdot),\ldots,L^{\infty}(\nu(F_{j(i,l_i)},\cdot)$
considered as sub-algebras of $\CM_T$. And by
\cite[Proposition~2.5.5]{vdn} these von Neumann subalgebras
are free in $(\CM_T,\tau_T)$ for varying $i$. 
It follows that the two families of operators:
\begin{equation*}
\medcup_{i=1}^k\big\{f\big(M_T(F_{j(i,1)})+\cdots+M_T(F_{j(i,l_i)})\big)
\bigm| f\in\CBF_b(\R)\big\}
\end{equation*}
and
\begin{equation*}
\medcup_{i=1}^k\big\{f\big(M_S(E_i))\big)
\bigm| f\in\CBF_b(\R)\big\}
\end{equation*}
have the same $*$-distribution, and since $M_S(E_1),\ldots,M_S(E_k)$
generate $\CM_S$ as a von Neumann algebra, it follows thus
from Proposition~\ref{iso_lemma} that there exists a normal, injective
$*$-homomorphism
$\iota_{S,T}\colon\CM_S\to\CM_T$ such that
\begin{equation*}
\oli_{S,T}\big(f(M_S(E_i))\big)
=f\big(M_T(F_{j(i,1)})+\cdots+M_T(F_{j(i,l_i)})\big)
\end{equation*}
for any $i$ in $\{1,\ldots,k\}$ and $f$ in $\CBF_b(\R)$. In addition
$\tau_S=\tau_T\circ\iota_{S,T}$. 

To establish \eqref{EFLB_eq1} we consider for each $n$ in $\N$ the
function $f_n\colon\R\to\R$ defined by:
\begin{equation*}
f_n(t)=t1_{[-n,n]}(t)-n1_{(-\infty,-n)}(t)+n1_{(n,\infty)}(t),
\qquad(t\in\R).
\end{equation*}
Then $f_n(t)\to t$ as $n\to\infty$ for all $t$ in $\R$, and this
implies that $f_n(M_S(E_i))\overset{\rm P}{\to} M_S(E_i)$ and that
$f_n(M_T(F_{j(i,1)})+\cdots+M_T(F_{j(i,i_l)}))\overset{\rm P}{\to}
M_T(F_{j(i,1)})+\cdots+M_T(F_{j(i,i_l)})$ as $n\to\infty$
(cf.\ the calculation
\eqref{VNP_eq5} in the proof of
Proposition~\ref{ext_to_unbounded_ops}). From formula \eqref{VNP_eq4} in that
same proof it follows then further that
\begin{equation}
\begin{split}
\oli_{S,T}(M_S(E_i))
&=\text{P-}\lim_{n\to\infty}\iota_{S,T}\big(f_n(M_S(E_i))\big)
=\text{P-}\lim_{n\to\infty}f_n\big(M_T(F_{j(i,1)})+\cdots+M_T(F_{j(i,i_l)})\big)
\\[.2cm]
&=M_T(F_{j(i,1)})+\cdots+M_T(F_{j(i,i_l)}),
\end{split}
\end{equation}
as desired. This completes the proof.
\end{proof}

\begin{lemma}\label{delIII}
  The family $(\CM_S,\tau_S)_{S\in\II}$ of $W^*$-probability spaces
  equipped with the family $\{\iota_{S,T}\mid S,T\in\II,\ S\le T\}$ of
  $*$-homomorphisms described in Lemma~\ref{delII}
forms a directed system of $W^*$-algebras and
  injective, normal $*$-homomorphisms.
\end{lemma}

\begin{proof} Given $R,S,T$ in $\II$ such that $R\le S\le T$, we must show
that $\iota_{R,T}=\iota_{S,T}\circ\iota_{R,S}$. Writing
$R=\{D_1,\ldots,D_m\}$, $S=\{E_1,\ldots,E_k\}$ and $T=\{F_1,\ldots,F_l\}$
for suitable $D_h,E_i,F_j$ from $\CE\setminus\{\emptyset\}$, we know
that
\begin{equation*}
\begin{split}
D_h&=E_{i(h,1)}\cup\cdots\cup E_{i(h,k_h)}, \qquad(h=1,\ldots,m),
\\[.2cm]
E_i&=F_{j(i,1)}\cup\cdots\cup F_{j(i,l_i)}, \qquad(i=1,\ldots,k),
\end{split}
\end{equation*}
for suitable $i(h,1),\ldots,i(h,k_h)$ in $\{1,\ldots,k\}$ and
$j(i,1),\ldots,j(i,l_i)$ from $\{1,\ldots,l\}$. Then
\begin{equation*}
\begin{split}
\oli_{R,T}(M_R(D_h))&=
\sum_{r=1}^{l_{i(h,1)}}M_T(F_{j(i(h,1),r)})
+\cdots+
\sum_{r=1}^{l_{i(h,k_h)}}M_T(F_{j(i(h,k_h),r)})
\\[.2cm]
&=\oli_{S,T}(M_S(E_{i(h,1)}))+\cdots+
\oli_{S,T}(M_S(E_{i(h,k_h)}))
\\[.2cm]
&=\oli_{S,T}\big((M_S(E_{i(h,1)})+\cdots+M_S(E_{i(h,k_h)})\big)
\\[.2cm]
&=\oli_{S,T}\big(\oli_{R,S}(D_h)\big).
\end{split}
\end{equation*}
For any $f$ in $\CBF_b(\R)$ it follows then (cf.\
Proposition~\ref{ext_to_unbounded_ops}) that
\begin{equation*}
\begin{split}
\iota_{R,T}\big(f(M_R(D_h))\big)&=f\big(\oli_{R,T}(M_R(D_h))\big)
=f\big(\oli_{S,T}\circ\oli_{R,S}(M_R(D_h))\big)
\\[.2cm]
&=\iota_{S,T}\big(f(\oli_{R,S}(M_R(D_h)))\big)
=\iota_{S,T}\circ\iota_{R,S}\big(f(M_R(D_h))\big).
\end{split}
\end{equation*}
Since $\CM_R$ is generated as a von Neumann algebra by the family
\begin{equation*}
\medcup_{h=1}^m\{f(M_R(D_h))\mid f\in\CBF_b(\R)\},
\end{equation*}
and since $\iota_{R,T}$ and $\iota_{S,T}\circ\iota_{R,S}$ are both
normal, it follows by an application of Kaplansky's Density Theorem that
$\iota_{R,T}=\iota_{R,S}\circ\iota_{S,T}$, as desired.
\end{proof}

\begin{proofof}[Proof of Theorem~\ref{existence_FLB}.]

We note first that assertions \ref{existence_FLB(b)} and
\ref{existence_FLB(c)} are direct consequences of
\ref{existence_FLB(a)}. To prove \ref{existence_FLB(a)},
we consider the directed system (cf.\ Lemma~\ref{delIII})
\begin{equation*}
(\CM_S,\tau_S)_{S\in\II}, \quad \{\iota_{S,T}\mid S,T\in\II, \ S\le
T\}
\end{equation*}
of $W^*$-probability spaces and trace preserving
$*$-homomorphisms. Using Proposition~\ref{W*_inductive_limit}, there
exists a $W^*$-probability space $(\CM,\tau)$ and injective, normal
$*$-homomorphisms $\iota_S\colon\CM_S\to\CM$ ($S\in\II$), satisfying
that $\tau_S=\tau\circ\iota_S$ for all $S$ in $\II$, and that
$\iota_S=\iota_T\circ\iota_{S,T}$ for any $S,T$ in $\II$ such that
$S\le T$. 
We now define
\begin{equation}
M(\emptyset)=0, \qquad\text{and}\qquad
M(E)=\oli_{\{E\}}(M_{\{E\}}(E)) \quad\text{for $E$ in $\CE\setminus\{\emptyset\}$},
\label{EFLB_eq2}
\end{equation}
where $M_{\{E\}}(E)$ denotes the identity function $\id_{\R}$ on $\R$ considered as
an element of $\overline{L^{\infty}(\nu(E,\cdot))}=\olCM_{\{E\}}$. We will
show that the family $\{M(E)\mid E\in\CE\}$
satisfies the conditions \ref{def_FLB(b)}-\ref{def_FLB(c)} in
Definition~\ref{def_FLB} and that $\spL\{M(E)\}=\nu(E,\cdot)$ for all
$E$ in $\CE$.

\begin{lenumerate}[a]

\item
Assume that $E_1,\ldots,E_r$ are disjoint sets from
$\CE\setminus\{\emptyset\}$, and put $S=\{E_1,\ldots,E_r\}\in\II$.
Consider further arbitrary functions $f_1,\ldots,f_r$ from
$\CBF_b(\R)$. We must show that the bounded operators
$f_1(M(E_1)),\ldots,f_r(M(E_r))$ are freely independent with respect
to $\tau$. For any polynomial $p$ in $r$ non-commuting variables we
note (cf.\ Proposition~\ref{ext_to_unbounded_ops}) that
\begin{equation}
\begin{split}
\tau\big[p\big(f_1(M(E_1)),&\ldots,f_r(M(E_r))\big)\big]
\\[.2cm]
&=\tau\big[p\big(f_1(\oli_{\{E_1\}}(M_{\{E_1\}}(E_1))),\ldots,
f_r(\oli_{\{E_r\}}(M_{\{E_r\}}(E_r)))\big)\big]
\\[.2cm]
&=\tau\big[p\big(\iota_{\{E_1\}}(f_1(M_{\{E_1\}}(E_1))),\ldots,
\iota_{\{E_r\}}(f_r(M_{\{E_r\}}(E_r)))\big)\big]
\\[.2cm]
&=\tau\big[p\big(\iota_S\circ\iota_{\{E_1\},S}(f_1),\ldots,
\iota_S\circ\iota_{\{E_r\},S}(f_r)\big)\big]
\\[.2cm]
&=\tau\big[\iota_S\big(p\big(\iota_{\{E_1\},S}(f_1),\ldots,
\iota_{\{E_r\},S}(f_r)\big)\big)\big]
\\[.2cm]
&=\tau_S\big[p\big(\iota_{\{E_1\},S}(f_1),\ldots,
\iota_{\{E_r\},S}(f_r)\big)\big].
\end{split}
\label{EFLB_eq3}
\end{equation}
For each $j$ in $\{1,\ldots,r\}$ recall that $\iota_{\{E_j\},S}$ is
the canonical embedding of $\CM_{\{E_j\}}=L^{\infty}(\nu(E_j,\cdot))$
into the reduced free product $\CM_S=L^{\infty}(\nu(E_1,\cdot))*\cdots*
L^{\infty}(\nu(E_r,\cdot))$. Hence the ranges of
$\iota_{\{E_1\},S},\cdots,\iota_{\{E_r\},S}$ are free in
$(\CM_S,\tau_S)$, and in particular
$\iota_{\{E_1\},S}(f_1),\ldots,\iota_{\{E_r\},S}(f_r)$ are freely
independent with respect to $\tau_S$. Since \eqref{EFLB_eq3} holds for
any polynomial $p$ in $r$ non-commuting variables, it follows then that
$f_1(M(E_1)),\ldots,f_r(M(E_r))$ are freely independent with respect
to $\tau$.

\item
Let $E_1,\ldots,E_r$ be disjoint sets from
$\CE\setminus\{\emptyset\}$, and put $E=\medcup_{j=1}^rE_j$ and
$S=\{E_1,\ldots,E_r\}\in\II$. We must show that
$M(E)=M(E_1)+\cdots+M(E_r)$. Using Corollary~\ref{composition_lemma} we
find that
\begin{equation*}
\begin{split}
M(E_1)+&\cdots+M(E_r)
\\[.2cm]
&=\oli_{\{E_1\}}(\id_{\R})+\cdots+\oli_{\{E_r\}}(\id_{\R})
=\overline{\iota_S\circ\iota_{\{E_1\},S}}(\id_{\R})+\cdots+
\overline{\iota_S\circ\iota_{\{E_r\},S}}(\id_{\R})
\\[.2cm]
&=\oli_S\circ\oli_{\{E_1\},S}(\id_{\R})+\cdots+
\oli_S\circ\oli_{\{E_r\},S}(\id_{\R})
=\oli_S\big(M_S(E_1)+\cdots+M_S(E_r)\big)
\\[.2cm]
&=\oli_S\big(\oli_{\{E\},S}(M_{\{E\}}(E))\big)
=\overline{\iota_S\circ\iota_{\{E\},S}}(\id_{\R})
\\[.2cm]
&=\oli_{\{E\}}(\id_{\R})
=M(E),
\end{split}
\end{equation*}
where in the first, fourth, fifth and last equality we applied 
\eqref{EFLB_eq2}, \eqref{EFLB_eq1}, \eqref{EFLB_eq1} and
\eqref{EFLB_eq2}, respectively. 
\end{lenumerate}
To show finally that $\spL\{M(E)\}=\nu(E,\cdot)$ for all
$E$ in $\CE$, we assume without loss of generality that
$E\ne\emptyset$. Then since
$\tau_{\{E\}}=\tau\circ\iota_{\{E\}}$ we find for any function $f$ in
$\CBF_b(\R)$ that
\begin{equation*}
 \tau\big[f(M(E))\big]
=\tau\big[f\big(\oli_{\{E\}}(M_{\{E\}}(E))\big)\big]
=\tau\big[\iota_{\{E\}}\big(f(M_{\{E\}}(E))\big)\big]
=\tau_{\{E\}}(f)=\int_{\R}f\6\nu(E,\cdot),
\end{equation*}
which proves the desired identity.
This completes the proof.
\end{proofof}

\appendix

\section{Von Neumann algebra preliminaries}

To accommodate potential readers with limited background in the theory
of operator algebras, we start by recalling briefly various basic
concepts from that theory. For a thorough introduction to operator
algebras we refer to the classical text \cite{kr}. First
of all an algebra (over $\C$) is a vector space
$\CA$ over $\C$, which is also furnished with an associative
multiplication satisfying
the usual distributive laws in relation to the linear operations. One
may think of the matrix algebra $M_n(\C)$ as a concrete example. As in
this particular case the multiplication is generally not assumed to be
commutative. We say that $\CA$ is a $*$-algebra, if it is additionally
equipped with an involution (or $*$-operation) $a\mapsto
a^*\colon\CA\to\CA$, satisfying that $(a+b)^*=a^*+b^*$,
$(za)^*=\overline{z}a^*$, $(ab)^*=b^*a^*$ and $(a^*)^*=a$ for all
$a,b$ in $\CA$ and $z$ in $\C$. 

A $C^*$-algebra is a $*$-algebra $\CA$,
which is also a Banach space with respect to a norm $\|\cdot\|$,
satisfying additionally that $\|ab\|\le\|a\|\|b\|$ and $\|a^*a\|=\|a\|^2$ for all
$a,b$ in $\CA$. Again,
$M_n(\C)$ provides a (finite dimensional) example of a $C^*$-algebra,
and more generally the space $\CB(\CH)$ of continuous linear mappings
$T\colon\CH\to\CH$ on a Hilbert space $\CH$ is a canonical example of
a $C^*$-algebra. In fact any $C^*$-algebra may be identified
  with a norm closed, $*$-invariant subalgebra of $\CB(\CH)$. 
As in \cite{kr} we shall generally assume that a $C^*$-algebra $\CA$ comes equipped with a 
multiplicative neutral element $\unit_{\CA}$. If $\CA$ and $\CB$ are 
two $C^*$-algebras, a linear mapping $\phi\colon\CA\to\CB$ is called a
$*$-homomorphism, if $\phi(\unit_{\CA})=\unit_{\CB}$,
$\phi(ab)=\phi(a)\phi(b)$, and $\phi(a^*)=\phi(a)^*$ for all $a,b$ in $\CA$.

 A von Neumann algebra acting on a Hilbert space $\CH$ is a
 $*$-invariant subalgebra $\CM$ of $\CB(\CH)$, which is closed in the
weak operator topology, i.e.\ the weak
topology on $\CB(\CH)$
induced by the family $\{\omega_{\xi,\eta}\mid
 \xi,\eta\in\CH\}$ of linear functionals given by
\[
\omega_{\xi,\eta}(a)=\<a\xi,\eta\>, \qquad(a\in\CB(\CH)).
\]
As this topology is weaker than that induced by the $C^*$-norm
on $\CB(\CH)$, a von Neumann algebra is automatically a
 $C^*$-algebra. If $\CM$ and $\CN$ are two von Neumann-algebras 
(possibly acting on different Hilbert spaces) and
 $\phi\colon\CM\to\CN$ is a $*$-homomorphism, then $\phi$ is said to be
 normal if its restriction to the unit ball of $\CM$ is continuous
 with respect to the weak operator topologies on $\CM$ and $\CN$.
 
With the above basic concepts in place, we recall next that a
$W^*$-probability space is a pair $(\CM,\tau)$, where
$\CM$ is a von Neumann algebra (acting on some Hilbert space), and
$\tau$ is a faithful, normal, tracial state on $\CM$. Specifically
$\tau$ is a linear mapping from $\CM$ into $\C$, which is continuous
on the unit ball of $\CM$ with respect to the weak operator topology 
and satisfies the following conditions: $\tau(a^*a)>0$ for all $a$ in
$\CM\setminus\{0\}$, $\tau(ab)=\tau(ba)$ for all $a,b$ in $\CM$, and 
$\tau(\unit_{\CM})=1$.

If $I$ is an arbitrary non-empty index set, and $(x_i)_{i\in I}$ is a
corresponding family of operators in a $W^*$-probability space
$(\CM,\tau)$, then the $*$-distribution of $(x_i)_{i\in I}$ is the
collection of all complex numbers in the form
\begin{equation*}
\tau\big(z_1^{p_1}z_2^{p_2}\cdots z_n^{p_n}\big),
\end{equation*}
where $n\in\N$, $p_1,\ldots,p_n\in\N$ and 
$z_1,\ldots,z_n\in\{x_i\mid i\in I\}\cup\{x^*_i\mid i\in I\}$. 

\begin{proposition}\label{iso_lemma}
Let $(\CM,\tau)$ and $(\CN,\psi)$ be $W^*$-probability spaces, let $I$
be a non-empty index set, and
assume that $(x_i)_{i\in I}$ and $(y_i)_{i\in I}$ are families of
operators from $\CM$ and $\CN$, respectively. Let $\CM_0$ denote the
von Neumann subalgebra of $\CM$ generated by $(x_i)_{i\in I}$, and let
$\CN_0$ denote the von Neumann subalgebra of $\CN$ generated by
$(y_i)_{i\in I}$.

If the $*$-distribution of $(x_i)_{i\in I}$ (with respect to $\tau$) equals
that of $(y_i)_{i\in I}$ (with respect to $\psi$), then there exists a
normal $*$-isomorphism $\Phi$ of $\CM_0$ onto $\CN_0$, such that 
$\tau=\psi\circ\Phi$ on $\CM_0$, and such that $\Phi(x_i)=y_i$ for all
$i$ in $I$.
\end{proposition}

For the proof of Proposition~\ref{iso_lemma} we refer to
\cite[Theorem~2 in Section 6.5]{MiSp} or
\cite[Remark~1.8]{vo2}.

\begin{corollary}\label{automatic_normality}
Let $(\CM,\tau)$ and $(\CN,\psi)$ be $W^*$-probability spaces, and let
$\Phi\colon\CM\to\CN$ be a $*$-homomorphism such that
$\tau=\psi\circ\Phi$. Then $\Phi$ is automatically normal and
injective, and $\Phi(\CM)$ is a von Neumann subalgebra of $\CN$.
\end{corollary}

\begin{proof} Since $\tau=\psi\circ\Phi$, the two families of operators $\CM$
and $\Phi(\CM)$ (indexed by $\CM$) have the same $*$-distribution.
By application of Proposition~\ref{iso_lemma}, we obtain thus a normal
$*$-isomorphism $\tilde{\Phi}$ of the von Neumann subalgebra generated by
$\CM$ (i.e.\ $\CM$ itself) onto the von Neumann subalgebra of $\CN$ generated by
$\Phi(\CM)$, such that $\tilde{\Phi}(a)=\Phi(a)$ for all $a$ in
$\CM$. Obviously then $\Phi=\tilde{\Phi}$, so $\Phi$ is normal and
injective. In addition $\Phi(\CM)=\tilde{\Phi}(\CM)$, which is a von
Neumann algebra.
\end{proof} 

Recall that a partial order ``$\le$'' on a set $\CS$ is called \emph{upward
filtering}, if, for any elements $S,T$ in $\CS$, there exists an
element $U$ in $\CS$ such that $S\le U$ and $T\le U$.

The following result essentially amounts to the existence of inductive
limits in the category of $W^*$-probability spaces.

\begin{proposition}\label{W*_inductive_limit}
Consider a set $\CS$ equipped with an upward filtering partial order
``$\le$''. Consider additionally a corresponding family
$(\CM_S,\tau_S)_{S\in\CS}$ of $W^*$-probability spaces, and assume
that whenever $S,T\in\CS$, such that $S\le T$, there is a
$*$-homomorphism $\Phi_{S,T}\colon\CM_S\to\CM_T$ such that
$\tau_S=\tau_T\circ\Phi_{S,T}$.

Then there exists a $W^*$-probability space $(\CM,\tau)$ and injective
normal $*$-homomorphisms $\Phi_S\colon\CM_S\to\CM$ ($S\in\CS$), such
that
\begin{equation*}
\tau\circ\Phi_S=\tau_S \quad\text{for all $S$ in $\CS$,}
\end{equation*}
and
\begin{equation*}
\Phi_T\circ\Phi_{S,T}=\Phi_S \quad\text{for all $S,T$ in $\CS$, such
  that $S\le T$.}
\end{equation*}
In addition $\CM$ is generated as a von Neumann algebra by the
$*$-subalgebra $\medcup_{S\in\CS}\Phi_S(\CM_S)$.

The properties listed above characterize $(\CM,\tau)$ up to trace
preserving $*$-isomorphisms.
\end{proposition}

\begin{proof} From Corollary~\ref{automatic_normality} we know that $\Phi_{S,T}$
is normal and injective for any $S,T$ in $\CS$ such that $S\le T$, and
moreover $\Phi_{S,T}(\CM_S)$ is a von Neumann subalgebra of $\CM_T$.
Now let $\CM^0$ be the $C^*$-algebra inductive limit of the directed
system
\begin{equation*}
\{\CM_S\mid S\in\CS\}, \quad \{\Phi_{S,T}\mid S,T\in\CS, \ S\le T\}
\end{equation*}
(see \cite[Proposition~11.4.1]{kr} for details). Then for any $S$ in $\CS$ there is a
$*$-monomorphism $\Phi_S^0\colon\CM_S\to\CM^0$, such that
$\Phi_{T}^0\circ\Phi_{S,T}=\Phi_S^0$, whenever $S,T\in\CS$ such that
$S\le T$. Putting $\CM^{00}=\medcup_{S\in\CS}\Phi_S^0(\CM_S)$, we may
then define a linear functional $\tau^{00}\colon\CM^{00}\to\C$ such
that
\begin{equation}
\tau_S=\tau^{00}\circ\Phi_S^0 \quad\text{for all $S$ in $\CS$.}
\label{VNP_eq2}
\end{equation}
Indeed, if $a\in\Phi_S^0(\CM_S)\cap\Phi_T^0(\CM_T)$ for $S,T$ in $\CS$,
we have that $a=\Phi_S^0(a')=\Phi_T^0(a'')$ for
suitable $a'$ in $S$ and $a''$ in $T$, and we must show that
$\tau_S(a')=\tau_T(a'')$. Since ``$\le$'' is upward filtering, we may
choose an element $U$ of $\CS$, such that $S,T\le U$. 
Now 
\begin{equation*}
\Phi_U^0\circ\Phi_{S,U}(a')=\Phi_{S}^0(a')=\Phi_T^0(a'')
=\Phi_{U}^0\circ\Phi_{T,U}(a''),
\end{equation*}
so the injectivity of $\Phi_U^0$ implies that
$\Phi_{S,U}(a')=\Phi_{T,U}(a'')$, and therefore  
\begin{equation*}
\tau_S(a')=\tau_U\circ\Phi_{S,U}(a')=\tau_U\circ\Phi_{T,U}(a'')=
\tau_T(a''),
\end{equation*}
as desired. Thus \eqref{VNP_eq2} gives rise to a well-defined mapping
$\tau^{00}\colon\CM^{00}\to\C$, and by similar reasoning it follows
that $\tau^{00}$ is a linear, positive, tracial and norm-decreasing
functional on $\CM^{00}$. Since $\CM^{00}$ is dense in $\CM^0$ with
respect to the operator norm (cf.\ \cite[Proposition~11.4.1]{kr}), $\tau^{00}$
thus extends to a linear, tracial, norm-decreasing functional
$\tau^0\colon\CM^0\to\C$, and since
$\tau^0(\unit_{\CM^0})=1=\|\tau^0\|$, $\tau^0$ is a state on $\CM^0$. 

Consider next the GNS-representation
$\pi_{\tau^0}\colon\CM^0\to\CB(\CH_{\tau^0})$ of $\CM^0$ associated
with $\tau^0$ (see \cite[Theorem~4.5.2]{kr}),
and let $\xi_0$ denote the unit $\unit_{\CM^0}$ of
$\CM^0$ considered as an element of $\CH_{\tau_0}$.
Let $\CM$ denote
the closure of $\pi_{\tau^0}(\CM^0)$ in the strong operator topology,
and define $\tau\colon\CM\to\C$ by
\begin{equation*}
\tau(a)=\<a\xi_0,\xi_0\>, \qquad(a\in\CM).
\end{equation*}
Then $\tau^0=\tau\circ\pi_{\tau^0}$, so $\tau$ is tracial on
$\pi_{\tau^0}(\CM^0)$. Since multiplication is separately continuous
in each variable in the strong operator topology, it follows by a
``bootstrap'' argument that $\tau$ is tracial on all of $\CM$. Hence
$\xi_0$ is a generating trace vector for $\CM$ and hence also for the
commutant $\CM'$ (see \cite[Lemma~7.2.14]{kr}). This implies that $\xi_0$ is
separating for $\CM$ (see \cite[Corollary~5.5.12]{kr}), and hence $\tau$ is
faithful, so that $(\CM,\tau)$ is indeed a $W^*$-probability space.

For any $S$ in $\CS$ we define next $\Phi_S\colon\CM_S\to\CM$ by
$\Phi_S=\pi_{\tau^0}\circ\Phi_S^0$, and we note for $a$ in $\CM_S$ that
\begin{equation*}
\tau\circ\Phi_S(a)=(\tau\circ\pi_{\tau^0})\circ\Phi_S^0(a)
=\tau^0\circ\Phi_S^0(a)=\tau^{00}\circ\Phi_S^0(a)=\tau_S(a).
\end{equation*}
Hence Corollary~\ref{automatic_normality} implies that $\Phi_S$ is
injective and normal. If $S,T\in\CS$, such that $S\le T$, and
$a\in\CM_S$, we note furthermore that
\begin{equation*}
\Phi_T\circ\Phi_{S,T}(a)=\pi_{\tau^0}\circ\Phi_T^0\circ\Phi_{S,T}
=\pi_{\tau^0}\circ\Phi_S^0(a)=\Phi_S(a).
\end{equation*}
To see that $\CM$ is generated as a von Neumann algebra by
$\medcup_{S\in\CS}\Phi_S(\CM_S)$ we use again that
$\CM^0=(\medcup_{S\in\CS}\Phi_S^0(\CM_S))^{=}$ (where $\CC^=$ denotes
the norm closure of $\CC$).
Since $\pi_{\tau^0}$ is norm-continuous and $\pi_{\tau^0}(\CM^0)$ is a
$C^*$-algebra, this implies that
\begin{equation*}
\pi_{\tau^0}(\CM^0)
=\big(\medcup_{S\in\CS}\pi_{\tau^0}\circ\Phi_S^0(\CM_S)\big)^=
=\big(\medcup_{S\in\CS}\Phi_S(\CM_S)\big)^=.
\end{equation*}
Since the norm topology is stronger than the strong operator topology,
this further entails that
\begin{equation*}
\CM=\big(\pi_{\tau^0}(\CM^0)\big)^{-s}
=\Big(\big(\medcup_{S\in\CS}\Phi_S(\CM_S)\big)^=\Big)^{-s} 
=\big(\medcup_{S\in\CS}\Phi_S(\CM_S)\big)^{-s}
\end{equation*}
(where $\CC^{-s}$ denotes the closure of $\CC$ in the strong operator
topology) as desired. 

We establish finally the uniqueness statement.
If $(\CM',\tau')$ is another $W^*$-probability space
satisfying the conditions listed for $(\CM,\tau)$, we consider the 
injective $*$-monomorphisms $\Phi_S'\colon\CM_S\to\CM'$ corresponding
to the $\Phi_S$'s. It follows then that the two families of operators
$\medcup_{S\in\CS}\Phi_S(\CM_S)$ and $\medcup_{S\in\CS}\Phi_S'(\CM_S)$
(indexed by $\medcup_{S\in\CS}\CM_S$) have the same $*$-distribution.
Hence Proposition~\ref{iso_lemma} yields an injective, normal
$*$-isomorphism $\Psi$ from $\CM=(\medcup_{S\in\CS}\Phi_S(\CM_S))^{-s}$ onto
$(\medcup_{S\in\CS}\Phi_S'(\CM_S))^{-s}=\CM'$, such that $\tau=\tau'\circ\Psi$.
\end{proof}

Before stating the next proposition we recall that the symbol
``$\overset{\mathrm{P}}{\to}$'' refers to convergence in the measure
topology. 

\begin{proposition}\label{ext_to_unbounded_ops}
Let $(\CM,\tau)$ and $(\CN,\psi)$ be $W^*$-probability spaces, and let
$\Phi\colon\CM\to\CN$ be a $*$-homomorphism such that
$\tau=\psi\circ\Phi$. Let further $\overline{\CM}$ and
$\overline{\CN}$ denote the set of (closed densely defined) operators
affiliated with $\CM$ and $\CN$, respectively. We then have

\begin{enumerate}[i]

\item
$\Phi$ extends to an injective mapping
$\overline{\Phi}\colon\overline{\CM}\to\overline{\CN}$ which preserves
the operations of scalar multiplication, strong sum, strong
multiplication and the $*$-operation. In addition
$\overline{\Phi}(\overline{\CM})=\overline{\Phi(\CM)}$, and
$(\overline{\Phi})^{-1}=\overline{\Phi^{-1}}$.

\item
If $(a_l)_{l\in\N}$ is a sequence of operators from $\CM$,
$a\in\overline{\CM}$ and $a_l\overset{\rm P}{\to} a$, then also
$\Phi(a_l)\overset{\rm P}{\to}\overline{\Phi}(a)$.

\item
$\olPhi$ preserves spectral calculus in the sense that
\begin{equation*}
\olPhi(f(a))=f(\olPhi(a))
\end{equation*}
for any selfadjoint operator $a$ in $\olCM$ and any function $f$ from
$\CBF(\R)$. 

\end{enumerate}
\end{proposition}

\begin{proof} Since $\tau$ is a finite trace it follows from \cite[Example~1,
page~22]{te} that $\olCM$ equals the class of
$\tau$-measurable operators affiliated with $\CM$. Hence (see
\cite[Theorem~28]{te}) $\olCM$ is a complete Hausdorff topological
$*$-algebra with respect to the measure topology. In addition $\CM$
is dense in $\olCM$ with respect to the measure topology,
and this topology is first countable. Of course similar statements hold 
for $\olCN$ in relation to $(\CN,\psi)$.

Now given $a$ in $\olCM$ we may choose a sequence $(a_n)_{n\in\N}$ from $\CM$
such that $a_n\overset{\rm P}{\to}a$. Since $\Phi$ is normal it
follows then for any positive $\epsilon$ that
\begin{equation}
\begin{split}
\psi\big[1_{[\epsilon,\infty)}\big(|\Phi(a_n)-\Phi(a_m)|\big)\big]
&=\psi\big[1_{[\epsilon,\infty)}\big(|\Phi(a_n-a_m)|\big)\big]
=\psi\big[1_{[\epsilon,\infty)}\big(\Phi(|a_n-a_m|)\big)\big]
\\[.2cm]
&=\psi\big[\Phi\big(1_{[\epsilon,\infty)}(|a_n-a_m|)\big)\big]
=\tau\big[1_{[\epsilon,\infty)}(|a_n-a_m|)\big]
\longrightarrow0,
\end{split}
\label{VNP_eq3}
\end{equation}
as $n,m\to\infty$. Hence $(\Phi(a_n))_{n\in\N}$ is a Cauchy sequence in $\CN$
($\subseteq\olCN$) with respect to the measure topology, so there
exists an element $b$ in $\olCN$ such that $\Phi(a_n)\overset{\rm P}{\to} b$ as
$n\to\infty$. Considering another sequence $(a_n')_{n\in\N}$ from $\CM$ such
that $a_n'\overset{\rm P}{\to} a$, we may further consider the mixed
sequence $a_1,a_1',a_2,a_2',\ldots$ which also converges to $a$ in
probability. Hence the argument above shows that the sequences
$(\Phi(a_n'))_{n\in\N}$ and $\Phi(a_1),\Phi(a_1'),\Phi(a_2),\Phi(a_2'),\ldots$
converge in probability to elements $b'$ respectively $b''$ from
$\olCN$. By subsequence considerations we must have that $b=b''=b'$,
and hence we may define a mapping $\olPhi\colon\olCM\to\olCN$ by
setting
\begin{equation}
\olPhi(a)=\text{P-}\lim_{n\to\infty}\Phi(a_n), \qquad(a\in\olCM),
\label{VNP_eq4}
\end{equation}
where $(a_n)_{n\in\N}$ is any sequence from $\CM$ such that $a_n\overset{\rm
  P}{\to}a$ as $n\to\infty$. 

From this definition and the fact that scalar multiplication, strong
sum, strong multiplication and the $*$-operation are all continuous
operations in the measure topology, it follows by standard arguments
that
\begin{equation*}
\olPhi(\lambda a)=\lambda\Phi(a), \quad
\olPhi(a+b)=\olPhi(a)+\olPhi(b), \quad 
\olPhi(ab)=\olPhi(a)\olPhi(b), \quad
\olPhi(a^*)=\olPhi(a)^*
\end{equation*}
for any $a,b$ from $\olCM$ and $\lambda$ in $\C$. In other words
$\olPhi$ is a $*$-homomorphism.

Recalling from Lemma~\ref{automatic_normality} that $\Phi(\CM)$ is a
von Neumann subalgebra of $\CN$,
we check next that $\olPhi(\olCM)=\overline{\Phi(\CM)}$. Since
$\overline{\Phi(\CM)}$ is the closure of $\Phi(\CM)$ in the measure
topology, the definition \eqref{VNP_eq4} clearly implies that
$\olPhi(\olCM)\subseteq\overline{\Phi(\CM)}$. Conversely, given $b$ in
$\overline{\Phi(\CM)}$, we may choose a sequence $(a_n)_{n\in\N}$ from $\CM$
such that $\Phi(a_n)\overset{\rm P}{\to} b$. The calculation
\eqref{VNP_eq3} then shows that $(a_n)_{n\in\N}$ is a Cauchy sequence
and hence convergent in the measure topology  to some $a$ from
$\olCM$. Now \eqref{VNP_eq4} implies that
$b=\overline{\Phi}(a)\in\olPhi(\olCM)$. The mapping
$\Phi^{-1}\colon\Phi(\CM)\to\CM$ (cf.\ Corollary~\ref{automatic_normality})
similarly gives rise to a mapping
$\overline{\Phi^{-1}}\colon\overline{\Phi(\CM)}\to\olCM$, and it
  follows easily from \eqref{VNP_eq4} (and the corresponding
  definition of $\overline{\Phi^{-1}}$) that
  $\overline{\Phi^{-1}}\circ\olPhi(a)=a$ for all $a$ in $\olCM$ and
  that $\olPhi\circ\overline{\Phi^{-1}}(b)=b$ for all $b$ in
  $\overline{\Phi(\CM)}$. In particular $\overline{\Phi}$ is injective.

Consider finally a selfadjoint element $a$ from $\olCM$, put
$b=\olPhi(a)\in\overline{\CN}$ and note that $b=b^*$. 
We then define a mapping $\Psi\colon\CBF(\R)\to\overline{\CN}$ by setting
\begin{equation*}
\Psi(f)=\olPhi(f(a)), \qquad(f\in\CBF(\R)).
\end{equation*}
We show next that
$\Psi(f)\in\overline{\Phi(W^*(\{a\}))}$
for all $f$ in $\CBF(\R)$,  where $W^*(\{a\})$ denotes the (abelian)
von Neumann sub-algebra of $\CM$  generated by $a$. For this note
first that $\Phi(W^*(\{a\}))$ is again a von Neumann algebra (cf.\
Lemma~\ref{automatic_normality}). Given $f$ in $\CBF(\R)$ we put
$f_n=f1_{\{|f|\le n\}}$ and note that $f_n(a)\in W^*(\{a\})$ for all
$n$. Using \cite[Corollary~5.6.29]{kr} we find then that
\begin{equation}
\begin{split}
\tau\big[1_{[\epsilon,\infty)}(|f_n(a)-f(a)|)\big]
&=\tau\big[\big(1_{[\epsilon,\infty)}\circ|f_n-f|\big)(a)\big]
=\int_{\R}1_{\{|f_n-f|\ge\epsilon\}}(t)\, \spL\{a\}(\d t)
\\[.2cm]
&=\spL(\{a\})(\{|f_n-f|\ge\epsilon\})
\xrightarrow[n\to\infty]{}0,
\label{VNP_eq5}
\end{split}
\end{equation}
where we used that $f_n\to f$ point-wise and that $\spL\{a\}$ is a
finite measure. Thus $f_n(a)\overset{\rm P}{\to}f(a)$ and hence also
$\Phi(f_n(a))\overset{\rm P}{\to}\olPhi(f(a))$. Since
$\overline{\Phi(W^*(\{a\}))}$ is complete in the measure topology, it
follows that $\overline{\Phi}(f(a))\in\overline{\Phi(W^*(\{a\}))}$ as desired,
and in particular we have that $b=\overline{\Phi}(a)\in\overline{\Phi(W^*(\{a\}))}$.

Note next that $\Psi$ is a $*$-homomorphism (since both $\olPhi$ and the mapping
$f\mapsto f(a)$ are $*$-homomorphisms), and furthermore $\Psi$ is
$\sigma$-normal in the sense of $\cite{kr}$, since the mapping
$f\mapsto f(a)$ is $\sigma$-normal (cf.\ \cite[Theorem~5.6.26]{kr}), and since
$\olPhi$ preserves least upper bounds (because $\olPhi$ and
$\olPhi^{-1}$ both preserve positivity). 

The observations above allow us to apply
\cite[Theorem~5.6.27]{kr} by which we infer that $\Psi$ is the
spectral mapping associated to $b$, i.e.\ $f(\olPhi(a))=\olPhi(f(a))$ for all
$f$ in $\CBF(\R)$. This completes the proof.
\end{proof}

\begin{corollary}\label{composition_lemma}
Let $(\CM,\tau)$, $(\CN,\psi)$ and $(\CL,\varpi)$ be $W^*$-probability
spaces, and let $\Phi\colon\CM\to\CN$ and $\Gamma\colon\CN\to\CL$ be
$*$-homomorphisms such that $\tau=\psi\circ\Phi$, and
$\psi=\varpi\circ\Gamma$. 

Then $\Phi$ and $\Gamma$ are both normal and injective, and for any
$a$ in $\olCM$ we have that
$\overline{\Gamma\circ\Phi}(a)=\overline{\Gamma}\circ\olPhi(a)$.
\end{corollary}

\begin{proof} Lemma~\ref{automatic_normality} ensures that $\Phi$, $\Gamma$
and $\Gamma\circ\Phi$ are normal and injective, and
Proposition~\ref{ext_to_unbounded_ops} ensures that the
mappings $\olPhi$, $\overline{\Gamma}$ and
$\overline{\Gamma\circ\Phi}$ are all well-defined.

Given $a$ in $\olCM$ we choose a sequence $(a_n)_{n\in\N}$ from $\CM$
converging to $a$ in the measure topology.
Proposition~\ref{ext_to_unbounded_ops} then entails that
$\Phi(a_n)\xrightarrow[]{\rm P}\olPhi(a)$, and hence that
$\Gamma\circ\Phi(a_n)\xrightarrow[]{\rm P}\overline{\Gamma}(\olPhi(a))$. In addition
$\Gamma\circ\Phi(a_n)\xrightarrow[]{\rm P}\overline{\Gamma\circ\Phi}(a)$,
and since the measure topology is a Hausdorff topology, we obtain the
desired conclusion.
\end{proof}

\begin{small}

\end{small}

\bigskip
\begin{minipage}[c]{0.3\textwidth}
\begin{small}
Dipartimento di Matematica\\
Univ.~degli Studi di Padova\\                                                                          
via Trieste 63\\
35121 Padova\\ 
Italy\\
{\tt fcollet@math.unipd.it}
\end{small}
\end{minipage}
\hfill
\begin{minipage}[c]{0.3\textwidth}
\begin{small}
School of Math.~Sciences\\
University of Nottingham\\
University Park\\
Nottingham NG7 2RD\\
United Kingdom\\
{\tt fabrizio.leisen@gmail.com}
\end{small}
\end{minipage}
\hfill
\begin{minipage}[c]{0.3\textwidth}
\begin{small}
Dept.~of Mathematics\\
University of Aarhus\\
Ny Munkegade 118\\
8000 Aarhus C\\
Denmark\\
{\tt steenth@imf.au.dk}
\end{small}
\end{minipage}

\end{document}